\documentclass[draft,10pt]{article}%
\usepackage{amscd}
\usepackage{amsmath}
\usepackage{amsfonts}
\usepackage{amssymb}
\usepackage{graphicx}
\usepackage[T1]{fontenc}
\usepackage[margin=0.5in]{geometry}%
\providecommand{\U}[1]{\protect\rule{.1in}{.1in}}
\providecommand{\U}[1]{\protect\rule{.1in}{.1in}}
\providecommand{\U}[1]{\protect\rule{.1in}{.1in}}

\numberwithin{equation}{section}
\newtheorem{theorem}{Theorem}[section]

\newtheorem{lemma}[theorem]{Lemma}

\newtheorem{proposition}[theorem]{Proposition}
\newtheorem{condition}[theorem]{Condition}

\newtheorem{remark}[theorem]{Remark}

\newenvironment{proof}[1][Proof]{\noindent\textbf{#1.} }{\ \rule{0.5em}{0.5em}}

\def\p2{\mathcal A_{\Phi,2\pi}(B)}
\def\0p2{\mathcal A_{\Phi,2\pi}(0)}
\def\sp2{\mathcal A_{\Phi,2\pi,\hbox{\rm SR}}(B)}
\def\beq{\begin{equation}}
\def\ene{\end{equation}}

\def\qed{\ifhmode\unskip\nobreak\fi\ifmmode\ifinner
\else\hskip5pt\fi\fi\hbox{\hskip5pt\vrule width4pt height6pt
depth1.5pt\hskip1pt}}

\def\+out{x^{\rm out}}
\begin{document}

\title{Nonlinear Schr\"{o}dinger equations with exceptional potentials}
\author{Ivan Naumkin\thanks{Electronic Mail: ivannaumkinkaikin@gmail.com} \thanks{We
would like to thank the project ERC-2014-CdG 646.650 SingWave for its
financial support, and the Laboratoire J.A. Dieudonn\'e of the Universit\'e de
Nice Sophia-Antipolis for its kind hospitality.}\\Laboratoire J.A. Dieudonn\'{e}, UMR CNRS 7351, \\Universit\'{e} de Nice Sophia-Antipolis, \\Parc Valrose, 06108 Nice Cedex 02, France}
\date{}
\maketitle

\begin{abstract}
We consider the cubic nonlinear Schr\"{o}dinger equation with an exceptional
potential. We obtain a sharp time decay for the global in time solution and we
get the large time asymptotic profile of small solutions. We prove the
existence of modified scattering for this model, that is, linear scattering
modulated by a phase. Our approach is based on the spectral theorem for the
perturbed linear Schr\"{o}dinger operator and a factorization technique, that
allows us to control the resonant nonlinear term. We make some parity
assumptions in order to control the small-energy behavior of the scattering
coefficients and of the wave functions.

\end{abstract}


\section{\label{S1}Introduction}

In this article, we consider the cubic nonlinear Schr\"{o}dinger equation
\begin{equation}
\left\{
\begin{array}
[c]{c}%
i\partial_{t}u=-\frac{1}{2}\partial_{x}^{2}u+Vu+\lambda\left\vert u\right\vert
^{2}u,\text{ }x\in\mathbb{R},\text{ }t>0,\\
u\left(  0,x\right)  =u_{0}\left(  x\right)  \text{, }x\in\mathbb{R},
\end{array}
\right.  \label{1.1}%
\end{equation}
where the potential $V\left(  x\right)  $ is a real-valued given function and
the parameter $\lambda\in\mathbb{R}$. This equation is related to the
Ginzburg-Landau equation of superconductivity (\cite{deGennes}), it was used
to describe one-dimensional self-modulation of a monochromatic wave
(\cite{Taniuti}, \cite{Yajima}), stationary two-dimensional self-focusing of a
plane wave (\cite{Bespalov}), propagation of a heat pulse in a solid Langmuir
waves in plasmas (\cite{Shimizu}) and the self-trapping phenomena of nonlinear
optics (\cite{Karpman}). For other applications of (\ref{1.1}), we refer to
\cite{Berge}, \cite{Scott}, \cite{Eboli}, \cite{Combes}, \cite{Benney}, and
\cite{Sulem}.

In the case when the external potential $V=0,$ it is well known that the
Cauchy problem
\begin{equation}
\left\{
\begin{array}
[c]{c}%
i\partial_{t}u=-\frac{1}{2}\partial_{x}^{2}u+\lambda\left\vert u\right\vert
^{2}u,\text{ }x\in\mathbb{R},\text{ }t>0,\\
u\left(  0,x\right)  =u_{0}\left(  x\right)  \text{, }x\in\mathbb{R},
\end{array}
\right.  \label{NLS0}%
\end{equation}
is globally well posed in a variety of spaces, for instance in $\mathbf{H}%
^{1}$ or $\mathbf{L}^{2}.$ See e.g. \cite{Kato}. Concerning the asymptotic
behavior of the solutions the cubic nonlinearity in dimension $1$ is a
limiting case. Indeed, if the nonlinearity $\lambda\left\vert u\right\vert
^{2}u$ in (\ref{NLS0}) is replaced by $\lambda\left\vert u\right\vert
^{\alpha}u,$ with $\alpha>2$, there is low energy scattering, i.e. a solution
of~(\ref{NLS0}) with a sufficiently small initial value (in some appropriate
sense) behaves asymptotically free as $\left\vert t\right\vert \rightarrow
\infty$, i.e. similar to the solutions of the linear Schr\"{o}dinger equation
$i\partial_{t}u=-\triangle u$ (\cite{Strauss2, GinibreV1, GinibreV2,
CazenaveW1, GinibreOV, NakanishiO, CN}). On the other hand, as it shown in
\cite[Theorem~3.2 and Example~3.3, p.~68]{Strauss} and~\cite{Barab}, if the
power $\alpha\leq2$, then low energy scattering for (\ref{NLS0}) cannot be
expected. In the case of the cubic power $\alpha=2$, the relevant notion for
equation (\ref{NLS0}) is modified scattering, i.e. standard scattering
modulated by a phase. The existence of modified wave operators was established
in~\cite{Ozawa1}. That is, for all sufficiently small asymptotic state $u^{+}%
$, there exists a solution of~(\ref{NLS0}), which behaves like $e^{i\phi
(t,\cdot)}e^{t\Delta}u^{+},$ as $t\rightarrow\infty$, where the phase $\phi$
is given explicitly in terms of $u^{+}$. (See also~\cite{Carles, HNST,
ShimomuraT}.) Conversely, for small initial values, it was shown
in~\cite{HayashiNau1} that there is such modification in the asymptotic
behavior of the corresponding solution to (\ref{NLS0}). (See
also~\cite{haka1998, HayashiNau2, KitaW, CN1}.)

The well-posedness of the perturbed NLS equation (\ref{1.1}) is also known
for a wide class of potentials $V$ (\cite{Cazenave}). Compared to the case of
the free NLS equation, not so much is known however about the large time
behaviour of the solutions to (\ref{1.1}). The existence of standing wave
solutions to (\ref{1.1}) was investigated in \cite{Floer}. Concerning the
existence of low energy scattering, the first paper in this direction to our
knowledge is \cite{Weder2000}. It was established that for the NLS equation
\begin{equation}
i\partial_{t}u=Hu+f\left(  u\right)  u, \label{1.49}%
\end{equation}
where the power of the nonlinearity $f\left(  u\right)  $ is $5\leq p<\infty,$
the nonlinear scattering operator $\mathbf{S}_{V}$, associated to
(\ref{1.49}), is a homeomorphism from some neighborhood of $0$ in
$\mathbf{L}^{2}$ onto itself. For the power nonlinearities with $3<p<5,$ the
existence of low energy scattering was established in \cite{Cuccagna}. In the
case of the NLS\ equation under a partial quadratic confinement, the existence
of low energy scattering was proved in \cite{Carles3}. The long-time behavior
of solutions to a focusing cubic one-dimensional NLS equation with a Dirac
potential was considered in \cite{deift}.

In our previous paper \cite{Ivan} we considered the cubic nonlinear
Schr\"{o}dinger equation (\ref{1.1}) in the case of generic potentials $V$.
That is, potentials $V$ such that the corresponding Jost solutions $f_{\pm
}\left(  x,k\right)  $ at zero energy ($k=0$) satisfy $\left[  f_{+}\left(
x,0\right)  ,f_{-}\left(  x,0\right)  \right]  \neq0,$ where $\left[
f,g\right]  $ denote the Wronskian of $f$ and $g,$ i.e. $\left[  f,g\right]
=g\partial_{x}f-f\partial_{x}g.$ See Section \ref{S2} for the definition of
the Jost functions. In \cite{Ivan}, we proved the existence of unique
solutions $u\in\mathbf{C}\left(  \left[  0,\infty\right)  ;\mathbf{H}%
^{1}\right)  $ of (\ref{1.1}) which decay for long times as solutions to the
linear free Schr\"{o}dinger equation $i\partial_{t}u+\frac{1}{2}\partial
_{x}^{2}u=0.$ Namely, $u$ satisfies $\left\Vert u\left(  t\right)  \right\Vert
_{\mathbf{L}^{\infty}}\leq C\left(  1+t\right)  ^{-\frac{1}{2}}.$ Moreover, we
obtained the large time asymptotic representation for these solutions, which
have a modified character. Thus, we proved that as in the case of the free NLS
equation (\ref{NLS0}), the cubic nonlinearity in the problem (\ref{1.1}) in
the case of generic potentials $V$ is critical from the point of view of large
time asymptotic behavior.

In the present paper we continue this study and consider the case of
exceptional potentials. We say that a potential $V$ is exceptional if the Jost
solutions at zero energy satisfy $\left[  f_{+}\left(  x,0\right)
,f_{-}\left(  x,0\right)  \right]  =0$. Equivalent definitions of the
exceptional potentials, as well as a comprehensive study of their properties
can be found in papers \cite{aktosun}, \cite{Aktosun2}, \cite{aktosun1}, and
the references there in. Note that the trivial potential $V\left(  x\right)
=0$ is exceptional. If a potential $V$ is non-trivial and positive, it is
generic. As an example of an exceptional case consider the square-well
potential in the depth when a bound state is added to the potential. Observe
that at any other depth the square-well potential is generic. This example
exhibits the unstable nature of the exceptional potentials: a small
perturbation of the exceptional case usually makes the case generic
(\cite{aktosun}).\ For the linear Schr\"{o}dinger equation $i\partial
_{t}u=-\frac{1}{2}\partial_{x}^{2}u+Vu,$ the transmission coefficient $T(k)$
defines the probability $\left\vert T(k)\right\vert ^{2}$ that a particle of
energy $k^{2}$ can tunnel through the potential $V.$ (For the definition of
the transmission and reflection coefficients $T\left(  k\right)  ,$ $R_{\pm
}\left(  k\right)  $ see Section \ref{S2}.) In the case of generic potentials,
the zero-energy transmission coefficient $T\left(  0\right)  $ is zero,
meaning that a particle with no energy cannot tunnel through a non-trivial
potential. This is not longer true for exceptional potentials, since in this
case $T\left(  0\right)  \neq0.$ Also, we observe that in the exceptional
case, the energy zero can be a resonant energy - an energy where the potential
$V$ is "perfectly transparent", that is $R_{\pm}\left(  0\right)  =0.$ We note
that the resonant energies are important for tunneling spectroscopy
(\cite{Wolf}). These special properties of the exceptional potentials make
them particulary interesting.

We now state our main result. First, we define the class of exceptional
potentials we will consider in this paper. For some $N\geq1,$ let us partition
the real axis $\mathbb{R}$ as $-\infty<x_{1}<x_{2}<...<x_{N}<+\infty.$ We
denote $I_{j}=\left(  x_{j-1},x_{j}\right)  $ for $j=1,...,$ $N+1,$ with
$x_{0}=-\infty,$ $x_{N+1}=+\infty.$ We obtain a fragmentation of the potential
$V\in\mathbf{L}^{1,1}$ by setting
\begin{equation}
V\left(  x\right)  =\sum_{j=1}^{N+1}V_{j}\left(  x\right)  , \label{3.23}%
\end{equation}
where
\[
V_{j}\left(  x\right)  =\left\{
\begin{array}
[c]{c}%
V\left(  x\right)  ,\text{ \ }x\in I_{j},\\
0,\text{ \ elsewhere.}%
\end{array}
\right.
\]
We assume the following:

\begin{condition}
\label{Cond1}The potential $V\in\mathbf{L}^{1,3}$ is exceptional, symmetric
and such that the linear Schr\"{o}dinger operator $H=-\frac{1}{2}\frac{d^{2}%
}{dx^{2}}+V$ does not have negative eigenvalues. In addition, for some
$N\geq1,$ there is a partition (\ref{3.23}) such that each part $\left\langle
\cdot\right\rangle ^{2}V_{j}\in W^{1,1}\left(  I_{j}\right)  .$
\end{condition}

\begin{remark}
Of course the trivial potential $V=0$ satisfies Condition \ref{Cond1}. An
example of a non-trivial potential satisfying the above assumptions is%
\begin{equation}
V\left(  x\right)  =V_{1}\left(  x\right)  +V_{2}\left(  x\right)  ,
\label{1.4}%
\end{equation}
where
\[
V_{1}\left(  x\right)  =\left\{
\begin{array}
[c]{c}%
A^{2},\text{ }x\in\left(  0,1\right)  ,\\
-B^{2},\text{ }x\in\left(  1,2\right)  ,\\
0,\text{ \ \ elsewhere,}%
\end{array}
\right.  \text{ \ \ }%
\]%
\[
V_{2}\left(  x\right)  =V_{1}\left(  -x\right)  ,
\]
with $B=\frac{\pi}{4}$ and $A$ satisfying $A\tanh A=\frac{\pi}{4}.$ Indeed, it
follows from Section IV of \cite{aktosun1} that $V_{1}$ is an exceptional
potential without bound states. Then, the same is true for $V_{2}$. Thus, from
Theorem 2.3 of \cite{aktosun} we see that $V$ is exceptional. Moreover, by (i)
on page 4254 of \cite{aktosun1} we conclude that $V$ does not have any eigenvalue.
\end{remark}

Define
\[
M=e^{\frac{ix^{2}}{2t}},
\]
and%
\[
\mathcal{D}_{t}\phi=\left(  it\right)  ^{-\frac{1}{2}}\phi\left(
xt^{-1}\right)  .
\]

\begin{theorem}
\label{Theorem 1.1} Suppose that $V$ satisfies Condition \ref{Cond1}. Let the
initial data $u_{0}\in$ $\mathbf{H}^{1}\cap\mathbf{H}^{0,1}$. In addition, we
suppose that $u_{0}\left(  x\right)  $ are odd if $T\left(  0\right)  =1$ and
$u_{0}\left(  x\right)  $ are even in the case when $T\left(  0\right)  =-1$.
Then, there exists $\varepsilon_{0}>0$ such that for any $0<\varepsilon
<\varepsilon_{0}$ and $\left\Vert u_{0}\right\Vert _{\mathbf{H}^{1}%
}+\left\Vert u_{0}\right\Vert _{\mathbf{H}^{0,1}}\leq\varepsilon,$ there
exists a unique solution $u\in\mathbf{C}\left(  \left[  0,\infty\right)
;\mathbf{H}^{1}\right)  $ of the Cauchy problem (\ref{1.1}) satisfying the
estimate
\begin{equation}
\left\Vert u\left(  t\right)  \right\Vert _{\mathbf{L}^{\infty}}\leq C\left(
1+t\right)  ^{-\frac{1}{2}}, \label{1.3}%
\end{equation}
for any $t\geq0.$ Moreover there exists a unique modified final state
$w_{+}\in\mathbf{L}^{\infty}$ such that the following asymptotics is valid
\begin{equation}
u\left(  t\right)  =M{\mathcal{D}}_{t}w_{+}e^{-i\lambda\left\vert
w_{+}\right\vert ^{2}\log t}+O\left(  \varepsilon t^{\beta-\frac{3}{4}%
}\right)  , \label{1.2}%
\end{equation}
for $t\rightarrow\infty,$ uniformly with respect to $x\in\mathbb{R}$, where
$\beta>0.$
\end{theorem}

\begin{remark}
Here are some comments on Theorem \ref{Theorem 1.1}.

(i) If the potential is symmetric, $\lim_{x\rightarrow-\infty}f_{+}\left(
x,0\right)  =\pm1$ and then, $T\left(  0\right)  =\pm1$ (see (\ref{7.27}) and
(\ref{3.3})).

(ii) The difference between the generic and exceptional cases becomes relevant
when the small-energy behavior of the scattering coefficients and of the wave
functions is considered (\cite{aktosun}). In the case of Theorem
\ref{Theorem 1.1}\ we need to impose the parity condition on the potential and
on the initial data in order to control these functions. This explains the
extra assumptions in Theorem \ref{Theorem 1.1} compared with Theorem 1.1 of
\cite{Ivan}.

(iii) We do not know if the results of Theorem \ref{Theorem 1.1} remain valid
if we remove the parity assumption. In the case of a generic potential without
any symmetry, an asymptotic expansion of the form (\ref{1.2}) was obtained in
\cite{Ivan}. This asymptotic expansion might be not true for exceptional
potentials if the symmetry condition is removed: there is an evidence that
(\ref{1.2}) need to be modified in the case of more general exceptional
potentials (see Remark \ref{Rem1}).

(iv) As far as we know the case when the potential $V$ has bounded states
remains open (even in the generic case). We aim to consider this situation in
future works.
\end{remark}

As in the generic case \cite{Ivan}, our proof of Theorem \ref{Theorem 1.1} is
based on the following decomposition of the solutions for (\ref{1.1}) (see
(\ref{7.49}))%
\begin{equation}
u\left(  t\right)  =M{\mathcal{D}}_{t}\mathcal{V}\left(  t\right)  w\left(
t\right)  , \label{1.5}%
\end{equation}
where $\mathcal{V}\left(  t\right)  $ is defined by (\ref{2.10}). Using the
asymptotics of $\mathcal{V}\left(  t\right)  $ we translate the problem to
estimating the $\mathbf{L}^{\infty}$ and $\mathbf{H}^{1}$ norms of the new
dependent variable $w$. Namely, we need to control the norm%
\[
\left\Vert u\right\Vert _{\mathbf{X}_{T}}=\sup_{t\in\left[  0,T\right]
}\left(  \left\Vert w\right\Vert _{\mathbf{L}^{\infty}}+\left\langle
t\right\rangle ^{-\beta}\left\Vert w\right\Vert _{\mathbf{H}^{1}}\right)  ,
\]
with $\beta>0$ small enough. Observe that we allow some small growth of the
Sobolev norm $\mathbf{H}^{1}$ of $w,$ as $t\rightarrow\infty.$ Hence, to
estimate $\left\Vert w\right\Vert _{\mathbf{L}^{\infty}}$ it is not enough to
use the Sobolev's embedding theorem. To estimate $\left\Vert u\right\Vert
_{\mathbf{X}_{T}}$ we are forced to control the derivative $\partial
_{x}\mathcal{V}\left(  t\right)  .$ We are able to do so under the parity
assumptions of Theorem \ref{Theorem 1.1} since in this case the scattering
coefficients, as well as the function $w$ have some special symmetries. We
also need to control the function $w$ at zero energy. In the case of generic
potentials, $w\left(  0\right)  =0$ and $\left(  \mathcal{V}\left(  t\right)
w\right)  \left(  0\right)  $ gain some extra decay, as $t\rightarrow\infty$.
This is not true for arbitrary exceptional potentials: in general $w\left(
0\right)  $ may not tend to zero$,$ as $t\rightarrow\infty.$ Fortunately,
under the assumptions of Theorem \ref{Theorem 1.1} $w\left(  t,0\right)  =0$
(see (\ref{7.40})). Also, we prove that $\left(  \mathcal{V}\left(  t\right)
w\right)  \left(  0\right)  $ has an extra decay, when $t\rightarrow\infty$
(see (\ref{5.15})). As it is pointed in Remark \ref{Rem1}, if $\lim
_{t\rightarrow\infty}w\left(  t,0\right)  \neq0,$ the asymptotics of the
solution $u$ may differ from (\ref{1.2}) because of an extra term of the form
$h\left(  t,x\right)  ,$ with $h\in$ $\mathbf{L}^{\infty}$. Once, we control
$\left\Vert u\right\Vert _{\mathbf{X}_{T}},$ via a local existence result, we
extend the solution for all times $t>0.$

The rest of the paper is organized as follows. In Section \ref{S2} we recall
some known properties for the linear Schr\"{o}dinger equations. The proof of
Theorem \ref{Theorem 1.1} is given in Section~\ref{S7}. This proof depends on
several intermediate results, which are stated in Section~\ref{S7}, but whose
proofs are deferred until Sections~\ref{S6} to~\ref{S5}. Finally, in Section
\ref{S3} we present some properties of the Jost solutions in the case of
exceptional potentials that are used in the proof.

Throughout this paper, we use the following notation. The usual Fourier
transforms $\mathcal{F}_{0}$ and $\mathcal{F}_{0}^{-1}$ are defined by
\[
\mathcal{F}_{0}\phi=\frac{1}{\sqrt{2\pi}}\int_{-\infty}^{\infty}e^{-ixk}%
\phi\left(  x\right)  dx
\]
and%
\[
\mathcal{F}_{0}^{-1}\phi=\frac{1}{\sqrt{2\pi}}\int_{-\infty}^{\infty}%
e^{ixk}\phi\left(  k\right)  dk,
\]
respectively. We denote by $\mathbf{L}^{p}=\mathbf{L}^{p}\left(
\Omega\right)  ,$ $1\leq p\leq\infty,$ the Lebesgue space on a domain
$\Omega\subset\mathbb{R}$ and the weighted Lebesgue spaces are given by%
\[
\mathbf{L}^{p,s}=\left\{  \phi:\left\Vert \phi\right\Vert _{\mathbf{L}^{p,s}%
}=\left\Vert \left\langle \cdot\right\rangle ^{s}\phi\left(  \cdot\right)
\right\Vert _{\mathbf{L}^{p}}<\infty\right\}  ,\text{ \ }s\in\mathbb{R}%
\text{,}%
\]
where
\[
\left\langle x\right\rangle =\left(  1+x^{2}\right)  ^{1/2}.
\]
For any $k\in\mathbb{N}\cup\{0\}$ and $p\geq1,$ we denote by $W^{k,p}=\left\{
\phi:\left(  \sum_{j=0}^{k}\left\Vert \phi^{\left(  k\right)  }\right\Vert
_{\mathbf{L}^{p}}^{p}\right)  ^{1/p}<\infty\right\}  $, the Sobolev space of
order $k$ based on $\mathbf{L}^{p}$ (see e.g.~\cite{adams} for the definitions
and properties of these spaces.) For any $r\in\mathbb{R}$, we denote by
$\mathbf{H}^{r}=\mathbf{H}^{r}\left(  \mathbb{R}\right)  $ the Sobolev space
based on $\mathbf{L}^{2}$ consisting of the completion of the Schwartz class
in the norm $\left\Vert \phi\right\Vert _{\mathbf{H}^{r}}=\left\Vert
\mathcal{F}_{0}\phi\right\Vert _{\mathbf{L}^{2,r}}$. Moreover, for any
$r,s\in\mathbb{R},$ we define the weighted Sobolev spaces by
\[
\mathbf{H}^{r,s}=\left\{  \phi:\left\Vert \phi\right\Vert _{\mathbf{H}^{r,s}%
}=\left\Vert \left\langle \cdot\right\rangle ^{s}\phi\left(  \cdot\right)
\right\Vert _{\mathbf{H}^{r}}<\infty\right\}  .
\]
Finally, the same letter $C$ may denote different positive constants which
particular value is irrelevant.

\section{\label{S2}Basic notions.}

\subsection{Free Schr\"{o}dinger equation.}

We consider first the free linear Schr\"{o}dinger equation%
\begin{equation}
\left\{
\begin{array}
[c]{c}%
i\partial_{t}u=H_{0}u,\\
u\left(  0,x\right)  =u_{0}\left(  x\right)
\end{array}
\right.  \label{2.1}%
\end{equation}
where $H_{0}=-\frac{1}{2}\frac{d^{2}}{dx^{2}}$ is the free Schr\"{o}dinger
operator. $H_{0}$ is a self-adjoint operator on $L^{2}$ with domain $D\left(
H_{0}\right)  =\mathbf{H}^{2}.$ The solution to (\ref{2.1}) is given by
$u\left(  t,x\right)  =\mathcal{U}_{0}\left(  t\right)  u_{0},$ where
\begin{equation}
\mathcal{U}_{0}\left(  t\right)  =e^{-itH_{0}}=\mathcal{F}_{0}^{-1}%
e^{-\frac{it}{2}k^{2}}\mathcal{F}_{0} \label{2.13}%
\end{equation}
is the free Schr\"{o}dinger evolution group. The asymptotic representation
formula for $\mathcal{U}_{0}\left(  t\right)  $ is the relation (see
\cite{H-O})%
\begin{equation}
\mathcal{U}_{0}\left(  t\right)  \mathcal{F}_{0}^{-1}=M\mathcal{D}%
_{t}\mathcal{V}_{0}\left(  t\right)  , \label{2.2}%
\end{equation}
where $M=e^{\frac{ix^{2}}{2t}},$ $\mathcal{D}_{t}\phi=\left(  it\right)
^{-\frac{1}{2}}\phi\left(  xt^{-1}\right)  $ and
\begin{equation}
\mathcal{V}_{0}\left(  t\right)  \phi=\mathcal{F}_{0}M\mathcal{F}_{0}^{-1}%
\phi=\sqrt{\frac{it}{2\pi}}%
{\displaystyle\int\limits_{\mathbf{-\infty}}^{\infty}}
e^{-\frac{it}{2}\left(  x-\xi\right)  ^{2}}\phi\left(  \xi\right)  d\xi.
\label{2.12}%
\end{equation}
Observe that
\begin{equation}
\left\Vert \mathcal{V}_{0}\left(  t\right)  \phi\right\Vert _{\mathbf{L}^{2}%
}=\left\Vert \mathcal{F}_{0}M\mathcal{F}_{0}^{-1}\phi\right\Vert
_{\mathbf{L}^{2}}=\left\Vert \phi\right\Vert _{\mathbf{L}^{2}}, \label{2.11}%
\end{equation}%
\[
\left\Vert \left(  \mathcal{V}_{0}\left(  t\right)  -1\right)  \phi\right\Vert
_{\mathbf{L}^{2}}=\left\Vert \mathcal{F}_{0}\left(  M-1\right)  \mathcal{F}%
_{0}^{-1}\phi\right\Vert _{\mathbf{L}^{2}}\leq C\left\Vert \left(  \frac
{x^{2}}{2t}\right)  ^{\frac{1}{2}}\mathcal{F}_{0}^{-1}\phi\right\Vert
_{\mathbf{L}^{2}}\leq Ct^{-\frac{1}{2}}\left\Vert \partial_{k}\phi\right\Vert
_{\mathbf{L}^{2}}%
\]
and
\[
\left\Vert \partial_{x}\left(  \mathcal{V}_{0}\left(  t\right)  -1\right)
\phi\right\Vert _{\mathbf{L}^{2}}\leq C\left\Vert \partial_{k}\phi\right\Vert
_{\mathbf{L}^{2}}.
\]
Then, using that $\left\Vert \phi\right\Vert _{\mathbf{L}^{\infty}}\leq
C\left\Vert \phi\right\Vert _{\mathbf{L}^{2}}^{\frac{1}{2}}\left\Vert
\partial_{x}\phi\right\Vert _{\mathbf{L}^{2}}^{\frac{1}{2}}$ we obtain the
estimate
\begin{equation}
\left\Vert \left(  \mathcal{V}_{0}\left(  t\right)  -1\right)  \phi\right\Vert
_{\mathbf{L}^{\infty}}\leq Ct^{-\frac{1}{4}}\left\Vert \partial_{k}%
\phi\right\Vert _{\mathbf{L}^{2}}. \label{2.3}%
\end{equation}

\subsection{Linear Schr\"{o}dinger equation with a potential.}

Consider now the linear Schr\"{o}dinger equation with a potential
\begin{equation}
\left\{
\begin{array}
[c]{c}%
i\partial_{t}u=Hu,\text{ }\\
u\left(  0,x\right)  =u_{0}\left(  x\right)  \text{,}%
\end{array}
\right.  \label{2.4}%
\end{equation}
where $H=H_{0}+V\left(  x\right)  $ and the potential $V\left(  x\right)  $ is
a real valued function for all $x\in\mathbb{R}$. If $V\in\mathbf{L}^{1,1},$ it
is known that the operator $H$ is self-adjoint (\cite{Weder2000}), the
absolutely-continuous spectrum of $H$ is given by $\sigma_{\operatorname*{ac}%
}\left(  H\right)  =[0,\infty),$ $H$ has no singular-continuous spectrum, $H$
has no eigenvalues that are positive or equal to $0$ and $H$ has a finite
number of negative eigenvalues that are simple (\cite{DeiftTrubowitz}).

We want to derive an asymptotic representation formula for the evolution group
defined by (\ref{2.4}), similar to (\ref{2.2}). For that purpose, we need the
spectral theorem for the perturbed operator $H$, that follows from the
Weyl-Kodaira-Titchmarsch theory (see, for example, \cite{DeiftTrubowitz}).
Assume that $V\in\mathbf{L}^{1,1}.$ The Jost functions $f_{\pm}\left(
x,k\right)  $ are solutions to the stationary Schr\"{o}dinger equation
\begin{equation}
\left\{
\begin{array}
[c]{c}%
-\frac{1}{2}\frac{d^{2}}{dx^{2}}f_{\pm}+Vf_{\pm}=\frac{k^{2}}{2}f_{\pm},\\
f_{\pm}\left(  x,k\right)  =e^{\pm ikx},\text{ }x\rightarrow\pm\infty.
\end{array}
\right.  \label{2.6}%
\end{equation}
They are solutions to Volterra integral equations and are obtained by
iteration as uniformly convergent series (\cite{DeiftTrubowitz},
\cite{Weder2000}). Since $f_{\pm}\left(  x,k\right)  $ are independent
solutions to (\ref{2.6}) for $k\neq0,$ there are unique coefficients $T\left(
k\right)  $ and $R_{\pm}\left(  k\right)  $ such that%
\begin{equation}
f_{\pm}\left(  x,k\right)  =\frac{R_{\mp}\left(  k\right)  }{T\left(
k\right)  }f_{\mp}\left(  x,k\right)  +\frac{1}{T\left(  k\right)  }f_{\mp
}\left(  x,-k\right)  , \label{2.5}%
\end{equation}
for $k\neq0$ (see page 144 of \cite{DeiftTrubowitz}). The function $T\left(
k\right)  $ is the transmission coefficient, $R_{-}\left(  k\right)  $ is the
reflection coefficient from left to right of the plane wave $e^{ikx}$ and
$R_{+}\left(  k\right)  $ is the reflection coefficient from right to left of
the plane wave $e^{-ikx}.$ They satisfy the identity (the unitarity of the
scattering matrix)
\begin{equation}
\left\vert T\left(  k\right)  \right\vert ^{2}+\left\vert R_{\pm}\left(
k\right)  \right\vert ^{2}=1,\text{ }k\in\mathbb{R}. \label{3.4}%
\end{equation}

We now introduce the generalized Fourier transform $\mathcal{F}$. Let the
Heaviside function $\theta\left(  k\right)  $ be $\theta\left(  k\right)  =1$
for \ $k\geq0,$ and $\theta\left(  k\right)  =0$ for \ $k<0.$ We define%
\begin{equation}
\Psi\left(  x,k\right)  =\theta\left(  k\right)  T\left(  k\right)
f_{+}\left(  x,k\right)  +\theta\left(  -k\right)  T\left(  -k\right)
f_{-}\left(  x,-k\right)  . \label{2.9}%
\end{equation}
The generalized Fourier transform
\begin{equation}
\mathcal{F}\phi=\frac{1}{\sqrt{2\pi}}\int_{-\infty}^{\infty}\overline
{\Psi\left(  x,k\right)  }\phi\left(  x\right)  dx \label{2.15}%
\end{equation}
is unitary from the continuous subspace of $H$ denoted by
$\mathcal{\mathcal{H}}_{c}$ onto $\mathbf{L}^{2}.$ Under the assumption that
$V$ has no negative eigenvalues $\mathcal{F}$ is unitary on $\mathbf{L}^{2}$
and $\mathcal{F}^{-1}$ is given by
\begin{equation}
\mathcal{F}^{-1}\phi=\frac{1}{\sqrt{2\pi}}\int_{-\infty}^{\infty}\Psi\left(
x,k\right)  \phi\left(  k\right)  dk. \label{2.16}%
\end{equation}
Then, the operator $\mathcal{F}H\mathcal{F}^{-1}$ acts as multiplication by
$\frac{k^{2}}{2}$ and the solution to (\ref{2.4}) is given by $u=\mathcal{U}%
\left(  t\right)  u_{0},$ where
\[
\mathcal{U}\left(  t\right)  =e^{-itH}=\mathcal{F}^{-1}e^{-\frac{it}{2}k^{2}%
}\mathcal{F}.
\]

We now deduce an asymptotic representation formula for $\mathcal{U}\left(
t\right)  $. We define
\begin{equation}
\Phi\left(  x,k\right)  =e^{-ikx}\Psi\left(  x,k\right)  . \label{2.8}%
\end{equation}
Then, we have
\[
\mathcal{U}\left(  t\right)  \mathcal{F}^{-1}\phi=\frac{1}{\sqrt{2\pi}}%
\int_{-\infty}^{\infty}e^{-\frac{it}{2}k^{2}}\Psi\left(  x,k\right)
\phi\left(  k\right)  dk=\frac{1}{\sqrt{2\pi}}\int_{-\infty}^{\infty}%
e^{-\frac{it}{2}k^{2}+ikx}\Phi\left(  x,k\right)  \phi\left(  k\right)  dk,
\]
and thus%
\begin{equation}
\left.
\begin{array}
[c]{c}%
\mathcal{U}\left(  t\right)  \mathcal{F}^{-1}\phi=\frac{1}{\sqrt{2\pi}}%
\int_{-\infty}^{\infty}e^{-\frac{it}{2}k^{2}+ikx}\Phi\left(  x,k\right)
\phi\left(  k\right)  dk={\mathcal{D}}_{t}\sqrt{\frac{it}{2\pi}}\int_{-\infty
}^{\infty}e^{-\frac{it}{2}k^{2}+itkx}\Phi\left(  tx,k\right)  \phi\left(
k\right)  dk\\
=\sqrt{\frac{it}{2\pi}}M{\mathcal{D}}_{t}\int_{-\infty}^{\infty}e^{-\frac
{it}{2}x^{2}-\frac{it}{2}k^{2}+itkx}\Phi\left(  tx,k\right)  \phi\left(
k\right)  dk=\sqrt{\frac{it}{2\pi}}M{\mathcal{D}}_{t}\int_{-\infty}^{\infty
}e^{-\frac{it}{2}\left(  k-x\right)  ^{2}}\Phi\left(  tx,k\right)  \phi\left(
k\right)  dk.
\end{array}
\right.  \label{2.7}%
\end{equation}
Denoting
\begin{equation}
\mathcal{V}\left(  t\right)  \phi=\sqrt{\frac{it}{2\pi}}\int_{-\infty}%
^{\infty}e^{-\frac{it}{2}\left(  k-x\right)  ^{2}}\Phi\left(  tx,k\right)
\phi\left(  k\right)  dk={\mathcal{D}}_{t}^{-1}M^{-1}\mathcal{F}^{-1}%
e^{-\frac{it}{2}k^{2}}\phi, \label{2.10}%
\end{equation}
from (\ref{2.7}) we obtain
\begin{equation}
\mathcal{U}\left(  t\right)  \mathcal{F}^{-1}\phi=M{\mathcal{D}}%
_{t}\mathcal{V}\left(  t\right)  \phi, \label{2.17}%
\end{equation}
which is an asymptotic representation formula for the evolution group
$\mathcal{U}\left(  t\right)  $. We also calculate the inverse asymptotic
representation formula%
\begin{equation}
\mathcal{FU}\left(  -t\right)  \phi=\mathcal{V}^{-1}\left(  t\right)
{\mathcal{D}}_{t}^{-1}\overline{M}\phi, \label{2.14}%
\end{equation}
where
\[
\mathcal{V}^{-1}\left(  t\right)  \phi=\sqrt{\frac{t}{2\pi i}}\int_{-\infty
}^{\infty}e^{\frac{it}{2}\left(  k-x\right)  ^{2}}\overline{\Phi\left(
tx,k\right)  }\phi\left(  x\right)  dx=e^{\frac{it}{2}k^{2}}\mathcal{F}%
M{\mathcal{D}}_{t}\phi.
\]

\section{\label{S7}Proof of Theorem \ref{Theorem 1.1}.}

This section is devoted to the proof of the main result. Suppose that the
potential $V\in\mathbf{L}^{1,1}$ is exceptional. Then, the Jost solutions at
$k=0$ satisfy $\left[  f_{+}\left(  x,0\right)  ,f_{-}\left(  x,0\right)
\right]  =0$. In this case, $a=\lim_{x\rightarrow-\infty}f_{+}\left(
x,0\right)  \neq0$ (\cite{Klaus}, \cite{aktosun}, \cite{Weder2000}). We also
assume that the potential $V\left(  x\right)  $ is symmetric, i.e. $V\left(
-x\right)  =V\left(  x\right)  .$ This ensures that the function $f_{+}\left(
x,0\right)  $ is either even or odd (see \cite{Bransden}, page 160). Indeed,
for symmetric potentials%
\begin{equation}
f_{+}\left(  x,k\right)  =f_{-}\left(  -x,k\right)  . \label{7.26}%
\end{equation}
Then, since $V$ is exceptional $f_{+}\left(  x,0\right)  =\alpha f_{+}\left(
-x,0\right)  ,$ for some $\alpha\in\mathbb{R}$. Changing $x\rightarrow-x$ in
the last equation we get $f_{+}\left(  -x,0\right)  =\alpha f_{+}\left(
x,0\right)  .$ The combination of these two equations yields $f_{+}\left(
x,0\right)  =\alpha^{2}f_{+}\left(  x,0\right)  .$ Then, $\alpha=\pm1,$ and
thus, $f_{+}\left(  x,0\right)  $ is even or odd. Therefore,
\begin{equation}
a=\lim_{x\rightarrow-\infty}f_{+}\left(  x,0\right)  =\pm1. \label{7.27}%
\end{equation}
Also, we observe that for symmetric potential $V\left(  x\right)  $ the
problem (\ref{1.1}) conserves the parity condition, i.e. if the\ initial data
$u_{0}\left(  x\right)  $ are even (odd) then the solution $u\left(
t,x\right)  $ is even (odd) for all time. From (\ref{7.26}) it follows that
$\Psi\left(  -x,-k\right)  =\Psi\left(  x,k\right)  ,$ so that if $\phi\left(
x\right)  $ is even (odd) and the potential $V$ is symmetric, then
$\mathcal{F}\phi$ is also even (odd). Then, we suppose additionally that the
initial data $u_{0}\left(  x\right)  $ are odd if $a=1$ and $u_{0}\left(
x\right)  $ are even if $a=-1$. We also note that in the case of symmetric
potentials $R_{+}\left(  k\right)  =R_{-}\left(  k\right)  .$ Then, we denote
\begin{equation}
R\left(  k\right)  =R_{+}\left(  k\right)  =R_{-}\left(  k\right)  .
\label{3.20}%
\end{equation}
The parity assumption on the potential and initial data is made in order to
assure that (\ref{7.27}) is true and $\mathcal{F}\left(  u\left(  t\right)
\right)  $ is either even or odd. We need this extra information in order to
control the small-energy behaviour of the scattering coefficients and the wave functions.

We begin by presenting the results of Sections \ref{S6}, \ref{S4} and \ref{S5}
that are involved in the proof of Theorem \ref{Theorem 1.1}. We note that
under Condition \ref{Cond1}, all the relations of Section \ref{S3} are valid.
Taking into account the commentaries about the properties of the symmetric
potentials made at the beginning of this section, we see that for potentials
satisfying Condition \ref{Cond1}, the results of Sections \ref{S6}, \ref{S4}
and \ref{S5} are true. Therefore, we have the following.

\begin{lemma}
\label{L3}Suppose that $V$ satisfies Condition \ref{Cond1}. If $T\left(
0\right)  =1,$ let $\phi\in\mathbf{H}^{0,1}$ be odd and if $T\left(  0\right)
=-1,$ suppose that $\phi$ is even. Then%
\begin{equation}
\left(  \mathcal{F}\phi\right)  \left(  0\right)  =0. \label{7.50}%
\end{equation}
Moreover
\begin{equation}
\left\Vert \mathcal{F}\phi\right\Vert _{\mathbf{H}^{1}}\leq K_{0}\left\Vert
\phi\right\Vert _{\mathbf{H}^{0,1}}, \label{7.21}%
\end{equation}
for some $K_{0}>0$ is true.\
\end{lemma}

\begin{lemma}
\label{L4}Suppose that $V$ satisfies Condition \ref{Cond1}. Then, for any
$t\in\mathbb{R}$ and $\phi\in\mathbf{H}^{1}\cap\mathbf{H}^{0,1}$ the estimates%
\begin{equation}
\left\Vert \mathcal{U}\left(  t\right)  \phi\right\Vert _{\mathbf{H}^{1}}\leq
K_{0}\left\Vert \phi\right\Vert _{\mathbf{H}^{1}} \label{7.8}%
\end{equation}
and
\begin{equation}
\left\Vert \mathcal{U}\left(  t\right)  \phi\right\Vert _{\mathbf{H}^{0,1}%
}\leq K_{0}\left(  \left\langle t\right\rangle \left\Vert \phi\right\Vert
_{\mathbf{H}^{1}}+\left\Vert \phi\right\Vert _{\mathbf{H}^{0,1}}\right)
\label{7.9}%
\end{equation}
for some $K_{0}>0$ are valid.
\end{lemma}

\begin{lemma}
\label{L1}Suppose that $V$ satisfies Condition \ref{Cond1}$.$ Then, if
$T\left(  0\right)  =1,$ the estimate
\begin{equation}
\left.  \left\Vert \mathcal{V}\left(  t\right)  w-T\left(  \left\vert
x\right\vert \right)  w\left(  x\right)  -R\left(  \left\vert x\right\vert
\right)  w\left(  -x\right)  \right\Vert _{\mathbf{L}^{\infty}}\leq
C\left\vert w\left(  0\right)  \right\vert +Ct^{-1/4}\left\Vert w\right\Vert
_{\mathbf{H}^{1}}\right.  \label{L1.1}%
\end{equation}
is true for all $t\geq1.$ Moreover, in the case $T\left(  0\right)  =-1,$
\begin{equation}
\left.  \left\Vert \mathcal{V}\left(  t\right)  w-T\left(  \left\vert
x\right\vert \right)  w\left(  x\right)  -R\left(  \left\vert x\right\vert
\right)  w\left(  -x\right)  -\sqrt{\frac{2i}{\pi}}w\left(  0\right)
{\displaystyle\int\limits_{\sqrt{t}x}^{\infty}}
e^{-\frac{ik^{2}}{2}}dk\right\Vert _{\mathbf{L}^{\infty}}\leq Ct^{-1/4}%
\left\Vert w\right\Vert _{\mathbf{H}^{1}}\right.  \label{L1.2}%
\end{equation}
Furthermore, if $T\left(  0\right)  =1,$ let $w\in\mathbf{H}^{1}$ be odd and
if $T\left(  0\right)  =-1,$ suppose that $w$ is even. Also suppose that $w$
can be represented as $w=\mathcal{F}\psi,$ for some $\psi\in\mathbf{H}^{0,1}.$
Then, the estimate
\begin{equation}
\left\Vert \partial_{x}\mathcal{V}\left(  t\right)  w\right\Vert
_{\mathbf{L}^{2}}\leq C\left\Vert w\right\Vert _{\mathbf{L}^{\infty}}%
\log\left\langle t\right\rangle +C\left\Vert w\right\Vert _{\mathbf{H}^{1}}
\label{L1.3}%
\end{equation}
is true for all $t\geq1$.
\end{lemma}

\begin{remark}
\label{Rem1}It can be proved that in general the following asymptotic formula
holds
\[
\left.  \left\Vert \mathcal{V}\left(  t\right)  w-T\left(  \left\vert
x\right\vert \right)  w\left(  x\right)  -R\left(  \left\vert x\right\vert
\right)  w\left(  -x\right)  -h\left(  t,x\right)  w\left(  0\right)
\right\Vert _{\mathbf{L}^{\infty}}\leq Ct^{-1/4}\left\Vert w\right\Vert
_{\mathbf{H}^{1}},\right.
\]
for some function $h\in\mathbf{L}^{\infty}\left(  \left(  0,\infty\right)
\times\mathbb{R}\right)  .$ In principle $h$ might be not identically zero.
Thus, if $\lim_{t\rightarrow\infty}w\left(  t,0\right)  \neq0,$ the above
expansion together with (\ref{1.5}) suggest a different behaviour for
$u\left(  t\right)  ,$ as $t\rightarrow\infty,$ compared with the asymptotics
in the generic case obtained in \cite{Ivan} or (\ref{1.2}).
\end{remark}

\begin{lemma}
\label{L2}If $V$ satisfies Condition \ref{Cond1}, the estimates%
\begin{equation}
\left\Vert \mathcal{V}^{-1}\left(  t\right)  \phi-\overline{T\left(
\left\vert k\right\vert \right)  }\phi\left(  k\right)  -\overline{R\left(
\left\vert k\right\vert \right)  }\phi\left(  -k\right)  \right\Vert
_{\mathbf{L}^{\infty}}\leq C\left\vert \phi\left(  0\right)  \right\vert
+Ct^{-\frac{1}{4}}\left\Vert \phi\right\Vert _{\mathbf{H}^{1}}, \label{L2.1}%
\end{equation}
and
\begin{equation}
\left\Vert \partial_{k}\mathcal{V}^{-1}\left(  t\right)  \phi\right\Vert
_{\mathbf{L}^{2}}\leq Ct^{\frac{1}{2}}\left\vert \phi\left(  0\right)
\right\vert +C\left\Vert \phi\right\Vert _{\mathbf{H}^{1}} \label{L2.2}%
\end{equation}
are valid for all $t\geq1$.
\end{lemma}

We postpone the proof of Lemmas \ref{L3}, \ref{L4}, \ref{L1} and \ref{L2} to
Sections \ref{S6}, \ref{S4} and \ref{S5}. We now use these results to prove
some estimates.

\begin{lemma}
\label{Lemma 6.1}Suppose that $V$ satisfies Condition \ref{Cond1}. If
$T\left(  0\right)  =1,$ let $w\in\mathbf{H}^{1}$ be odd and if $T\left(
0\right)  =-1,$ suppose that $w$ is even. Also suppose that $w$ can be
represented as $w=\mathcal{F}\psi,$ for some $\psi\in\mathbf{H}^{0,1}.$ Then,
for any $0<\beta<\frac{1}{4},$ the asymptotics%
\begin{equation}
\left.
\begin{array}
[c]{c}%
\mathcal{V}^{-1}\left(  t\right)  \left(  \left\vert \mathcal{V}\left(
t\right)  w\right\vert ^{2}\mathcal{V}\left(  t\right)  w\right)  =\left\vert
w\left(  k\right)  \right\vert ^{2}w\left(  k\right) \\
+O\left(  \left(  \left\vert w\left(  0\right)  \right\vert +t^{-\frac{1}{4}%
}\left\Vert w\right\Vert _{\mathbf{H}^{1}}\right)  \left(  \left\Vert
w\right\Vert _{\mathbf{L}^{\infty}}+t^{-\beta}\left\Vert w\right\Vert
_{\mathbf{H}^{1}}\right)  ^{2}\right)  +O\left(  t^{\beta-\frac{1}{4}}\left(
\left\Vert w\right\Vert _{\mathbf{L}^{\infty}}+t^{-\beta}\left\Vert
w\right\Vert _{\mathbf{H}^{1}}\right)  ^{3}\right)
\end{array}
\right.  \label{5.18}%
\end{equation}
holds for all $t\geq1$.
\end{lemma}

\begin{proof}
It follows from (\ref{L2.1}) that
\begin{equation}
\mathcal{V}^{-1}\left(  t\right)  \left(  \left\vert \mathcal{V}\left(
t\right)  w\right\vert ^{2}\mathcal{V}\left(  t\right)  w\right)
=\overline{T\left(  \left\vert k\right\vert \right)  }\left(  \left\vert
\mathcal{V}\left(  t\right)  w\right\vert ^{2}\mathcal{V}\left(  t\right)
w\right)  \left(  k\right)  +\overline{R\left(  \left\vert k\right\vert
\right)  }\left(  \left\vert \mathcal{V}\left(  t\right)  w\right\vert
^{2}\mathcal{V}\left(  t\right)  w\right)  \left(  -k\right)  +\mathbf{R}_{1},
\label{5.21}%
\end{equation}
where
\begin{equation}
\left\vert \mathbf{R}_{1}\right\vert \leq C\left\vert \left(  \mathcal{V}%
\left(  t\right)  w\right)  \left(  0\right)  \right\vert ^{3}+C\left\langle
t\right\rangle ^{-\frac{1}{4}}\left\Vert \mathcal{V}\left(  t\right)
w\right\Vert _{\mathbf{L}^{\infty}}^{2}\left\Vert \mathcal{V}\left(  t\right)
w\right\Vert _{\mathbf{H}^{1}}. \label{5.14}%
\end{equation}
If $T\left(  0\right)  =1,$ $w$ is odd. In particular, $w\left(  0\right)
=0.$ If $T\left(  0\right)  =-1,$ then%
\[
T\left(  0\right)  w\left(  0\right)  +R\left(  0\right)  w\left(  0\right)
+\sqrt{\frac{2i}{\pi}}w\left(  0\right)
{\displaystyle\int\limits_{0}^{\infty}}
e^{-\frac{ik^{2}}{2}}dk=0.
\]
Hence, for $T\left(  0\right)  =\pm1,$ by (\ref{L1.1}) and (\ref{L1.2}) we
get
\begin{equation}
\left\vert \left(  \mathcal{V}\left(  t\right)  w\right)  \left(  0\right)
\right\vert \leq Ct^{-\frac{1}{4}}\left\Vert w\right\Vert _{\mathbf{H}^{1}}
\label{5.15}%
\end{equation}
and via (\ref{3.4})
\begin{equation}
\left\Vert \mathcal{V}\left(  t\right)  w\right\Vert _{\mathbf{L}^{\infty}%
}\leq C\left\Vert w\right\Vert _{\mathbf{L}^{\infty}}+Ct^{-1/4}\left\Vert
w\right\Vert _{\mathbf{H}^{1}}. \label{5.17}%
\end{equation}
Using (\ref{5.15}), (\ref{5.17}) and (\ref{L1.3}) in (\ref{5.14}), as
$\left\Vert \mathcal{V}\left(  t\right)  w\right\Vert _{\mathbf{L}^{2}%
}=\left\Vert w\right\Vert _{\mathbf{L}^{2}}$ we estimate%
\[
\left\vert \mathbf{R}_{1}\right\vert \leq Ct^{\beta-\frac{1}{4}}\left(
\left\Vert w\right\Vert _{\mathbf{L}^{\infty}}+t^{-\beta}\left\Vert
w\right\Vert _{\mathbf{H}^{1}}\right)  ^{3}.
\]
By (\ref{L1.1}) and (\ref{L1.2}) we get%
\[
\mathcal{V}\left(  t\right)  w=T\left(  \left\vert x\right\vert \right)
w\left(  x\right)  +R\left(  \left\vert x\right\vert \right)  w\left(
-x\right)  +O\left(  \left\vert w\left(  0\right)  \right\vert +t^{-\frac
{1}{4}}\left\Vert w\right\Vert _{\mathbf{H}^{1}}\right)  .
\]
Then, using (\ref{3.4}) and (\ref{5.17}), from (\ref{5.21}) we derive%
\begin{equation}
\left.
\begin{array}
[c]{c}%
\mathcal{V}^{-1}\left(  t\right)  \left(  \left\vert \mathcal{V}\left(
t\right)  w\right\vert ^{2}\mathcal{V}\left(  t\right)  w\right)
=\overline{T\left(  \left\vert k\right\vert \right)  }\left(  \left\vert
T\left(  \left\vert k\right\vert \right)  w\left(  k\right)  +R\left(
\left\vert k\right\vert \right)  w\left(  -k\right)  \right\vert ^{2}\left(
T\left(  \left\vert k\right\vert \right)  w\left(  k\right)  +R\left(
\left\vert k\right\vert \right)  w\left(  -k\right)  \right)  \right) \\
+\overline{R\left(  \left\vert k\right\vert \right)  }\left\vert T\left(
\left\vert k\right\vert \right)  w\left(  -k\right)  +R\left(  \left\vert
k\right\vert \right)  w\left(  k\right)  \right\vert ^{2}\left(  T\left(
\left\vert k\right\vert \right)  w\left(  -k\right)  +R\left(  \left\vert
k\right\vert \right)  w\left(  k\right)  \right) \\
+O\left(  \left(  \left\vert w\left(  0\right)  \right\vert +t^{-\frac{1}{4}%
}\left\Vert w\right\Vert _{\mathbf{H}^{1}}\right)  \left(  \left\Vert
w\right\Vert _{\mathbf{L}^{\infty}}+t^{-\beta}\left\Vert w\right\Vert
_{\mathbf{H}^{1}}\right)  ^{2}\right)  +O\left(  t^{\beta-\frac{1}{4}}\left(
\left\Vert w\right\Vert _{\mathbf{L}^{\infty}}+t^{-\beta}\left\Vert
w\right\Vert _{\mathbf{H}^{1}}\right)  ^{3}\right)  .
\end{array}
\right.  \label{5.19}%
\end{equation}
Note that (see (\ref{3.20}) and (\ref{i1}) below)
\begin{equation}
T\left(  k\right)  \overline{R\left(  k\right)  }+\overline{T\left(  k\right)
}R\left(  k\right)  =0. \label{5.24}%
\end{equation}
Then, using that $w$ is odd for $T\left(  0\right)  =1$ and even for $T\left(
0\right)  =-1,$ in view of (\ref{5.24}) and (\ref{3.4}) we obtain%
\begin{align*}
&  \left\vert T\left(  \left\vert k\right\vert \right)  w\left(  k\right)
+R\left(  \left\vert k\right\vert \right)  w\left(  -k\right)  \right\vert
^{2}=\left(  T\left(  \left\vert k\right\vert \right)  \mp R\left(  \left\vert
k\right\vert \right)  \right)  \left(  \overline{T\left(  \left\vert
k\right\vert \right)  }\mp\overline{R\left(  \left\vert k\right\vert \right)
}\right)  \left\vert w\left(  k\right)  \right\vert ^{2}\\
&  =\left(  \left\vert T\left(  \left\vert k\right\vert \right)  \right\vert
^{2}+\left\vert R\left(  \left\vert k\right\vert \right)  \right\vert ^{2}%
\mp\left(  T\left(  \left\vert k\right\vert \right)  \overline{R\left(
\left\vert k\right\vert \right)  }+\overline{T\left(  \left\vert k\right\vert
\right)  }R\left(  \left\vert k\right\vert \right)  \right)  \right)
\left\vert w\left(  k\right)  \right\vert ^{2}=\left\vert w\left(  k\right)
\right\vert ^{2},
\end{align*}
for $T\left(  0\right)  =\pm1.$ Hence, using again (\ref{5.24}) and
(\ref{3.4}) from (\ref{5.19}) we obtain (\ref{5.18}).
\end{proof}

Also, we present the estimate for the derivative of the nonlinear term.

\begin{lemma}
\label{Lemma 6.2}Suppose that $V$ satisfies Condition \ref{Cond1}. If
$T\left(  0\right)  =1,$ let $w\in\mathbf{H}^{1}$ be odd and if $T\left(
0\right)  =-1,$ suppose that $w$ is even. Also suppose that $w$ can be
represented as $w=\mathcal{F}\psi,$ for some $\psi\in\mathbf{H}^{0,1}.$ Then
the estimate%
\begin{equation}
\left\Vert \partial_{k}\mathcal{V}^{-1}\left(  t\right)  \left(  \left\vert
\mathcal{V}\left(  t\right)  w\right\vert ^{2}\mathcal{V}\left(  t\right)
w\right)  \right\Vert _{\mathbf{L}^{2}}\leq C\left\Vert w\right\Vert
_{\mathbf{L}^{\infty}}^{3}\log\left\langle t\right\rangle +Ct^{-1/4}\left\Vert
w\right\Vert _{\mathbf{H}^{1}}^{3}+C\left\Vert w\right\Vert _{\mathbf{L}%
^{\infty}}^{2}\left\Vert w\right\Vert _{\mathbf{H}^{1}} \label{5.20}%
\end{equation}
is valid for all $t\geq1$.
\end{lemma}

\begin{proof}
(\ref{L2.2}) implies%
\begin{align*}
&  \left\Vert \partial_{k}\mathcal{V}^{-1}\left(  t\right)  \left(  \left\vert
\mathcal{V}\left(  t\right)  w\right\vert ^{2}\mathcal{V}\left(  t\right)
w\right)  \right\Vert _{\mathbf{L}^{2}}\\
&  \leq Ct^{\frac{1}{2}}\left\vert \left(  \left\vert \mathcal{V}\left(
t\right)  w\right\vert ^{2}\mathcal{V}\left(  t\right)  w\right)  \left(
0\right)  \right\vert +C\left\Vert \left\vert \mathcal{V}\left(  t\right)
w\right\vert ^{2}\mathcal{V}\left(  t\right)  w\right\Vert _{\mathbf{H}^{1}}\\
&  \leq Ct^{\frac{1}{2}}\left\vert \left(  \left\vert \mathcal{V}\left(
t\right)  w\right\vert ^{2}\mathcal{V}\left(  t\right)  w\right)  \left(
0\right)  \right\vert +C\left\Vert \mathcal{V}\left(  t\right)  w\right\Vert
_{\mathbf{L}^{\infty}}^{2}\left\Vert \mathcal{V}\left(  t\right)  w\right\Vert
_{\mathbf{H}^{1}}.
\end{align*}
By (\ref{L1.3})%
\[
\left\Vert \partial_{x}\mathcal{V}\left(  t\right)  \phi\right\Vert
_{\mathbf{L}^{2}}\leq C\left\Vert w\right\Vert _{\mathbf{L}^{\infty}}%
\log\left\langle t\right\rangle +C\left\Vert w\right\Vert _{\mathbf{H}^{1}}.
\]
Hence, using (\ref{5.15}), (\ref{5.17}) and $\left\Vert \mathcal{V}\left(
t\right)  w\right\Vert _{\mathbf{L}^{2}}=\left\Vert w\right\Vert
_{\mathbf{L}^{2}}$ we deduce (\ref{5.20}).
\end{proof}

We are now in position to prove the main result. Using (\ref{2.17}) we
represent the solution to (\ref{1.1}) as
\begin{equation}
u\left(  t\right)  =\mathcal{U}\left(  t\right)  \mathcal{F}^{-1}%
w=M{\mathcal{D}}_{t}\mathcal{V}\left(  t\right)  w\left(  t\right)  ,
\label{7.49}%
\end{equation}
Then, by (\ref{L1.1}) and (\ref{L1.2})
\begin{equation}
u\left(  t\right)  =M{\mathcal{D}}_{t}\left(  T\left(  \left\vert x\right\vert
\right)  w\left(  x\right)  +R\left(  \left\vert x\right\vert \right)
w\left(  -x\right)  \right)  +O\left(  \frac{\left\vert w\left(  0\right)
\right\vert }{\sqrt{t}}\right)  +O\left(  t^{-\frac{3}{4}}\left\Vert
w\right\Vert _{\mathbf{H}^{1}}\right)  , \label{7.28}%
\end{equation}
as $t\rightarrow\infty,$ uniformly with respect to $x\in\mathbb{R}.$ This
expression shows that in order to estimate the solution in the uniform norm
$\left\Vert u\left(  t\right)  \right\Vert _{\mathbf{L}^{\infty}},$ we need to
control $\left\Vert w\right\Vert _{\mathbf{L}^{\infty}}$ and $\left\Vert
w\right\Vert _{\mathbf{H}^{1}}.$ Therefore, given $u\left(  t\right)  $, we
introduce a new dependent variable by letting $w\left(  t\right)
=\mathcal{FU}\left(  -t\right)  u\left(  t\right)  $ and define the following
space%
\[
\mathbf{X}_{T}=\left\{  u\in\mathbf{C}\left(  \left[  0,T\right]
;\mathbf{H}^{1}\right)  ;\left\Vert u\right\Vert _{\mathbf{X}_{T}}%
<\infty\right\}  ,\text{ \ }T>0,
\]
where
\[
\left\Vert u\right\Vert _{\mathbf{X}_{T}}=\sup_{t\in\left[  0,T\right]
}\left(  \left\Vert w\right\Vert _{\mathbf{L}^{\infty}}+\left\langle
t\right\rangle ^{-\beta}\left\Vert w\right\Vert _{\mathbf{H}^{1}}\right)  ,
\]
with $\beta>0$.

For the local well-posedness of equation (\ref{1.1}) in $\mathbf{H}^{1}$ we
refer to \cite{Cazenave} (see Theorem 4.3.1 and Corollary 4.3.3). We now prove
a local existence result which is adapted to our purposes. We have the following.

\begin{theorem}
\label{Theorem 7.1}Suppose that $V$ satisfies Condition \ref{Cond1}. Let the
initial data $u_{0}\in$ $\mathbf{H}^{1}\cap\mathbf{H}^{0,1}.$ In addition,
suppose that $u_{0}\left(  x\right)  $ are odd if $T\left(  0\right)  =1$ and
$u_{0}\left(  x\right)  $ are even in the case when $T\left(  0\right)  =-1$.
Then, for some $T>0,$ the Cauchy problem for (\ref{1.1}) has a unique solution
$u\in$ $\mathbf{C}\left(  \left[  0,T\right]  ;\mathbf{H}^{1}\cap
\mathbf{H}^{0,1}\right)  ,$ such that
\begin{equation}
\mathcal{U}\left(  -t\right)  u\in\mathbf{C}\left(  \left[  0,T\right]
;\mathbf{H}^{1}\cap\mathbf{H}^{0,1}\right)  \label{7.4}%
\end{equation}
with
\begin{equation}
\left\Vert u\right\Vert _{\mathbf{X}_{T}}+\sup_{t\in\left(  0,T\right)
}\left(  \left\langle t\right\rangle ^{-1-\beta}\left\Vert u\left(  t\right)
\right\Vert _{\mathbf{H}^{0,1}}\right)  +\sup_{t\in\left(  0,T\right)
}\left(  \left\langle t\right\rangle ^{-\beta}\left\Vert u\left(  t\right)
\right\Vert _{\mathbf{H}^{1}}\right)  \leq C\left(  \left\Vert u_{0}%
\right\Vert _{\mathbf{H}^{1}}+\left\Vert u_{0}\right\Vert _{\mathbf{H}^{0,1}%
}\right)  . \label{7.5}%
\end{equation}
Moreover, if $\left\Vert u_{0}\right\Vert _{\mathbf{H}^{1}}+\left\Vert
u_{0}\right\Vert _{\mathbf{H}^{0,1}}\leq\varepsilon$, there is $\varepsilon
_{0}>0,$ such that for every $0<\varepsilon<\varepsilon_{0}$ the existence
time $T>1.$ Furthermore, $u$ can be extended on a maximal existence interval
$[0,T_{\max}),$ to a solution $u\in$ $\mathbf{C}\left(  [0,T_{\max
});\mathbf{H}^{1}\cap\mathbf{H}^{0,1}\right)  ,$ and if $T_{\max}<\infty,$
then
\begin{equation}
\sup_{\tau\in\left(  0,t\right)  }\left(  \left\langle \tau\right\rangle
^{-1-\beta}\left\Vert u\left(  \tau\right)  \right\Vert _{\mathbf{H}^{0,1}%
}\right)  +\sup_{\tau\in\left(  0,t\right)  }\left(  \left\langle
\tau\right\rangle ^{-\beta}\left\Vert u\left(  \tau\right)  \right\Vert
_{\mathbf{H}^{1}}\right)  \rightarrow\infty,\text{ as }t\uparrow T_{\max}.
\label{7.41}%
\end{equation}

\end{theorem}

\begin{proof}
We consider the integral equation corresponding to (\ref{1.1}), that is%
\begin{equation}
u\left(  t\right)  =\mathcal{U}\left(  t\right)  u_{0}-i\lambda\int_{0}%
^{t}\mathcal{U}\left(  t-\tau\right)  \left(  \left\vert u\right\vert
^{2}u\right)  \left(  \tau\right)  d\tau. \label{7.22}%
\end{equation}
Using (\ref{7.8}) we estimate from (\ref{7.22})
\begin{equation}
\left\Vert u\left(  t\right)  \right\Vert _{\mathbf{H}^{1}}\leq\left\Vert
\mathcal{U}\left(  t\right)  u_{0}\right\Vert _{\mathbf{H}^{1}}+\left\vert
\lambda\right\vert \int_{0}^{t}\left\Vert \mathcal{U}\left(  t-\tau\right)
\left(  \left\vert u\right\vert ^{2}u\right)  \left(  \tau\right)  \right\Vert
_{\mathbf{H}^{1}}d\tau\leq K_{0}\left\Vert u_{0}\right\Vert _{\mathbf{H}^{1}%
}+K_{0}\left\vert \lambda\right\vert \int_{0}^{t}\left\Vert \left(  \left\vert
u\right\vert ^{2}u\right)  \left(  \tau\right)  \right\Vert _{\mathbf{H}^{1}%
}d\tau. \label{7.44}%
\end{equation}
Then%
\begin{equation}
\left\Vert u\left(  t\right)  \right\Vert _{\mathbf{H}^{1}}\leq K_{0}%
\left\Vert u_{0}\right\Vert _{\mathbf{H}^{1}}+3K_{0}\left\vert \lambda
\right\vert \int_{0}^{t}\left\Vert u\left(  \tau\right)  \right\Vert
_{\mathbf{L}^{\infty}}^{2}\left\Vert u\left(  \tau\right)  \right\Vert
_{\mathbf{H}^{1}}d\tau. \label{7.46}%
\end{equation}
Moreover, (\ref{7.22}) and (\ref{7.9}) imply
\begin{equation}
\left.
\begin{array}
[c]{c}%
\left\Vert u\left(  t\right)  \right\Vert _{\mathbf{H}^{0,1}}\leq\left\Vert
\mathcal{U}\left(  t\right)  u_{0}\right\Vert _{\mathbf{H}^{0,1}}+\left\vert
\lambda\right\vert \int_{0}^{t}\left\Vert \mathcal{U}\left(  t-\tau\right)
\left(  \left\vert u\right\vert ^{2}u\right)  \left(  \tau\right)  \right\Vert
_{\mathbf{H}^{0,1}}d\tau\\
\leq K_{0}\left(  \left\langle t\right\rangle \left\Vert u_{0}\right\Vert
_{\mathbf{H}^{1}}+\left\Vert u_{0}\right\Vert _{\mathbf{H}^{0,1}}\right)
+K_{0}\left\vert \lambda\right\vert \int_{0}^{t}\left(  \left\langle
t\right\rangle \left\Vert \left(  \left\vert u\right\vert ^{2}u\right)
\left(  \tau\right)  \right\Vert _{\mathbf{H}^{1}}+\left\Vert \left(
\left\vert u\right\vert ^{2}u\right)  \left(  \tau\right)  \right\Vert
_{\mathbf{H}^{0,1}}\right)  d\tau\\
\leq K_{0}\left(  \left\langle t\right\rangle \left\Vert u_{0}\right\Vert
_{\mathbf{H}^{1}}+\left\Vert u_{0}\right\Vert _{\mathbf{H}^{0,1}}\right)
+3K_{0}\left\vert \lambda\right\vert \int_{0}^{t}\left\Vert u\left(
\tau\right)  \right\Vert _{\mathbf{L}^{\infty}}^{2}\left(  \left\langle
t\right\rangle \left\Vert u\left(  \tau\right)  \right\Vert _{\mathbf{H}^{1}%
}+\left\Vert u\left(  \tau\right)  \right\Vert _{\mathbf{H}^{0,1}}\right)
d\tau.
\end{array}
\right.  \label{7.45}%
\end{equation}
The existence of a unique solution $u\in\mathbf{C}\left(  \left[  0,T\right]
;\mathbf{H}^{1}\cap\mathbf{H}^{0,1}\right)  $ follows from a contraction
mapping argument on the space
\[
\left.  \mathcal{E}_{T}\mathcal{=\{}u\in\mathbf{C}\left(  \left[  0,T\right]
;\mathbf{H}^{1}\cap\mathbf{H}^{0,1}\right)  ;\left\Vert u\left(  t\right)
\right\Vert _{L^{\infty}\left(  \left(  0,T\right)  ,\mathbf{H}^{1}\right)
}\leq M\text{ and }\left\Vert u\left(  t\right)  \right\Vert _{L^{\infty
}\left(  \left(  0,T\right)  ,\mathbf{H}^{0,1}\right)  }\leq M\}\right.
\]
where%
\[
M=\left(  K_{0}+1\right)  \left(  \left\langle T\right\rangle \left\Vert
u_{0}\right\Vert _{\mathbf{H}^{1}}+\left\Vert u_{0}\right\Vert _{\mathbf{H}%
^{0,1}}\right)
\]
with $T>0$ satisfying%
\begin{equation}
12\left\vert \lambda\right\vert \left(  K_{0}+1\right)  ^{2}M^{2}\left\langle
T\right\rangle T<1. \label{7.30}%
\end{equation}
The inclusion (\ref{7.4}) is consequence of (\ref{7.8}), (\ref{7.9}) and
$u\in\mathbf{C}\left(  \left[  0,T\right]  ;\mathbf{H}^{1}\cap\mathbf{H}%
^{0,1}\right)  .$ Using (\ref{7.21}) and (\ref{7.9}) we get
\begin{equation}
\left\Vert w\left(  t\right)  \right\Vert _{\mathbf{H}^{1}}=\left\Vert
\mathcal{FU}\left(  -t\right)  u\left(  t\right)  \right\Vert _{\mathbf{H}%
^{1}}\leq K_{0}\left\Vert \mathcal{U}\left(  -t\right)  u\left(  t\right)
\right\Vert _{\mathbf{H}^{0,1}}\leq K_{0}^{2}\left(  \left\langle
t\right\rangle \left\Vert u\left(  t\right)  \right\Vert _{\mathbf{H}^{1}%
}+\left\Vert u\left(  t\right)  \right\Vert _{\mathbf{H}^{0,1}}\right)  .
\label{7.24}%
\end{equation}
As $\left\Vert u\left(  t\right)  \right\Vert _{L^{\infty}\left(  \left(
0,T\right)  ,\mathbf{H}^{1}\right)  }+\left\Vert u\left(  t\right)
\right\Vert _{L^{\infty}\left(  \left(  0,T\right)  ,\mathbf{H}^{0,1}\right)
}\leq2\left(  K_{0}+1\right)  \left(  \left\langle T\right\rangle \left\Vert
u_{0}\right\Vert _{\mathbf{H}^{1}}+\left\Vert u_{0}\right\Vert _{\mathbf{H}%
^{0,1}}\right)  ,$ via Sobolev's inequality we have
\[
\left\Vert u\right\Vert _{\mathbf{X}_{T}}\leq C\left(  \left\Vert
u_{0}\right\Vert _{\mathbf{H}^{1}}+\left\Vert u_{0}\right\Vert _{\mathbf{H}%
^{0,1}}\right)  .
\]
Since $u\in$ $\mathcal{E}_{T},$ $\left\Vert u\left(  t\right)  \right\Vert
_{L^{\infty}\left(  \left(  0,T\right)  ,\mathbf{H}^{1}\right)  }\leq\left(
K_{0}+1\right)  \left\langle T\right\rangle \left\Vert u_{0}\right\Vert
_{\mathbf{H}^{1}}.$ Thus, we deduce (\ref{7.5}). If $\left\Vert u_{0}%
\right\Vert _{\mathbf{H}^{1}}+\left\Vert u_{0}\right\Vert _{\mathbf{H}^{0,1}%
}\leq\varepsilon,$ then from (\ref{7.30}) we conclude that there is
$\varepsilon_{0}>0,$ such that for every $0<\varepsilon<\varepsilon_{0}$ the
existence time $T>1.$ Finally, we prove the blowup alternative (\ref{7.41}).
Assume by contradiction that $T_{\max}<\infty,$ and that there exist $B>0$ and
a sequence $\left(  t_{n}\right)  _{n\geq1}$ such that $t_{n}\uparrow T_{\max
}$ and
\[
\sup_{\tau\in\left(  0,t_{n}\right)  }\left(  \left\langle \tau\right\rangle
^{-1-\beta}\left\Vert u\left(  \tau\right)  \right\Vert _{\mathbf{H}^{0,1}%
}\right)  +\sup_{\tau\in\left(  0,t_{n}\right)  }\left(  \left\langle
\tau\right\rangle ^{-\beta}\left\Vert u\left(  \tau\right)  \right\Vert
_{\mathbf{H}^{1}}\right)  \leq B.
\]
We consider the space $\mathcal{E}_{T}$ with
\[
M=\left(  2K_{0}\left\langle T_{\max}\right\rangle ^{2+\beta}+1\right)  B.
\]
We fix $T_{\max}<\tau_{1}<\infty$, then we choose $0<T<\tau_{1}-T_{\max}$
sufficiently small so that%
\[
12K_{0}\left\vert \lambda\right\vert M^{2}\left(  2K_{0}\left\langle T_{\max
}\right\rangle ^{2+\beta}+1\right)  \left\langle T\right\rangle T<1.
\]
Then, by a contraction mapping argument, for all $n\geq1$ there exists
$u_{n}\in\mathbf{C}\left(  \left[  0,T\right]  ;\mathbf{H}^{1}\cap
\mathbf{H}^{0,1}\right)  $ which is a solution to the equation%
\[
u_{n}\left(  t\right)  =\mathcal{U}\left(  t\right)  u\left(  t_{n}\right)
-i\lambda\int_{0}^{t}\mathcal{U}\left(  t-\tau\right)  \left(  \left\vert
u_{n}\right\vert ^{2}u_{n}\right)  \left(  \tau\right)  d\tau.
\]
Setting now
\[
v_{n}\left(  t\right)  =\left\{
\begin{array}
[c]{c}%
u\left(  t\right)  \text{ \ \ }0\leq t\leq t_{n}\\
u_{n}\left(  t-t_{n}\right)  \text{ }t_{n}\leq t\leq t_{n}+T
\end{array}
\right.
\]
we see that $v_{n}\in\mathbf{C}\left(  \left[  0,t_{n}+T\right]
;\mathbf{H}^{1}\cap\mathbf{H}^{0,1}\right)  $ is a solution of (\ref{1.1}) on
$\left[  0,t_{n}+T\right]  .$ Since $t_{n}+T>T_{\max,},$ for $n$ big enough,
this yields a contradiction. Hence, (\ref{7.41}) holds.
\end{proof}

We now prove an a priori estimate for the solutions of the Cauchy problem
(\ref{1.1}). Let us define
\[
K=13\left(  K_{0}+1\right)  ^{3},
\]
where $K_{0}>0$ is given by Lemma \ref{L3} and \ref{L4}. We have the following.

\begin{lemma}
\label{Lemma 7.1}Suppose that $V$ satisfies Condition \ref{Cond1}. Let the
initial data $u_{0}\in$ $\mathbf{H}^{1}\cap\mathbf{H}^{0,1}$. In addition, we
suppose that $u_{0}\left(  x\right)  $ are odd if $T\left(  0\right)  =1$ and
$u_{0}\left(  x\right)  $ are even in the case when $T\left(  0\right)  =-1$.
Let $T>0$ be given by Theorem \ref{Theorem 7.1}. Then, there exists
$\varepsilon_{0}>0$ such that for any $0<\varepsilon<\varepsilon_{0}$ and
$\left\Vert u_{0}\right\Vert _{\mathbf{H}^{1}}+\left\Vert u_{0}\right\Vert
_{\mathbf{H}^{0,1}}\leq\varepsilon,$ the estimate
\begin{equation}
\left\Vert u\right\Vert _{\mathbf{X}_{T}}+\sup_{\tau\in\left(  0,T\right)
}\left(  \left\langle \tau\right\rangle ^{-1-\beta}\left\Vert u\left(
\tau\right)  \right\Vert _{\mathbf{H}^{0,1}}\right)  +\sup_{\tau\in\left(
0,T\right)  }\left(  \left\langle \tau\right\rangle ^{-\beta}\left\Vert
u\left(  \tau\right)  \right\Vert _{\mathbf{H}^{1}}\right)  <K\varepsilon
\label{7.23}%
\end{equation}
holds for the solutions$\ u$ of the Cauchy problem (\ref{1.1}).
\end{lemma}

\begin{proof}
We prove (\ref{7.23}) by a contradiction argument. Suppose that the statement
of the lemma is not true. Similarly to (\ref{7.24}) from inclusion $u\in$
$\mathbf{C}\left(  \left[  0,T\right]  ;\mathbf{H}^{1}\cap\mathbf{H}%
^{0,1}\right)  $ and (\ref{7.4}) it follows that the left-hand side in
(\ref{7.23}) is continuous as function of $T>0.$ Then, for any $\varepsilon>0$
and initial data $\left\Vert u_{0}\right\Vert _{\mathbf{H}^{1}}+\left\Vert
u_{0}\right\Vert _{\mathbf{H}^{0,1}}\leq\varepsilon,$ we can find $T_{1}\leq
T,$ with the property
\begin{equation}
\left\Vert u\right\Vert _{\mathbf{X}_{T_{1}}}+\sup_{\tau\in\left(
0,T_{1}\right)  }\left(  \left\langle \tau\right\rangle ^{-1-\beta}\left\Vert
u\left(  \tau\right)  \right\Vert _{\mathbf{H}^{0,1}}\right)  +\sup_{\tau
\in\left(  0,T_{1}\right)  }\left(  \left\langle \tau\right\rangle ^{-\beta
}\left\Vert u\left(  \tau\right)  \right\Vert _{\mathbf{H}^{1}}\right)
=K\varepsilon.\label{7.25}%
\end{equation}
Let us show that this yields a contradiction. Theorem \ref{Theorem 7.1}
guarantees that there is $\varepsilon_{0}>0$ such that for any $0<\varepsilon
\leq\varepsilon_{0}$ and $\left\Vert u_{0}\right\Vert _{\mathbf{H}^{1}%
}+\left\Vert u_{0}\right\Vert _{\mathbf{H}^{0,1}}\leq\varepsilon$ the
existence time $T>1.$ In particular, as $u\in$ $\mathcal{E}_{1},$ it follows
that
\begin{equation}
\sup_{t\in\left(  0,1\right)  }\left\Vert u\left(  t\right)  \right\Vert
_{\mathbf{H}^{1}}\leq\left(  K_{0}+1\right)  \left(  \sqrt{2}\left\Vert
u_{0}\right\Vert _{\mathbf{H}^{1}}+\left\Vert u_{0}\right\Vert _{\mathbf{H}%
^{0,1}}\right)  \leq\sqrt{2}\left(  K_{0}+1\right)  \varepsilon\label{7.51}%
\end{equation}
and%
\[
\sup_{t\in\left(  0,1\right)  }\left\Vert u\left(  t\right)  \right\Vert
_{\mathbf{H}^{0,1}}\leq\sqrt{2}\left(  K_{0}+1\right)  \varepsilon.
\]
Then, using (\ref{7.24}) we have%
\begin{equation}
\sup_{t\in\left(  0,1\right)  }\left\Vert w\left(  t\right)  \right\Vert
_{\mathbf{H}^{1}}\leq\sqrt{2}K_{0}^{2}\left(  \sup_{t\in\left(  0,1\right)
}\left\Vert u\left(  t\right)  \right\Vert _{\mathbf{H}^{1}}+\sup_{t\in\left(
0,1\right)  }\left\Vert u\left(  t\right)  \right\Vert _{\mathbf{H}^{0,1}%
}\right)  \leq4K_{0}^{2}\left(  K_{0}+1\right)  \varepsilon.\label{7.33}%
\end{equation}
Moreover,
\begin{equation}
\sup_{t\in\left(  0,1\right)  }\left\Vert w\left(  t\right)  \right\Vert
_{\mathbf{L}^{\infty}}\leq\sqrt{2}\sup_{t\in\left(  0,1\right)  }\left\Vert
w\left(  t\right)  \right\Vert _{\mathbf{H}^{1}}\leq4\sqrt{2}K_{0}^{2}\left(
K_{0}+1\right)  \varepsilon.\label{7.34}%
\end{equation}
Hence,
\begin{equation}
\left\Vert u\right\Vert _{\mathbf{X}_{1}}\leq\sup_{t\in\left(  0,1\right)
}\left\Vert w\right\Vert _{\mathbf{L}^{\infty}}+\sup_{t\in\left(  0,1\right)
}\left\Vert w\right\Vert _{\mathbf{H}^{1}}\leq4\left(  1+\sqrt{2}\right)
K_{0}^{2}\left(  K_{0}+1\right)  \varepsilon.\label{7.52}%
\end{equation}
Suppose now that $t\geq1.$ We represent the solution $u\left(  t\right)
=\mathcal{U}\left(  t\right)  \mathcal{F}^{-1}w=M{\mathcal{D}}_{t}%
\mathcal{V}\left(  t\right)  w\left(  t\right)  \ $and use (\ref{7.28}) to
reduce the problem to estimate $w.$ As mentioned at the beginning of this
Section, if the potential is symmetric and the\ initial data $u_{0}\left(
x\right)  $ are even or odd, then $u\left(  t,x\right)  $ is even or odd,
respectively, for all $t\in\lbrack0,T].$ Then, since $\mathcal{F}$ preserves
the parity,\ by the assumptions of Theorem \ref{Theorem 1.1}, $w\left(
k\right)  $ is odd if $T\left(  0\right)  =1,$ and $w\left(  k\right)  $ is
even if $T\left(  0\right)  =-1.$ To derive the equation for $w\left(
t\right)  $ we apply the operator $\mathcal{FU}\left(  -t\right)  $ to
equation (\ref{1.1}). We find%
\[
i\partial_{t}\left(  \mathcal{FU}\left(  -t\right)  u\right)  =\lambda
\mathcal{FU}\left(  -t\right)  \left(  \left\vert u\right\vert ^{2}u\right)  .
\]
Hence, substituting $u\left(  t\right)  =\mathcal{U}\left(  t\right)
\mathcal{F}^{-1}w$ we obtain the equation for $w,$
\[
i\partial_{t}w=\lambda\mathcal{FU}\left(  -t\right)  \left(  \left\vert
\mathcal{U}\left(  t\right)  \mathcal{F}^{-1}w\right\vert ^{2}\mathcal{U}%
\left(  t\right)  \mathcal{F}^{-1}w\right)  .
\]
By virtue of operators $\mathcal{V}\left(  t\right)  $ and $\mathcal{V}%
^{-1}\left(  t\right)  $ we find
\begin{align*}
&  \lambda\mathcal{FU}\left(  -t\right)  \left(  \left\vert \mathcal{U}\left(
t\right)  \mathcal{F}^{-1}w\right\vert ^{2}\mathcal{U}\left(  t\right)
\mathcal{F}^{-1}w\right)  \\
&  =\lambda\mathcal{V}^{-1}\left(  t\right)  {\mathcal{D}}_{t}^{-1}%
\overline{M}\left(  \left\vert M{\mathcal{D}}_{t}\mathcal{V}\left(  t\right)
w\right\vert ^{2}M{\mathcal{D}}_{t}\mathcal{V}\left(  t\right)  w\right)  \\
&  =\lambda\mathcal{V}^{-1}\left(  t\right)  {\mathcal{D}}_{t}^{-1}\left(
\left\vert {\mathcal{D}}_{t}\mathcal{V}\left(  t\right)  w\right\vert
^{2}{\mathcal{D}}_{t}\mathcal{V}\left(  t\right)  w\right)  =\lambda
t^{-1}\mathcal{V}^{-1}\left(  t\right)  \left(  \left\vert \mathcal{V}\left(
t\right)  w\right\vert ^{2}\mathcal{V}\left(  t\right)  w\right)  .
\end{align*}
Thus we arrive to the equation
\begin{equation}
i\partial_{t}w=\lambda t^{-1}\mathcal{V}^{-1}\left(  t\right)  \left(
\left\vert \mathcal{V}\left(  t\right)  w\right\vert ^{2}\mathcal{V}\left(
t\right)  w\right)  .\label{7.1}%
\end{equation}
Applying Lemma \ref{Lemma 6.1} we get
\begin{equation}
\left.  i\partial_{t}w=\lambda t^{-1}\left\vert w\left(  k\right)  \right\vert
^{2}w\left(  k\right)  +O\left(  t^{-1}\left(  \left\vert w\left(  0\right)
\right\vert +t^{-\frac{1}{4}}\left\Vert w\right\Vert _{\mathbf{H}^{1}}\right)
\left(  \left\Vert w\right\Vert _{\mathbf{L}^{\infty}}+t^{-\beta}\left\Vert
w\right\Vert _{\mathbf{H}^{1}}\right)  ^{2}\right)  +O\left(  t^{\beta
-\frac{5}{4}}\left(  \left\Vert w\right\Vert _{\mathbf{L}^{\infty}}+t^{-\beta
}\left\Vert w\right\Vert _{\mathbf{H}^{1}}\right)  ^{3}\right)  .\right.
\label{7.2}%
\end{equation}
We claim that
\begin{equation}
w\left(  t,0\right)  =0.\label{7.40}%
\end{equation}
Indeed, we have $w\left(  t,0\right)  =\left(  \mathcal{F}u\left(  t\right)
\right)  \left(  0\right)  .$ Since $u\left(  t,x\right)  $ is odd if
$T\left(  0\right)  =1$ and it is even when $T\left(  0\right)  =-1$,
(\ref{7.40}) follows from (\ref{7.50}). Going back to (\ref{7.2}) and using
(\ref{7.25}) we obtain%
\begin{equation}
i\partial_{t}w=\lambda t^{-1}\left\vert w\left(  k\right)  \right\vert
^{2}w\left(  k\right)  +O\left(  \varepsilon^{3}t^{\beta-\frac{5}{4}}\right)
.\label{7.32}%
\end{equation}
Multiplying (\ref{7.32}) by $\overline{w\left(  k\right)  }$, taking the
imaginary part and using (\ref{7.25}) we find
\begin{equation}
\partial_{t}\left\vert w\left(  t,k\right)  \right\vert ^{2}=O\left(
\varepsilon^{3}\left\langle t\right\rangle ^{\beta-\frac{5}{4}}\left\Vert
w\right\Vert _{\mathbf{L}^{\infty}}\right)  =O\left(  \varepsilon
^{3}\left\langle t\right\rangle ^{\beta-\frac{5}{4}}\right)  .\label{7.3}%
\end{equation}
By (\ref{7.34})
\[
\left\Vert w\left(  1\right)  \right\Vert _{\mathbf{L}^{\infty}}\leq4\sqrt
{2}K_{0}^{2}\left(  K_{0}+1\right)  \varepsilon.
\]
Then, equation (\ref{7.3}) shows that
\begin{equation}
\sup_{t\in\left(  1,T_{1}\right)  }\left\Vert w\left(  t\right)  \right\Vert
_{\mathbf{L}^{\infty}}\leq4\sqrt{2}K_{0}^{2}\left(  K_{0}+1\right)
\varepsilon+C\varepsilon^{3/2}.\label{7.37}%
\end{equation}
To estimate the derivative $\left\Vert \partial_{k}w\right\Vert _{\mathbf{L}%
^{2}},$ we differentiate equation (\ref{7.1}). We get
\[
i\partial_{t}w_{k}=\lambda t^{-1}\partial_{k}\mathcal{V}^{-1}\left(  t\right)
\left(  \left\vert \mathcal{V}\left(  t\right)  w\right\vert ^{2}%
\mathcal{V}\left(  t\right)  w\right)  .
\]
Taking the $\mathbf{L}^{2}$-norm in the last equation, using the estimate of
Lemma \ref{Lemma 6.2} and relation (\ref{7.25}) we get
\[
\frac{d}{dt}\left\Vert \partial_{k}w\right\Vert _{\mathbf{L}^{2}}\leq
Ct^{-1}\log\left\langle t\right\rangle \left\Vert w\right\Vert _{\mathbf{L}%
^{\infty}}^{3}+Ct^{-\frac{5}{4}}\left\Vert w\right\Vert _{\mathbf{H}^{1}}%
^{3}+Ct^{-1}\left\Vert w\right\Vert _{\mathbf{L}^{\infty}}^{2}\left\Vert
w\right\Vert _{\mathbf{H}^{1}}\leq C\varepsilon^{3}t^{\beta-1},
\]
for $0<\beta\leq\frac{1}{8}.$ Integrating the last inequality we obtain%
\[
\left\Vert \partial_{k}w\right\Vert _{\mathbf{L}^{2}}\leq\left\Vert
\partial_{k}w\left(  1\right)  \right\Vert _{\mathbf{L}^{2}}+C\varepsilon
^{3}\left\langle t\right\rangle ^{\beta}.
\]
By (\ref{7.33}) $\left\Vert \partial_{k}w\left(  1\right)  \right\Vert
_{\mathbf{L}^{2}}\leq4K_{0}^{2}\left(  K_{0}+1\right)  \varepsilon.$ Then,
\begin{equation}
\left\Vert \partial_{k}w\right\Vert _{\mathbf{L}^{2}}\leq\left\langle
t\right\rangle ^{\beta}\left(  4K_{0}^{2}\left(  K_{0}+1\right)
\varepsilon+C\varepsilon^{3}\right)  .\label{7.35}%
\end{equation}
Moreover, as $\left\Vert u\left(  t\right)  \right\Vert _{\mathbf{L}^{2}%
}=\left\Vert u_{0}\right\Vert _{\mathbf{L}^{2}},$ for all $t\in\lbrack0,T],$
and $\mathcal{F}$ and $\mathcal{U}\left(  -t\right)  $ are unitary on
$\mathbf{L}^{2}$, we get%
\begin{equation}
\left\Vert w\right\Vert _{\mathbf{L}^{2}}=\left\Vert \mathcal{FU}\left(
-t\right)  u\left(  t\right)  \right\Vert _{\mathbf{L}^{2}}=\left\Vert
u\left(  t\right)  \right\Vert _{\mathbf{L}^{2}}=\left\Vert u_{0}\right\Vert
_{\mathbf{L}^{2}}\leq\varepsilon\leq\varepsilon\left\langle t\right\rangle
^{\beta}.\label{7.36}%
\end{equation}
(\ref{7.35}) and (\ref{7.36}) imply
\begin{equation}
\sup_{t\in\left(  1,T_{1}\right)  }\left(  \left\langle t\right\rangle
^{-\beta}\left\Vert w\left(  t\right)  \right\Vert _{\mathbf{H}^{1}}\right)
\leq4K_{0}^{2}\left(  K_{0}+1\right)  \varepsilon+\varepsilon+C\varepsilon
^{3}.\label{7.38}%
\end{equation}
Hence, by (\ref{7.37}) and (\ref{7.38})
\begin{equation}
\sup_{t\in\left(  1,T_{1}\right)  }\left(  \left\Vert w\right\Vert
_{\mathbf{L}^{\infty}}+\left\langle t\right\rangle ^{-\beta}\left\Vert
w\right\Vert _{\mathbf{H}^{1}}\right)  \leq\left(  4+4\sqrt{2}\right)
K_{0}^{2}\left(  K_{0}+1\right)  \varepsilon+\varepsilon+C\varepsilon
^{3/2}\leq\left(  5+4\sqrt{2}\right)  \left(  K_{0}+1\right)  ^{3}%
\varepsilon+C\varepsilon^{3/2}.\label{7.47}%
\end{equation}
From (\ref{7.52}) and (\ref{7.47}) we deduce
\begin{equation}
\left\Vert u\right\Vert _{\mathbf{X}_{T_{1}}}\leq\left(  5+4\sqrt{2}\right)
\left(  K_{0}+1\right)  ^{3}\varepsilon+C\varepsilon^{3/2}.\label{7.55}%
\end{equation}
Using (\ref{3.4}), (\ref{7.28}), (\ref{7.25}) we estimate
\begin{equation}
\left\Vert u\left(  t\right)  \right\Vert _{\mathbf{L}^{\infty}}\leq
C\varepsilon\left\langle t\right\rangle ^{-1/2},\text{ \ }t\in\lbrack
0,T_{1}].\label{7.42}%
\end{equation}
Then, from (\ref{7.46}) and (\ref{7.25})
\[
\left\Vert u\left(  t\right)  \right\Vert _{\mathbf{H}^{1}}\leq K_{0}%
\varepsilon+C\varepsilon^{3}\int_{0}^{t}\left\langle t\right\rangle
^{-1+\beta}d\tau\leq K_{0}\varepsilon+C\varepsilon^{3}\left\langle
t\right\rangle ^{\beta}.
\]
Thus,%
\begin{equation}
\sup_{\tau\in\left(  0,T_{1}\right)  }\left(  \left\langle \tau\right\rangle
^{-\beta}\left\Vert u\left(  \tau\right)  \right\Vert _{\mathbf{H}^{1}%
}\right)  \leq K_{0}\varepsilon+C\varepsilon^{3}.\label{7.48}%
\end{equation}
Similarly, from (\ref{7.45}), (\ref{7.42}), (\ref{7.25}) we get%
\begin{equation}
\left.  \left\Vert u\left(  t\right)  \right\Vert _{\mathbf{H}^{0,1}}\leq
K_{0}\left\langle t\right\rangle \varepsilon+C\varepsilon^{3}\left\langle
t\right\rangle \int_{0}^{t}\left\langle \tau\right\rangle ^{-1+\beta}d\tau
\leq\left(  K_{0}\varepsilon+C\varepsilon^{3}\right)  \left\langle
t\right\rangle ^{1+\beta},\right.  \label{7.56}%
\end{equation}
for all $t\in\lbrack0,T_{1}].$ By (\ref{7.55}), (\ref{7.48}) and (\ref{7.56})
we see that%
\begin{equation}
\left\Vert u\right\Vert _{\mathbf{X}_{T_{1}}}+\sup_{\tau\in\left(
0,T_{1}\right)  }\left(  \left\langle \tau\right\rangle ^{-1-\beta}\left\Vert
u\left(  \tau\right)  \right\Vert _{\mathbf{H}^{0,1}}\right)  +\sup_{\tau
\in\left(  0,T_{1}\right)  }\left(  \left\langle \tau\right\rangle ^{-\beta
}\left\Vert u\left(  \tau\right)  \right\Vert _{\mathbf{H}^{1}}\right)
\leq\left(  7+4\sqrt{2}\right)  \left(  K_{0}+1\right)  ^{3}\varepsilon
+C\varepsilon^{3/2}<K\varepsilon,\label{7.54}%
\end{equation}
for all $\varepsilon>0$ small enough. This contradicts (\ref{7.25}).
Therefore, there is $\varepsilon_{1}>0,$ such that for all $0<\varepsilon
\leq\varepsilon_{1},$ the estimate (\ref{7.23}) is valid.
\end{proof}

\begin{proof}
[Proof of Theorem \ref{Theorem 1.1}]From Theorem \ref{Theorem 7.1} and Lemma
\ref{Lemma 7.1} it follows that there is a global solution $u\in$
$\mathbf{C}\left(  [0,\infty);\mathbf{H}^{1}\right)  $ to (\ref{1.1})
satisfying the estimate
\begin{equation}
\left\Vert u\right\Vert _{\mathbf{X}_{\infty}}<K\varepsilon. \label{7.39}%
\end{equation}
In particular, from (\ref{7.28}) via (\ref{3.4}) we get the uniform estimate
(\ref{1.3}). We now study the asymptotics of $w\left(  t,k\right)  $ as
$t\rightarrow\infty$. Integrating (\ref{7.3}) in time we get
\[
\left\vert w\left(  t,k\right)  \right\vert ^{2}-\left\vert w\left(
s,k\right)  \right\vert ^{2}=O\left(  \varepsilon^{3}s^{\beta-\frac{1}{4}%
}\right)
\]
for all $t>s>1.$ The last estimate shows that $\left\vert w\left(  t,k\right)
\right\vert $ is a Cauchy sequence. Then, there exists a limit $\Xi
\in\mathbf{L}^{\infty}$ with the property
\[
\Xi\left(  k\right)  -\left\vert w\left(  t,k\right)  \right\vert
^{2}=O\left(  \varepsilon^{3}t^{\beta-\frac{1}{4}}\right)  .
\]
Since $w\left(  t,k\right)  $ is even or odd, $\Xi\left(  k\right)  $ is an
even function. Next, we rewrite equation (\ref{7.32}) as%
\begin{align*}
i\partial_{t}w  &  =\lambda t^{-1}\Xi\left(  k\right)  w\left(  k\right)
+O\left(  \varepsilon^{3}\left\Vert w\left(  k\right)  \right\Vert
_{\mathbf{L}^{\infty}}t^{\beta-\frac{5}{4}}\right)  +O\left(  \varepsilon
^{3}t^{\beta-\frac{5}{4}}\right) \\
&  =\lambda t^{-1}\Xi\left(  k\right)  w\left(  k\right)  +O\left(
\varepsilon^{3}t^{\beta-\frac{5}{4}}\right)  ,
\end{align*}
where in the last equality we used (\ref{7.39}) to control $\left\Vert
w\left(  k\right)  \right\Vert _{\mathbf{L}^{\infty}}.$ Putting $w\left(
t,k\right)  =v\left(  t,k\right)  e^{-i\lambda\Xi\left(  k\right)  \log t},$
we exclude the resonant nonlinear term $\lambda t^{-1}\Xi\left(  k\right)
w\left(  k\right)  ,$ and we are left with the equation for $v$
\[
i\partial_{t}v=O\left(  \varepsilon^{3}t^{\beta-\frac{5}{4}}\right)  .
\]
Integrating in time the last equation we get
\[
v\left(  t,k\right)  -v\left(  s,k\right)  =O\left(  \varepsilon^{3}%
s^{\beta-\frac{1}{4}}\right)
\]
for all $t>s>1.$ Thus $v\left(  t,k\right)  $ is a Cauchy sequence, and hence,
there exists $v_{+}\in\mathbf{L}^{\infty}$ such that $v_{+}\left(  k\right)
-v\left(  t,k\right)  =O\left(  \varepsilon^{3}t^{\beta-\frac{1}{4}}\right)
.$ Therefore, we get the asymptotics for $w$
\[
w\left(  t,k\right)  =v_{+}\left(  k\right)  e^{-i\lambda\Xi\left(  k\right)
\log t}+O\left(  \varepsilon t^{\beta-\frac{1}{4}}\right)  ,
\]
as $t\rightarrow\infty,$ uniformly on $k\in\mathbb{R}$. Moreover, noting that
$\Xi\left(  k\right)  =\lim_{t\rightarrow\infty}\left\vert w\left(
t,k\right)  \right\vert ^{2}=\lim_{t\rightarrow\infty}\left\vert v\left(
t,k\right)  \right\vert ^{2}=\left\vert v_{+}\right\vert ^{2},$ we obtain%
\[
w\left(  t,k\right)  =v_{+}\left(  k\right)  e^{-i\lambda\left\vert
v_{+}\right\vert ^{2}\log t}+O\left(  \varepsilon t^{\beta-\frac{1}{4}%
}\right)  ,
\]
Since $\Xi\left(  k\right)  $ is even, $v$ has the same parity as $w,$ and
thus $v_{+}$ is either even or odd. We now conclude as follows. Relations
(\ref{7.28}), (\ref{7.40}) and (\ref{7.39}) imply%
\[
u\left(  t\right)  =M{\mathcal{D}}_{t}\left(  T\left(  \left\vert x\right\vert
\right)  w\left(  x\right)  +R\left(  \left\vert x\right\vert \right)
w\left(  -x\right)  \right)  +O\left(  \varepsilon t^{\beta-\frac{3}{4}%
}\right)  ,
\]
as $t\rightarrow\infty,$ uniformly with respect to $x\in\mathbb{R}.$ Hence
\[
u\left(  t\right)  =M{\mathcal{D}}_{t}w_{+}e^{-i\lambda\left\vert
v_{+}\right\vert ^{2}\log t}+O\left(  \varepsilon t^{\beta-\frac{3}{4}%
}\right)  ,
\]
with $w_{+}=\left(  T\left(  \left\vert x\right\vert \right)  v_{+}\left(
x\right)  +R\left(  \left\vert x\right\vert \right)  v_{+}\left(  -x\right)
\right)  $. Also we note that by (\ref{3.4}) and (\ref{5.24}) $\left\vert
v_{+}\right\vert ^{2}=\left\vert w_{+}\right\vert ^{2}.$ Therefore, the
asymptotic formula (\ref{1.2}) follows. Theorem \ref{Theorem 1.1} is proved.
\end{proof}

\section{\label{S3}Jost solutions.}

In this Section we expose some properties and estimates for the Jost solutions
that are involved in the proof of the main result. Assume that $V\in
\mathbf{L}^{1,1}.$ Let
\[
m_{\pm}\left(  x,k\right)  =e^{\mp ikx}f_{\pm}\left(  x,k\right)  .
\]
By Lemma 1 of \cite{DeiftTrubowitz} (see page 130), the functions $m_{\pm
}\left(  x,k\right)  $ are the unique solutions of the Volterra integral
equations%
\begin{equation}
m_{+}\left(  x,k\right)  =1+\int_{x}^{\infty}K\left(  y-x,k\right)  V\left(
y\right)  m_{+}\left(  y,k\right)  dy \label{3.24}%
\end{equation}
and
\[
m_{-}\left(  x,k\right)  =1+\int_{-\infty}^{x}K\left(  x-y,k\right)  V\left(
y\right)  m_{-}\left(  y,k\right)  dy,
\]
respectively, where $K\left(  x,k\right)  =\int_{0}^{x}e^{2ikz}dz.$ We need
the following estimates (see \cite{DeiftTrubowitz}, \cite{Ivan}).

\begin{proposition}
\label{Proposition 3.1} Suppose that $V\in\mathbf{L}^{1,1}\left(
\mathbb{R}\right)  .\ $Then,
\begin{equation}
\left\vert m_{\pm}\left(  x,k\right)  -1\right\vert \leq C\left\langle
k\right\rangle ^{-1}\left\langle x\right\rangle ,\text{ for all }%
x,k\in\mathbb{R}\text{.} \label{2.44}%
\end{equation}
If $V\in\mathbf{L}^{1,2+\delta}\left(  \mathbb{R}\right)  ,$ for some
$\delta\geq0,$ then, for all $k\in\mathbb{R}$ and $\pm x\geq0,$ we have%
\begin{equation}
\left\vert m_{\pm}\left(  x,k\right)  -1\right\vert \leq C\left\langle
k\right\rangle ^{-1}\left\langle x\right\rangle ^{-1-\delta}, \label{3.2}%
\end{equation}%
\begin{equation}
\left\vert \partial_{k}m_{\pm}\left(  x,k\right)  \right\vert \leq
C\left\langle k\right\rangle ^{-1}\left\langle x\right\rangle ^{-\delta}.
\label{3.5}%
\end{equation}

\end{proposition}

We also need the following.

\begin{proposition}
Assume that $V\in\mathbf{L}^{1,2+\delta}\left(  \mathbb{R}\right)  .$ In
addition, suppose that there is a partition (\ref{3.23}) such that each part
$\left\langle \cdot\right\rangle ^{2+\delta}V_{j}\in W^{1,1}\left(
I_{j}\right)  .$ Then, for all $k\in\mathbb{R}$ and $\pm x\geq0,$ estimates%
\begin{equation}
\left\vert \partial_{x}m_{\pm}\left(  x,k\right)  \right\vert \leq
C\left\langle k\right\rangle ^{-1}\left\langle x\right\rangle ^{-2-\delta
},\text{ \ } \label{3.6}%
\end{equation}
and%
\begin{equation}
\left\vert \partial_{k}\partial_{x}m_{\pm}\left(  x,k\right)  \right\vert \leq
C\left\langle k\right\rangle ^{-1}\left\langle x\right\rangle ^{-1-\delta}.
\label{3.7}%
\end{equation}
are valid.
\end{proposition}

\begin{proof}
First we prove (\ref{3.6}) for the upper sign. The other case is considered
similarly. From (\ref{3.24}) we obtain%
\begin{equation}
\partial_{x}m_{+}\left(  x,k\right)  =-\int_{x}^{\infty}e^{2ik\left(
y-x\right)  }V\left(  y\right)  m_{+}\left(  y,k\right)  dy. \label{3.25}%
\end{equation}
For $\left\vert k\right\vert \leq1,$ by using (\ref{3.2}) we get
\begin{equation}
\left\vert \partial_{x}m_{+}\left(  x,k\right)  \right\vert \leq C\left\langle
x\right\rangle ^{-2-\delta}\int_{x}^{\infty}\left\langle y\right\rangle
^{2+\delta}\left\vert V\left(  y\right)  \right\vert dy\leq C\left\langle
k\right\rangle ^{-1}\left\langle x\right\rangle ^{-2-\delta}\int_{x}^{\infty
}\left\langle y\right\rangle ^{2+\delta}\left\vert V\left(  y\right)
\right\vert dy. \label{3.29}%
\end{equation}
and hence (\ref{3.6}) follows. Let now $\left\vert k\right\vert \geq1.$ From
(\ref{3.25}) we have%
\begin{equation}
\partial_{x}m_{+}\left(  x,k\right)  =-\int_{x}^{\infty}e^{2ik\left(
y-x\right)  }V\left(  y\right)  dy-\int_{x}^{\infty}e^{2ik\left(  y-x\right)
}V\left(  y\right)  \left(  m_{+}\left(  y,k\right)  -1\right)  dy.
\label{3.26}%
\end{equation}
Using (\ref{3.2}) we estimate
\begin{equation}
\left\vert \int_{x}^{\infty}e^{2ik\left(  y-x\right)  }V\left(  y\right)
\left(  m_{+}\left(  y,k\right)  -1\right)  dy\right\vert \leq C\left\langle
k\right\rangle ^{-1}\left\langle x\right\rangle ^{-2-\delta}\int_{x}^{\infty
}\left\langle y\right\rangle \left\vert V\left(  y\right)  \right\vert dy\leq
C\left\langle k\right\rangle ^{-1}\left\langle x\right\rangle ^{-2-\delta}.
\label{3.27}%
\end{equation}
To control the first term in the right-hand side of (\ref{3.26}) we use
(\ref{3.23}). Suppose that $x\in I_{l},$ for some $l=1,...,N+1.$ Then,
\[
\int_{x}^{\infty}e^{2ik\left(  y-x\right)  }V\left(  y\right)  dy=\int
_{x}^{x_{l}}e^{2ik\left(  y-x\right)  }V_{l}\left(  y\right)  dy+\sum
_{j=l+1}^{N+1}\int_{x_{j-1}}^{x_{j}}e^{2ik\left(  y-x\right)  }V_{j}\left(
y\right)  dy.
\]
Integrating by parts we get
\begin{align*}
\int_{x}^{\infty}e^{2ik\left(  y-x\right)  }V\left(  y\right)  dy  &
=\frac{1}{2ik}\left(  \left.  e^{2ik\left(  y-x\right)  }V_{l}\left(
y\right)  \right\vert _{x}^{x_{l}}-\int_{x}^{x_{l}}e^{2ik\left(  y-x\right)
}\partial_{y}V_{l}\left(  y\right)  dy\right. \\
&  \left.  +\sum_{j=l+1}^{N+1}\left(  \left.  e^{2ik\left(  y-x\right)  }%
V_{j}\left(  y\right)  \right\vert _{x_{j-1}}^{x_{j}}-\int_{x_{j-1}}^{x_{j}%
}e^{2ik\left(  y-x\right)  }\partial_{y}V_{j}\left(  y\right)  dy\right)
\right)  .
\end{align*}
Hence, by Sobolev embedding theorem,
\begin{equation}
\left.
\begin{array}
[c]{c}%
\left\vert \int_{x}^{\infty}e^{2ik\left(  y-x\right)  }V\left(  y\right)
dy\right\vert \leq C\left\langle k\right\rangle ^{-1}\left\langle
x\right\rangle ^{-a}%
{\displaystyle\sum_{j=l}^{N+1}}
\left\Vert \left\langle \cdot\right\rangle ^{a}V_{j}\right\Vert _{\mathbf{L}%
^{\infty}\left(  I_{j}\right)  }+C\left\langle k\right\rangle ^{-1}%
\left\langle x\right\rangle ^{-a}%
{\displaystyle\sum_{j=l}^{N+1}}
\left\Vert \left\langle \cdot\right\rangle ^{a}V_{j}\right\Vert _{W^{1,1}%
\left(  I_{j}\right)  }\\
\leq C\left\langle k\right\rangle ^{-1}\left\langle x\right\rangle ^{-a}%
{\displaystyle\sum_{j=l}^{N+1}}
\left\Vert \left\langle \cdot\right\rangle ^{a}V_{j}\right\Vert _{W^{1,1}%
\left(  I_{j}\right)  }\leq C\left\langle k\right\rangle ^{-1}\left\langle
x\right\rangle ^{-a},
\end{array}
\right.  \label{3.28}%
\end{equation}
for $0\leq a\leq2+\delta.$ Using (\ref{3.27}) and (\ref{3.28}) with
$a=2+\delta$ in (\ref{3.26}) we get (\ref{3.6}) for $\left\vert k\right\vert
\geq1.$

Let us prove (\ref{3.7}) for the upper sign. Differentiating (\ref{3.25}) with
respect to $k$ we have%
\begin{equation}
\left.
\begin{array}
[c]{c}%
\partial_{k}\partial_{x}m_{+}\left(  x,k\right)  =-2i%
{\displaystyle\int\limits_{x}^{\infty}}
e^{2ik\left(  y-x\right)  }\left(  y-x\right)  V\left(  y\right)  \left(
m_{+}\left(  y,k\right)  -1\right)  dy-\\
-%
{\displaystyle\int\limits_{x}^{\infty}}
e^{2ik\left(  y-x\right)  }V\left(  y\right)  \left(  \partial_{k}m_{+}\left(
y,k\right)  \right)  dy-2i%
{\displaystyle\int\limits_{x}^{\infty}}
e^{2ik\left(  y-x\right)  }\left(  y-x\right)  V\left(  y\right)  dy.
\end{array}
\right.  \label{2.71}%
\end{equation}
Using (\ref{3.2}) and (\ref{3.5}) we estimate the first two terms in the
right-hand side of (\ref{2.71}) by $C\left\langle k\right\rangle
^{-1}\left\langle x\right\rangle ^{-2-\delta}.$ Using (\ref{3.28}) with
$V\left(  x\right)  $ replaced with $\left(  y-x\right)  V\left(  y\right)  $,
we control the last term in the right-hand side of (\ref{2.71}) by
$C\left\langle k\right\rangle ^{-1}\left\langle x\right\rangle ^{-1-\delta}.$
Hence, we deduce (\ref{3.7}) in the case $x\geq0.$ The case $x\leq0$ is
treated similarly.
\end{proof}

Equation (\ref{2.5}) can be rewritten in terms of $m_{\pm}\left(  x,k\right)
$ as follows
\begin{equation}
T\left(  k\right)  m_{\mp}\left(  x,k\right)  =R_{\pm}\left(  k\right)
e^{\pm2ikx}m_{\pm}\left(  x,k\right)  +m_{\pm}\left(  x,-k\right)  .
\label{3.1}%
\end{equation}
The coefficients $T\left(  k\right)  $ and $R_{\pm}\left(  k\right)  $ satisfy
the following relations (\cite{DeiftTrubowitz}, pages 144-146):%
\[
\overline{T\left(  k\right)  }=T\left(  -k\right)  ,
\]%
\[
\overline{R_{\pm}\left(  k\right)  }=R_{\pm}\left(  -k\right)
\]%
\begin{equation}
T\left(  k\right)  \overline{R_{-}\left(  k\right)  }+\overline{T\left(
k\right)  }R_{+}\left(  k\right)  =0. \label{i1}%
\end{equation}
\ Moreover, the integral representations%
\begin{equation}
\frac{1}{T\left(  k\right)  }=1-\frac{1}{2ik}\int_{-\infty}^{\infty}V\left(
y\right)  m_{+}\left(  y,k\right)  dy \label{3.14}%
\end{equation}
and%
\begin{equation}
\frac{\text{\ }R_{\pm}\left(  k\right)  }{T\left(  k\right)  }=\frac{1}%
{2ik}\int_{-\infty}^{\infty}e^{\mp2iky}V\left(  y\right)  m_{\mp}\left(
y,k\right)  dy \label{3.13}%
\end{equation}
hold. We observe that if $V\in\mathbf{L}^{1,1}\left(  \mathbb{R}\right)  ,$ by
(\ref{2.44}), (\ref{3.4}), (\ref{3.14}), (\ref{3.13})%
\begin{equation}
\left\vert T\left(  k\right)  -1\right\vert \leq\frac{C}{\left\langle
k\right\rangle } \label{3.18}%
\end{equation}
and%
\begin{equation}
\left\vert R_{\pm}\left(  k\right)  \right\vert \leq\frac{C}{\left\langle
k\right\rangle } \label{3.19}%
\end{equation}
Indeed, we have%
\[
\left\vert T\left(  k\right)  -1\right\vert \leq\frac{\left\vert T\left(
k\right)  \right\vert }{\left\vert k\right\vert }\int_{-\infty}^{\infty
}\left\langle y\right\rangle \left\vert V\left(  y\right)  \right\vert
dy\leq\frac{C}{\left\vert k\right\vert },\text{ for }\left\vert k\right\vert
\geq1,
\]
and%
\[
\left\vert R_{\pm}\left(  k\right)  \right\vert \leq\frac{\left\vert T\left(
k\right)  \right\vert }{\left\vert k\right\vert }\int_{-\infty}^{\infty
}\left\langle y\right\rangle \left\vert V\left(  y\right)  \right\vert
dy\leq\frac{C}{\left\vert k\right\vert },\text{ for }\left\vert k\right\vert
\geq1.
\]
In the following Proposition, we expose some estimates for the coefficients
$T\left(  k\right)  $ and $R_{\pm}\left(  k\right)  $ (see Theorem 2.3 of
\cite{Weder2000}).

\begin{proposition}
\label{Proposition 3.3}Suppose that $V\in\mathbf{L}^{1,3}$ is exceptional.
Then, the estimates
\begin{align}
\left\vert \frac{d}{dk}T\left(  k\right)  \right\vert  &  \leq C\left\langle
k\right\rangle ^{-1}\label{3.8}\\
\left\vert T\left(  k\right)  -T\left(  0\right)  \right\vert  &  \leq
C\left\vert k\right\vert ,\text{ }k\rightarrow0,\label{3.9}\\
\left\vert R_{\pm}\left(  k\right)  -R_{\pm}\left(  0\right)  \right\vert  &
\leq C\left\vert k\right\vert \text{, \ \ }k\rightarrow0, \label{3.10}%
\end{align}
are valid.
\end{proposition}

We also need to calculate the limit of $T\left(  k\right)  $ and $R_{\pm
}\left(  k\right)  ,$ as $k\rightarrow0.$ Let
\[
a=\lim_{x\rightarrow-\infty}f_{+}\left(  x,0\right)  .
\]
We have the following result (see Theorem 2.1 of \cite{Klaus} or
\cite{Weder2000}, page 52).

\begin{proposition}
\label{Proposition 3.2}Suppose that the potential $V\in\mathbf{L}^{1,3}\left(
\mathbb{R}\right)  $ is exceptional. Then, we get%
\begin{equation}
T\left(  k\right)  =\frac{2a}{1+a^{2}}+O\left(  k\right)  ,\text{\ as
}k\rightarrow0, \label{3.3}%
\end{equation}%
\begin{equation}
R_{\pm}\left(  k\right)  =\pm\frac{1-a^{2}}{1+a^{2}}+O\left(  k\right)
,\ \text{ as }k\rightarrow0. \label{3.12}%
\end{equation}

\end{proposition}

Finally, we present the following proposition.

\begin{proposition}
Suppose that the potential $V\in\mathbf{L}^{1,3}\left(  \mathbb{R}\right)  $
is exceptional. Then,
\begin{equation}
\left\vert \frac{d}{dk}R_{\pm}\left(  k\right)  \right\vert \leq C\left\langle
k\right\rangle ^{-1}. \label{3.17}%
\end{equation}

\end{proposition}

\begin{proof}
Similarly to the proof of (\ref{2.44}) given in Lemma 1
of~\cite{DeiftTrubowitz}, we show%
\begin{equation}
\left\vert \partial_{k}m_{+}\left(  x,k\right)  \right\vert \leq C\left\langle
x\right\rangle ^{2} \label{3.50}%
\end{equation}
and%
\begin{equation}
\left\vert \partial_{k}^{2}m_{+}\left(  x,k\right)  \right\vert \leq
C\left\langle x\right\rangle ^{3}, \label{3.51}%
\end{equation}
for any $x\in\mathbb{R}$. Using (\ref{2.44}), (\ref{3.50}) and (\ref{3.51}) we
see that
\begin{equation}
\left\vert \frac{d^{j}}{dk^{j}}\int_{-\infty}^{\infty}e^{\mp2iky}V\left(
y\right)  m_{\mp}\left(  y,k\right)  dy\right\vert \leq C,\text{ \ }j=0,1,2.
\label{3.53}%
\end{equation}
In the case of exceptional potentials $\int_{-\infty}^{\infty}V\left(
y\right)  m_{\mp}\left(  y,0\right)  dy=0$ (see (\cite{aktosun})). Then,
\begin{equation}
\left\vert \frac{d}{dk}\left(  \frac{1}{k}\int_{-\infty}^{\infty}e^{\mp
2iky}V\left(  y\right)  m_{\mp}\left(  y,k\right)  dy\right)  \right\vert
\leq\frac{C}{\left\langle k\right\rangle },\text{ \ }k\in\mathbb{R}\text{.}
\label{3.16}%
\end{equation}
Hence, multiplying equation (\ref{3.13}) by $T\left(  k\right)  ,$ derivating
the resulting relation and using (\ref{3.4}), (\ref{3.8}), (\ref{3.53}) and
(\ref{3.16}) we prove (\ref{3.17}).
\end{proof}

\section{\label{S6}Regularity and weighted estimates for $\mathcal{F}$ and
$\mathcal{F}^{-1}$.}

This section is dedicated to the proof of Lemmas \ref{L3} and \ref{L4}. In all
of the following results we only ask the properties of the Jost solutions that
we use in the proof to be true.\ It is straightforward to check that under
Condition \ref{Cond1} all the required properties (results of Section
\ref{S3}) are satisfied.

We need to control the $\mathbf{L}^{2}-$norm of the expressions $\partial
_{k}\mathcal{F}\phi,$ $x\mathcal{U}\left(  t\right)  \phi$ and $\partial
_{x}\left(  \mathcal{U}\left(  t\right)  \phi\right)  .$ We prove the following.

\begin{lemma}
\label{Lemma 7.2}Suppose that the estimates (\ref{3.4}), (\ref{3.2}) with
$\delta=0,$ (\ref{3.5}) with $\delta>1/2,$ (\ref{3.18}), (\ref{3.19}),
(\ref{3.8}), (\ref{3.3}), (\ref{3.12}) and (\ref{3.17}) hold. Also, assume
that $m_{+}\left(  x,0\right)  =m_{-}\left(  -x,0\right)  $. If $a=1,$ let
$\phi\in\mathbf{H}^{0,1}$ be odd and if $a=-1,$ suppose that $\phi$ is even.
Then, relation
\begin{equation}
\left(  \mathcal{F}\phi\right)  \left(  0\right)  =0. \label{7.60}%
\end{equation}
Furthermore
\begin{equation}
\left\Vert \mathcal{F}\phi\right\Vert _{\mathbf{H}^{1}}\leq K_{0}\left\Vert
\phi\right\Vert _{\mathbf{H}^{0,1}}, \label{7.61}%
\end{equation}
for some $K_{0}>0$ is true.\
\end{lemma}

\begin{proof}
Recall that
\begin{equation}
\mathcal{F}\phi=\frac{1}{\sqrt{2\pi}}\int_{-\infty}^{\infty}e^{-ikx}%
\overline{\Phi\left(  x,k\right)  }\phi\left(  x\right)  dx \label{7.62}%
\end{equation}
where
\[
\overline{\Phi\left(  x,k\right)  }=\theta\left(  k\right)  \overline{T\left(
k\right)  }\overline{m_{+}\left(  x,k\right)  }+\theta\left(  -k\right)
\overline{T\left(  -k\right)  }\overline{m_{-}\left(  x,-k\right)  }.
\]
In the domain $x>0$ we express $\Phi\left(  x,k\right)  $ in terms of $m_{+}$
using (\ref{3.1}). We find%
\begin{equation}
\overline{\Phi\left(  x,k\right)  }=\theta\left(  k\right)  \overline{T\left(
k\right)  }\overline{m_{+}\left(  x,k\right)  }+\theta\left(  -k\right)
\left(  e^{2ikx}\overline{R_{+}\left(  -k\right)  }\overline{m_{+}\left(
x,-k\right)  }+\overline{m_{+}\left(  x,k\right)  }\right)  . \label{4.68}%
\end{equation}
and in the domain $x<0$ we express $\Phi\left(  x,k\right)  $ in terms of
$m_{-}$%
\begin{equation}
\overline{\Phi\left(  x,k\right)  }=\theta\left(  k\right)  \left(
e^{2ikx}\overline{R_{-}\left(  k\right)  }\overline{m_{-}\left(  x,k\right)
}+\overline{m_{-}\left(  x,-k\right)  }\right)  +\theta\left(  -k\right)
\overline{T\left(  -k\right)  }\overline{m_{-}\left(  x,-k\right)  }.
\label{4.69}%
\end{equation}
Using (\ref{2.15}), (\ref{4.68}) and (\ref{4.69}) in (\ref{7.62}) we have%
\begin{equation}
\left.
\begin{array}
[c]{c}%
\left(  \mathcal{F}\phi\right)  \left(  k\right)  =\frac{\theta\left(
k\right)  }{\sqrt{2\pi}}\left(  \int_{0}^{\infty}e^{-ikx}\overline{T\left(
k\right)  }\overline{m_{+}\left(  x,k\right)  }\phi\left(  x\right)
dx+\int_{-\infty}^{0}e^{-ikx}\left(  e^{2ikx}\overline{R_{-}\left(  k\right)
}\overline{m_{-}\left(  x,k\right)  }+\overline{m_{-}\left(  x,-k\right)
}\right)  \phi\left(  x\right)  dx\right) \\
+\frac{\theta\left(  -k\right)  }{\sqrt{2\pi}}\left(  \int_{-\infty}%
^{0}e^{-ikx}\left(  \overline{T\left(  -k\right)  }\overline{m_{-}\left(
x,-k\right)  }\right)  \phi\left(  x\right)  dx+\int_{0}^{\infty}%
e^{-ikx}\left(  \left(  e^{2ikx}\overline{R_{+}\left(  -k\right)  }%
\overline{m_{+}\left(  x,-k\right)  }+\overline{m_{+}\left(  x,k\right)
}\right)  \right)  \phi\left(  x\right)  dx\right)  .
\end{array}
\right.  \label{7.43}%
\end{equation}
Note that if $a=\pm1,$ by (\ref{3.3}) $T\left(  0\right)  =\pm1$ and by
(\ref{3.12}) $R_{\pm}\left(  0\right)  =0.$ Also, by assumption $m_{+}\left(
x,0\right)  =m_{-}\left(  -x,0\right)  .$ Then, using that $\phi\ $is odd when
$a=1$ and $\phi$ is even if $a=-1,$ we get%
\[
\int_{0}^{\infty}\overline{T\left(  0\right)  }\overline{m_{+}\left(
x,0\right)  }\phi\left(  x\right)  dx+\int_{-\infty}^{0}\overline{m_{-}\left(
x,0\right)  }\phi\left(  x\right)  dx=0
\]
and%
\[
\int_{-\infty}^{0}\left(  \overline{T\left(  0\right)  }\overline{m_{-}\left(
x,0\right)  }\right)  \phi\left(  x\right)  dx+\int_{0}^{\infty}%
\overline{m_{+}\left(  x,0\right)  }\phi\left(  x\right)  dx=0.
\]
In particular, these equalities prove (\ref{7.60}). Moreover, we can commute
the derivative $\partial_{k}$ and the cut-off function $\theta$ in
(\ref{7.43}). Since $\mathcal{F}$ is unitary on $\mathbf{L}^{2}$ $\left\Vert
\mathcal{F}\phi\right\Vert _{\mathbf{L}^{2}}=\left\Vert \phi\right\Vert
_{\mathbf{L}^{2}}.$ We estimate $\partial_{k}\mathcal{F}\phi.$ Using
(\ref{7.43}) we decompose
\begin{equation}
\left.  \mathcal{F}\phi=\sum_{l=1}^{6}J_{l}+\mathcal{F}_{0}\phi,\right.
\label{7.13}%
\end{equation}
where%
\[
J_{1}=\frac{1}{\sqrt{2\pi}}\int_{0}^{\infty}e^{-ikx}\theta\left(  k\right)
\left(  \overline{T\left(  k\right)  }\overline{m_{+}\left(  x,k\right)
}-1\right)  \phi\left(  x\right)  dx,
\]%
\[
J_{2}=\frac{1}{\sqrt{2\pi}}\int_{0}^{\infty}e^{-ikx}\theta\left(  -k\right)
\left(  \overline{m_{+}\left(  x,k\right)  }-1\right)  \phi\left(  x\right)
dx,
\]%
\[
J_{3}=\frac{1}{\sqrt{2\pi}}\int_{-\infty}^{0}e^{-ikx}\theta\left(  -k\right)
\left(  \overline{T\left(  -k\right)  }\overline{m_{-}\left(  x,-k\right)
}-1\right)  \phi\left(  x\right)  dx,
\]%
\[
J_{4}=\frac{1}{\sqrt{2\pi}}\int_{-\infty}^{0}e^{-ikx}\theta\left(  k\right)
\left(  \overline{m_{-}\left(  x,-k\right)  }-1\right)  \phi\left(  x\right)
dx,
\]%
\[
J_{5}=\frac{\theta\left(  -k\right)  \overline{R_{+}\left(  -k\right)  }%
}{\sqrt{2\pi}}\int_{0}^{\infty}e^{ikx}\overline{m_{+}\left(  x,-k\right)
}\phi\left(  x\right)  dx
\]
and
\[
J_{6}=\frac{\theta\left(  k\right)  \overline{R_{-}\left(  k\right)  }}%
{\sqrt{2\pi}}\int_{-\infty}^{0}e^{ikx}\overline{m_{-}\left(  x,k\right)  }%
\phi\left(  x\right)  dx.
\]
We estimate $\left\Vert \partial_{k}J_{1}\right\Vert .$ We have
\begin{align*}
\left\Vert \partial_{k}J_{1}\right\Vert _{\mathbf{L}^{2}}  &  \leq C\left\Vert
\left(  \overline{T\left(  k\right)  }-1\right)  \int_{0}^{\infty}%
e^{-ikx}\left(  \overline{m_{+}\left(  x,k\right)  }-1\right)  x\phi\left(
x\right)  dx\right\Vert _{\mathbf{L}^{2}}+C\left\Vert \left\vert
\overline{T\left(  k\right)  }-1\right\vert \mathcal{F}_{0}\left(
\theta\left(  x\right)  x\phi\left(  x\right)  \right)  \right\Vert
_{\mathbf{L}^{2}}\\
&  +C\left\Vert \int_{0}^{\infty}e^{-ikx}\left(  \overline{m_{+}\left(
x,k\right)  }-1\right)  x\phi\left(  x\right)  dx\right\Vert _{\mathbf{L}^{2}%
}+C\left\Vert \partial_{k}\overline{T\left(  k\right)  }\int_{0}^{\infty
}e^{-ikx}\left(  \overline{m_{+}\left(  x,k\right)  }-1\right)  \phi\left(
x\right)  dx\right\Vert _{\mathbf{L}^{2}}\\
&  +C\left\Vert \partial_{k}\overline{T\left(  k\right)  }\mathcal{F}%
_{0}\left(  \theta\left(  x\right)  \phi\left(  x\right)  \right)  \right\Vert
_{\mathbf{L}^{2}}+C\left\Vert \overline{T\left(  k\right)  }\int_{0}^{\infty
}e^{-ikx}\partial_{k}\overline{m_{+}\left(  x,k\right)  }\phi\left(  x\right)
dx\right\Vert _{\mathbf{L}^{2}}.
\end{align*}
Then, using (\ref{3.4}), (\ref{3.2}) with $\delta=0,$ (\ref{3.5}) with
$\delta>1/2,$ (\ref{3.18}), (\ref{3.8}) and (\ref{3.17}) we obtain $\left\Vert
\partial_{k}J_{1}\right\Vert _{\mathbf{L}^{2}}\leq\left\Vert \phi\right\Vert
_{\mathbf{H}^{0,1}}.$ Similarly, we show that $\left\Vert \partial_{k}%
J_{l}\right\Vert _{\mathbf{L}^{2}}\leq\left\Vert \phi\right\Vert
_{\mathbf{H}^{0,1}},$ for $l=2,3,4,5,6,$ and hence, (\ref{7.61}) follows.
\end{proof}

\begin{lemma}
Suppose that the estimates (\ref{3.4}), (\ref{3.2}) with $\delta=0,$
(\ref{3.5}) with $\delta>1/2,$ (\ref{3.6}) with $\delta=0,$ (\ref{3.18}),
(\ref{3.19}), (\ref{3.8}) and (\ref{3.17}) hold. Suppose that (\ref{7.61}) is
true. Then, for any $t\in\mathbb{R}$, the estimates (\ref{7.8}) and
(\ref{7.9}) are satisfied.
\end{lemma}

\begin{proof}
Let us prove (\ref{7.8}). Due to the unitarity of $\mathcal{F}$ on
$\mathbf{L}^{2}$ $\left\Vert \mathcal{U}\left(  t\right)  \phi\right\Vert
_{\mathbf{L}^{2}}=\left\Vert \phi\right\Vert _{\mathbf{L}^{2}}.$ Since
$\mathcal{U}\left(  t\right)  \phi=\mathcal{F}^{-1}e^{-\frac{it}{2}k^{2}%
}\mathcal{F}\phi,$ to estimate $\left\Vert \partial_{x}\mathcal{U}\left(
t\right)  \phi\right\Vert _{\mathbf{L}^{2}}$ it suffices to prove that%
\begin{equation}
\left\Vert \partial_{x}\mathcal{F}^{-1}\phi\right\Vert _{\mathbf{L}^{2}}\leq
C\left\Vert \phi\right\Vert _{\mathbf{H}^{0,1}}\text{ } \label{7.11}%
\end{equation}
and%
\begin{equation}
\left\Vert \mathcal{F}\phi\right\Vert _{\mathbf{H}^{0,1}}\leq C\left\Vert
\phi\right\Vert _{\mathbf{H}^{1}}. \label{7.12}%
\end{equation}
First, we consider $\partial_{x}\mathcal{F}^{-1}\phi.$ By definition%
\begin{equation}
\partial_{x}\mathcal{F}^{-1}\phi=i\mathcal{F}^{-1}\left(  k\phi\left(
k\right)  \right)  +\frac{1}{\sqrt{2\pi}}\int_{-\infty}^{\infty}e^{ikx}\left(
\partial_{x}\Phi\left(  x,k\right)  \right)  \phi\left(  k\right)  dk.
\label{7.63}%
\end{equation}
By the unitarity of $\mathcal{F}^{-1}$, the first term in the right-hand side
of (\ref{7.63}) is estimated by $K\left\Vert \phi\right\Vert _{\mathbf{H}%
^{0,1}}.$ So, we turn to the second term. Using (\ref{2.9}) and (\ref{2.8}) we
get%
\[
\Phi\left(  x,k\right)  =e^{-ikx}\left(  \theta\left(  k\right)  T\left(
k\right)  f_{+}\left(  x,k\right)  +\theta\left(  -k\right)  T\left(
-k\right)  f_{-}\left(  x,-k\right)  \right)  .
\]
Substituting $m_{\pm}\left(  x,k\right)  =e^{\mp ikx}f_{\pm}\left(
x,k\right)  $ we get%
\begin{equation}
\Phi\left(  x,k\right)  =\theta\left(  k\right)  T\left(  k\right)
m_{+}\left(  x,k\right)  +\theta\left(  -k\right)  T\left(  -k\right)
m_{-}\left(  x,-k\right)  . \label{4.21}%
\end{equation}
Moreover, using (\ref{3.1}) for the upper sign we have%
\begin{equation}
\Phi\left(  x,k\right)  =\theta\left(  k\right)  T\left(  k\right)
m_{+}\left(  x,k\right)  +\theta\left(  -k\right)  m_{+}\left(  x,k\right)
+\theta\left(  -k\right)  R_{+}\left(  -k\right)  e^{-2ikx}m_{+}\left(
x,-k\right)  , \label{4.42}%
\end{equation}
for $x\geq0.$ Using (\ref{3.1}) for the lower sign in (\ref{4.21}) we get%
\begin{equation}
\Phi\left(  x,k\right)  =\theta\left(  k\right)  R_{-}\left(  k\right)
e^{-2ikx}m_{-}\left(  x,k\right)  +\theta\left(  k\right)  m_{-}\left(
x,-k\right)  +\theta\left(  -k\right)  T\left(  -k\right)  m_{-}\left(
x,-k\right)  \label{4.43}%
\end{equation}
in the case $x\leq0.$ Using (\ref{4.42}) and (\ref{4.43}) we have
\[
\partial_{x}\Phi\left(  x,k\right)  =\left(  \theta\left(  k\right)  T\left(
k\right)  +\theta\left(  -k\right)  \right)  \partial_{x}m_{+}\left(
x,k\right)  +\theta\left(  -k\right)  R_{+}\left(  -k\right)  \partial
_{x}\left(  e^{-2ikx}m_{+}\left(  x,-k\right)  \right)
\]
for $x\geq0$ and
\[
\partial_{x}\Phi\left(  x,k\right)  =\theta\left(  k\right)  R_{-}\left(
k\right)  \partial_{x}\left(  e^{-2ikx}m_{-}\left(  x,k\right)  \right)
+\left(  \theta\left(  k\right)  +\theta\left(  -k\right)  T\left(  -k\right)
\right)  \partial_{x}m_{-}\left(  x,-k\right)
\]
for $x\leq0.$ Then,
\begin{equation}
\int_{-\infty}^{\infty}e^{ikx}\left(  \partial_{x}\Phi\left(  x,k\right)
\right)  \phi\left(  k\right)  dk=\sum_{j=1}^{6}I_{j}, \label{7.10}%
\end{equation}
where%
\[
I_{1}=\theta\left(  x\right)  \int_{-\infty}^{\infty}e^{ikx}\left(
\theta\left(  k\right)  T\left(  k\right)  +\theta\left(  -k\right)  \right)
\partial_{x}m_{+}\left(  x,k\right)  \phi\left(  k\right)  dk
\]%
\[
I_{2}=\theta\left(  x\right)  \int_{-\infty}^{\infty}e^{-ikx}\theta\left(
-k\right)  R_{+}\left(  -k\right)  \partial_{x}m_{+}\left(  x,-k\right)
\phi\left(  k\right)  dk
\]%
\[
I_{3}=\theta\left(  -x\right)  \int_{-\infty}^{\infty}e^{-ikx}\theta\left(
k\right)  R_{-}\left(  k\right)  \partial_{x}m_{-}\left(  x,k\right)
\phi\left(  k\right)  dk
\]%
\[
I_{4}=\theta\left(  -x\right)  \int_{-\infty}^{\infty}e^{ikx}\left(
\theta\left(  k\right)  +\theta\left(  -k\right)  T\left(  -k\right)  \right)
\partial_{x}m_{-}\left(  x,-k\right)  \phi\left(  k\right)  dk
\]%
\[
I_{5}=-2i\theta\left(  x\right)  \int_{-\infty}^{\infty}e^{-ikx}\theta\left(
-k\right)  R_{+}\left(  -k\right)  m_{+}\left(  x,-k\right)  k\phi\left(
k\right)  dk
\]
and%
\[
I_{6}=-2i\theta\left(  -x\right)  \int_{-\infty}^{\infty}e^{-ikx}\theta\left(
k\right)  R_{-}\left(  k\right)  m_{-}\left(  x,k\right)  k\phi\left(
k\right)  dk.
\]
Using (\ref{3.4}) and (\ref{3.6}) with $\delta=0$ we estimate%
\[
\int_{-\infty}^{\infty}\left\vert \left(  \theta\left(  k\right)  T\left(
k\right)  +\theta\left(  -k\right)  \right)  \right\vert \left\vert
\theta\left(  x\right)  \partial_{x}m_{+}\left(  x,k\right)  \right\vert
\left\vert \phi\left(  k\right)  \right\vert dk\leq C\left\langle
x\right\rangle ^{-2}\int_{-\infty}^{\infty}\left\langle k\right\rangle
^{-1}\left\vert \phi\left(  k\right)  \right\vert dk.
\]
Then, by Cauchy-Schwartz inequality we get%
\[
\left\Vert I_{1}\right\Vert _{\mathbf{L}^{2}}\leq C\left\Vert \left\langle
x\right\rangle ^{-2}\right\Vert _{\mathbf{L}^{2}}\left\Vert \left\langle
k\right\rangle ^{-1}\right\Vert _{\mathbf{L}^{2}}\left\Vert \phi\right\Vert
_{\mathbf{L}^{2}}\leq C\left\Vert \phi\right\Vert _{\mathbf{L}^{2}}.
\]
Similarly we show%
\[
\left\Vert I_{j}\right\Vert _{\mathbf{L}^{2}}\leq C\left\Vert \phi\right\Vert
_{\mathbf{L}^{2}},
\]
for $j=2,3,4.$ We write $I_{5}$ as%
\begin{align*}
I_{5}  &  =-2i\theta\left(  x\right)  \int_{-\infty}^{\infty}e^{-ikx}%
\theta\left(  -k\right)  R_{+}\left(  -k\right)  \left(  m_{+}\left(
x,-k\right)  -1\right)  k\phi\left(  k\right)  dk\\
&  -2\sqrt{2\pi}i\theta\left(  x\right)  \mathcal{F}_{0}\left(  \theta\left(
-k\right)  R_{+}\left(  -k\right)  k\phi\left(  k\right)  \right)  .
\end{align*}
By (\ref{3.4}) and (\ref{3.2}) with $\delta=0$ we obtain%
\[
\int_{-\infty}^{\infty}\left\vert \theta\left(  -k\right)  R_{+}\left(
-k\right)  \right\vert \left\vert \theta\left(  x\right)  \left(  m_{+}\left(
x,-k\right)  -1\right)  \right\vert \left\vert k\phi\left(  k\right)
\right\vert dk\leq C\left\langle x\right\rangle ^{-1}\int_{-\infty}^{\infty
}\left\langle k\right\rangle ^{-1}\left\vert k\phi\left(  k\right)
\right\vert dk.
\]
Then, via Cauchy-Schwartz inequality
\[
\left\Vert \int_{-\infty}^{\infty}e^{-ikx}\theta\left(  -k\right)
R_{+}\left(  -k\right)  \theta\left(  x\right)  \left(  m_{+}\left(
x,-k\right)  -1\right)  k\phi\left(  k\right)  dk\right\Vert _{\mathbf{L}^{2}%
}\leq C\left\Vert \left\langle x\right\rangle ^{-1}\right\Vert _{\mathbf{L}%
^{2}}\left\Vert \left\langle k\right\rangle ^{-1}\right\Vert _{\mathbf{L}^{2}%
}\left\Vert k\phi\right\Vert _{\mathbf{L}^{2}}\leq C\left\Vert \phi\right\Vert
_{\mathbf{H}^{0,1}}.
\]
Since by (\ref{3.4}) and Parseval's identity $\left\Vert \theta\left(
x\right)  \mathcal{F}_{0}\left(  \theta\left(  -k\right)  R_{+}\left(
-k\right)  k\phi\left(  k\right)  \right)  \right\Vert _{\mathbf{L}^{2}}%
\leq\left\Vert k\phi\left(  k\right)  \right\Vert _{\mathbf{L}^{2}}%
\leq\left\Vert \phi\right\Vert _{\mathbf{H}^{0,1}},$ we deduce%
\[
\left\Vert I_{5}\right\Vert _{\mathbf{L}^{2}}\leq C\left\Vert \phi\right\Vert
_{\mathbf{H}^{0,1}}.
\]
Similarly, we show that $\left\Vert I_{6}\right\Vert _{\mathbf{L}^{2}}\leq
C\left\Vert \phi\right\Vert _{\mathbf{H}^{0,1}}.$ Gathering together the
estimates for $I_{j},$ $j=1,2,3,4,5,6,$ from (\ref{7.10}) we prove
(\ref{7.11}). We now consider (\ref{7.12}). We use (\ref{7.13}). We write
$kJ_{1}$ as
\begin{equation}
\left.  kJ_{1}=\theta\left(  k\right)  k\left(  \overline{T\left(  k\right)
}-1\right)  \mathcal{F}_{0}\left(  \theta\left(  x\right)  \phi\left(
x\right)  \right)  +\frac{\theta\left(  k\right)  \overline{T\left(  k\right)
}}{\sqrt{2\pi}}k\int_{0}^{\infty}e^{-ikx}\left(  \overline{m_{+}\left(
x,k\right)  }-1\right)  \phi\left(  x\right)  dx.\right.  \label{7.14}%
\end{equation}
Using (\ref{3.18}) we estimate%
\begin{equation}
\left\Vert \theta\left(  k\right)  k\left(  \overline{T\left(  k\right)
}-1\right)  \mathcal{F}_{0}\left(  \theta\left(  x\right)  \phi\left(
x\right)  \right)  \right\Vert _{\mathbf{L}^{2}}\leq C\left\Vert
\mathcal{F}_{0}\left(  \theta\left(  x\right)  \phi\left(  x\right)  \right)
\right\Vert _{\mathbf{L}^{2}}\leq C\left\Vert \phi\left(  x\right)
\right\Vert _{\mathbf{L}^{2}}. \label{7.16}%
\end{equation}
Integrating by parts in the second term of the right-hand side of (\ref{7.14})
we get%
\[
k\int_{0}^{\infty}e^{-ikx}\left(  \overline{m_{+}\left(  x,k\right)
}-1\right)  \phi\left(  x\right)  dx=-i\left.  \left(  \overline{m_{+}\left(
x,k\right)  }-1\right)  \phi\left(  x\right)  \right\vert _{x=0}-i\int
_{0}^{\infty}e^{-ikx}\partial_{x}\left(  \left(  \overline{m_{+}\left(
x,k\right)  }-1\right)  \phi\left(  x\right)  \right)  dx.
\]
Then, using (\ref{3.4}), (\ref{3.2}) and (\ref{3.6}) with $\delta=0$ we
estimate%
\begin{equation}
\left.
\begin{array}
[c]{c}%
\left\Vert \frac{\theta\left(  k\right)  \overline{T\left(  k\right)  }}%
{\sqrt{2\pi}}k\int_{0}^{\infty}e^{-ikx}\left(  \overline{m_{+}\left(
x,k\right)  }-1\right)  \phi\left(  x\right)  dx\right\Vert _{\mathbf{L}^{2}%
}\\
\leq C\left\Vert \left\langle k\right\rangle ^{-1}\right\Vert _{\mathbf{L}%
^{2}}\left\vert \phi\left(  0\right)  \right\vert +C\left\Vert \left\langle
k\right\rangle ^{-1}\int_{0}^{\infty}\left\langle x\right\rangle ^{-1}\left(
\left\vert \phi\left(  x\right)  \right\vert +\left\vert \partial_{x}%
\phi\left(  x\right)  \right\vert \right)  dx\right\Vert _{\mathbf{L}^{2}}\leq
C\left\Vert \phi\right\Vert _{\mathbf{H}^{1}}.
\end{array}
\right.  \label{7.17}%
\end{equation}
Using (\ref{7.16}) and (\ref{7.17}) in (\ref{7.14}) we show%
\[
\left\Vert J_{1}\right\Vert _{\mathbf{H}^{0,1}}\leq C\left\Vert \phi
\right\Vert _{\mathbf{H}^{1}}.
\]
Similarly we prove that $kJ_{2},$ $kJ_{3}$ and $kJ_{4}$ in (\ref{7.13}) are
controlled by $C\left\Vert \phi\right\Vert _{\mathbf{H}^{1}}.$ We write
$kJ_{5}$ as
\[
\left.
\begin{array}
[c]{c}%
kJ_{5}=k\frac{\theta\left(  -k\right)  \overline{R_{+}\left(  -k\right)  }%
}{\sqrt{2\pi}}\int_{0}^{\infty}e^{ikx}\left(  \overline{m_{+}\left(
x,-k\right)  }-1\right)  \phi\left(  x\right)  dx-i\frac{\theta\left(
-k\right)  \overline{R_{+}\left(  -k\right)  }}{\sqrt{2\pi}}\int_{0}^{\infty
}\partial_{x}e^{ikx}\phi\left(  x\right)  dx\\
=k\frac{\theta\left(  -k\right)  \overline{R_{+}\left(  -k\right)  }}%
{\sqrt{2\pi}}\int_{0}^{\infty}e^{ikx}\left(  \overline{m_{+}\left(
x,-k\right)  }-1\right)  \phi\left(  x\right)  dx+i\frac{\theta\left(
-k\right)  \overline{R_{+}\left(  -k\right)  }}{\sqrt{2\pi}}\phi\left(
0\right)  +i\frac{\theta\left(  -k\right)  \overline{R_{+}\left(  -k\right)
}}{\sqrt{2\pi}}\int_{0}^{\infty}e^{ikx}\partial_{x}\phi\left(  x\right)  dx.
\end{array}
\right.
\]
Then, using (\ref{3.2}) with $\delta=0$ and (\ref{3.19}) we show that%
\[
\left\Vert J_{5}\right\Vert _{\mathbf{H}^{0,1}}\leq C\left\Vert \phi
\right\Vert _{\mathbf{H}^{1}}.
\]
Similarly, we estimate $J_{6}$. Therefore, as $k\mathcal{F}_{0}\phi
=i\mathcal{F}_{0}\left(  \partial_{x}\phi\right)  ,$ from (\ref{7.13}) we
attain (\ref{7.12}).

We now turn to (\ref{7.9}). Observe that%
\begin{equation}
x\mathcal{U}\left(  t\right)  \phi=x\mathcal{F}^{-1}\left(  e^{-\frac{it}%
{2}k^{2}}-1\right)  \mathcal{F}\phi+x\phi=L+x\phi\label{3.22}%
\end{equation}
with%
\[
L=\frac{x}{\sqrt{2\pi}}\int_{-\infty}^{\infty}e^{ikx}\left(  \left(
e^{-\frac{it}{2}k^{2}}-1\right)  \Phi\left(  x,k\right)  \left(
\mathcal{F}\phi\right)  \left(  k\right)  \right)  dk.
\]
Using (\ref{4.42}) and (\ref{4.43}) we get%
\[
L=\theta\left(  x\right)  \left(  L_{1}+L_{2}\right)  +\theta\left(
-x\right)  \left(  L_{3}+L_{4}\right)  ,
\]
where%
\[
L_{1}=\frac{x}{\sqrt{2\pi}}\int_{-\infty}^{\infty}e^{ikx}\left(  e^{-\frac
{it}{2}k^{2}}-1\right)  \left(  \theta\left(  k\right)  T\left(  k\right)
+\theta\left(  -k\right)  \right)  m_{+}\left(  x,k\right)  \left(
\mathcal{F}\phi\right)  \left(  k\right)  dk
\]%
\[
L_{2}=\frac{x}{\sqrt{2\pi}}\int_{-\infty}^{\infty}e^{-ikx}\left(
e^{-\frac{it}{2}k^{2}}-1\right)  \theta\left(  -k\right)  R_{+}\left(
-k\right)  m_{+}\left(  x,-k\right)  \left(  \mathcal{F}\phi\right)  \left(
k\right)  dk
\]%
\[
L_{3}=\frac{x}{\sqrt{2\pi}}\int_{-\infty}^{\infty}e^{-ikx}\left(
e^{-\frac{it}{2}k^{2}}-1\right)  \theta\left(  k\right)  R_{-}\left(
k\right)  m_{-}\left(  x,k\right)  \left(  \mathcal{F}\phi\right)  \left(
k\right)  dk
\]%
\[
L_{4}=\frac{x}{\sqrt{2\pi}}\int_{-\infty}^{\infty}e^{ikx}\left(  e^{-\frac
{it}{2}k^{2}}-1\right)  \left(  \theta\left(  k\right)  +\theta\left(
-k\right)  T\left(  -k\right)  \right)  m_{-}\left(  x,-k\right)  \left(
\mathcal{F}\phi\right)  \left(  k\right)  dk.
\]
Integrating by parts in $L_{1}$ we have%
\begin{equation}
L_{1}=L_{11}+L_{12}+L_{13}+L_{14}, \label{3.21}%
\end{equation}
where%
\[
L_{11}=\frac{t}{\sqrt{2\pi}}\int_{-\infty}^{\infty}e^{ikx}e^{-\frac{it}%
{2}k^{2}}\left(  \theta\left(  k\right)  T\left(  k\right)  +\theta\left(
-k\right)  \right)  m_{+}\left(  x,k\right)  k\left(  \mathcal{F}\phi\right)
\left(  k\right)  dk,
\]%
\[
L_{12}=i\frac{1}{\sqrt{2\pi}}\int_{-\infty}^{\infty}e^{ikx}\left(
e^{-\frac{it}{2}k^{2}}-1\right)  \theta\left(  k\right)  \partial_{k}T\left(
k\right)  m_{+}\left(  x,k\right)  \left(  \mathcal{F}\phi\right)  \left(
k\right)  dk,
\]%
\[
L_{13}=i\frac{1}{\sqrt{2\pi}}\int_{-\infty}^{\infty}e^{ikx}\left(
e^{-\frac{it}{2}k^{2}}-1\right)  \left(  \theta\left(  k\right)  T\left(
k\right)  +\theta\left(  -k\right)  \right)  \partial_{k}m_{+}\left(
x,k\right)  \left(  \mathcal{F}\phi\right)  \left(  k\right)  dk
\]
and%
\[
L_{14}=i\frac{1}{\sqrt{2\pi}}\int_{-\infty}^{\infty}e^{ikx}\left(
e^{-\frac{it}{2}k^{2}}-1\right)  \left(  \theta\left(  k\right)  T\left(
k\right)  +\theta\left(  -k\right)  \right)  m_{+}\left(  x,k\right)
\partial_{k}\left(  \mathcal{F}\phi\right)  \left(  k\right)  dk.
\]
We estimate $L_{11}$ by using (\ref{3.4}) and (\ref{3.2}) with $\delta=0$
\begin{align*}
\left\vert L_{11}\right\vert  &  \leq\left\vert \frac{t}{\sqrt{2\pi}}%
\int_{-\infty}^{\infty}e^{ikx}e^{-\frac{it}{2}k^{2}}\left(  \theta\left(
k\right)  T\left(  k\right)  +\theta\left(  -k\right)  \right)  \left(
m_{+}\left(  x,k\right)  -1\right)  k\left(  \mathcal{F}\phi\right)  \left(
k\right)  dk\right\vert \\
&  +\left\vert t\mathcal{F}_{0}^{-1}\left(  e^{-\frac{it}{2}k^{2}}\left(
\theta\left(  k\right)  T\left(  k\right)  +\theta\left(  -k\right)  \right)
k\left(  \mathcal{F}\phi\right)  \left(  k\right)  \right)  dk\right\vert \\
&  \leq C\left\vert \frac{t}{\sqrt{2\pi}}\left\langle x\right\rangle ^{-1}%
\int_{-\infty}^{\infty}\left\langle k\right\rangle ^{-1}k\left(
\mathcal{F}\phi\right)  \left(  k\right)  dk\right\vert +C\left\vert
t\mathcal{F}_{0}^{-1}\left(  e^{-\frac{it}{2}k^{2}}\left(  \theta\left(
k\right)  T\left(  k\right)  +\theta\left(  -k\right)  \right)  k\left(
\mathcal{F}\phi\right)  \left(  k\right)  \right)  dk\right\vert .
\end{align*}
Then, via Cauchy-Schwartz inequality and (\ref{7.12}) we get%
\begin{equation}
\left.  \left\Vert L_{11}\right\Vert _{\mathbf{L}^{2}}\leq Ct\left(
\left\Vert \left\langle x\right\rangle ^{-1}\right\Vert _{\mathbf{L}^{2}%
}+1\right)  \left\Vert \mathcal{F}\phi\right\Vert _{\mathbf{H}^{0,1}}\leq
Ct\left\Vert \phi\right\Vert _{\mathbf{H}^{1}}.\right.  \label{7.18}%
\end{equation}
Similarly, by using the estimates (\ref{3.4}), (\ref{3.2}) with $\delta=0,$
(\ref{3.5}) with $\delta>1/2,$ and (\ref{3.8}) we prove that
\begin{equation}
\left\Vert L_{12}\right\Vert _{\mathbf{L}^{2}}+\left\Vert L_{13}\right\Vert
_{\mathbf{L}^{2}}\leq C\left\Vert \phi\right\Vert _{\mathbf{H}^{1}}.
\label{7.19}%
\end{equation}
We now study $L_{14}.$ We write
\begin{align*}
L_{14}  &  =i\mathcal{F}_{0}^{-1}\left(  \left(  e^{-\frac{it}{2}k^{2}%
}-1\right)  \left(  \theta\left(  k\right)  T\left(  k\right)  +\theta\left(
-k\right)  \right)  \partial_{k}\left(  \mathcal{F}\phi\right)  \left(
k\right)  \right) \\
&  +i\frac{1}{\sqrt{2\pi}}\int_{-\infty}^{\infty}e^{ikx}\left(  e^{-\frac
{it}{2}k^{2}}-1\right)  \left(  \theta\left(  k\right)  T\left(  k\right)
+\theta\left(  -k\right)  \right)  \left(  m_{+}\left(  x,k\right)  -1\right)
\partial_{k}\left(  \mathcal{F}\phi\right)  \left(  k\right)  dk.
\end{align*}
Then,
\begin{align*}
\left\vert L_{14}\right\vert  &  \leq\left\vert \mathcal{F}_{0}^{-1}\left(
\left(  e^{-\frac{it}{2}k^{2}}-1\right)  \left(  \theta\left(  k\right)
T\left(  k\right)  +\theta\left(  -k\right)  \right)  \partial_{k}\left(
\mathcal{F}\phi\right)  \left(  k\right)  \right)  \right\vert \\
&  +C\int_{-\infty}^{\infty}\left(  \left\vert T\left(  k\right)  \right\vert
+1\right)  \left\vert m_{+}\left(  x,k\right)  -1\right\vert \left\vert
\partial_{k}\left(  \mathcal{F}\phi\right)  \left(  k\right)  \right\vert dk.
\end{align*}
Thus, via (\ref{3.4}) and (\ref{3.2}) with $\delta=0$ we have%
\begin{align*}
\left\Vert L_{14}\right\Vert _{\mathbf{L}^{2}}  &  \leq\left\Vert \partial
_{k}\left(  \mathcal{F}\phi\right)  \left(  k\right)  \right\Vert
_{\mathbf{L}^{2}}+C\left\Vert \left\langle x\right\rangle ^{-1}\int_{-\infty
}^{\infty}\left\langle k\right\rangle ^{-1}\left\vert \partial_{k}\left(
\mathcal{F}\phi\right)  \left(  k\right)  \right\vert dk\right\Vert
_{\mathbf{L}^{2}}\\
&  \leq\left\Vert \partial_{k}\left(  \mathcal{F}\phi\right)  \left(
k\right)  \right\Vert _{\mathbf{L}^{2}}+C\left\Vert \left\langle
x\right\rangle ^{-1}\right\Vert _{\mathbf{L}^{2}}^{2}\left\Vert \partial
_{k}\left(  \mathcal{F}\phi\right)  \left(  k\right)  \right\Vert
_{\mathbf{L}^{2}}.
\end{align*}
Hence, using (\ref{7.61}) we estimate%
\begin{equation}
\left\Vert L_{14}\right\Vert _{\mathbf{L}^{2}}\leq C\left\Vert \phi\right\Vert
_{\mathbf{H}^{0,1}}. \label{7.20}%
\end{equation}
Relations (\ref{3.21}),\ (\ref{7.18}), (\ref{7.19}) and (\ref{7.20}) imply
\[
\left\Vert L_{1}\right\Vert _{\mathbf{L}^{2}}\leq C\left(  \left\langle
t\right\rangle \left\Vert \phi\right\Vert _{\mathbf{H}^{1}}+\left\Vert
\phi\right\Vert _{\mathbf{H}^{0,1}}\right)
\]
In the same spirit we estimate $L_{2},L_{3}$ and $L_{4}$. Hence, from
(\ref{3.22}) we see that
\[
\left\Vert x\mathcal{U}\left(  t\right)  \phi\right\Vert _{\mathbf{L}^{2}}\leq
C\left(  \left\langle t\right\rangle \left\Vert \phi\right\Vert _{\mathbf{H}%
^{1}}+\left\Vert \phi\right\Vert _{\mathbf{H}^{0,1}}\right)  .
\]
Using that $\left\Vert \mathcal{U}\left(  t\right)  \phi\right\Vert
_{\mathbf{L}^{2}}=\left\Vert \phi\right\Vert _{\mathbf{L}^{2}}$, we attain
(\ref{7.9}).
\end{proof}

\section{\label{S4}Estimates for dilatation $\mathcal{V}$.}

This section is devoted to the proof of Lemma \ref{L1}. In the following lemma
we study the large-time asymptotics for $\mathcal{V}\left(  t\right)  .$ We
denote by%
\begin{equation}
\Lambda\left(  x\right)  =\theta\left(  x\right)  R_{+}\left(  x\right)
+\theta\left(  -x\right)  R_{-}\left(  -x\right)  . \label{lamda}%
\end{equation}
We prove the following:

\begin{lemma}
Let $a=\pm1.$ Suppose that (\ref{3.4}), (\ref{3.2}) and (\ref{3.5}) with
$\delta=0$, (\ref{3.8}), (\ref{3.3}), (\ref{3.12}) and (\ref{3.17}) hold.
Then, if $a=1,$ the estimate
\begin{equation}
\left.  \left\Vert \mathcal{V}\left(  t\right)  \phi-T\left(  \left\vert
x\right\vert \right)  \phi\left(  x\right)  -\Lambda\left(  x\right)
\phi\left(  -x\right)  \right\Vert _{\mathbf{L}^{\infty}}\leq C\left\vert
\phi\left(  0\right)  \right\vert +Ct^{-1/4}\left\Vert \phi\right\Vert
_{\mathbf{H}^{1}}\right.  \label{4.20}%
\end{equation}
is true for all $t\geq1.$ Moreover, in the case $a=-1,$
\begin{equation}
\left.  \left\Vert \mathcal{V}\left(  t\right)  \phi-T\left(  \left\vert
x\right\vert \right)  \phi\left(  x\right)  -\Lambda\left(  x\right)
\phi\left(  -x\right)  -\sqrt{\frac{2i}{\pi}}\phi\left(  0\right)
{\displaystyle\int\limits_{\sqrt{t}x}^{\infty}}
e^{-\frac{ik^{2}}{2}}dk\right\Vert _{\mathbf{L}^{\infty}}\leq Ct^{-1/4}%
\left\Vert \phi\right\Vert _{\mathbf{H}^{1}}\right.  \label{4.20a}%
\end{equation}
holds for all $t\geq1.$
\end{lemma}

\begin{proof}
We consider first the case $x\geq0$. Using (\ref{4.42}) in (\ref{2.10}) we
have%
\begin{equation}
\mathcal{V}\left(  t\right)  \phi=\sqrt{\frac{it}{2\pi}}\int_{-\infty}%
^{\infty}e^{-\frac{it}{2}\left(  k-x\right)  ^{2}}\left(  \theta\left(
k\right)  T\left(  k\right)  m_{+}\left(  tx,k\right)  +\theta\left(
-k\right)  m_{+}\left(  tx,k\right)  +\theta\left(  -k\right)  R_{+}\left(
-k\right)  e^{-2itkx}m_{+}\left(  tx,-k\right)  \right)  \phi\left(  k\right)
dk. \label{4.19}%
\end{equation}
Let us denote by $\mathcal{V}^{+}\left(  t\right)  $ to $\mathcal{V}\left(
t\right)  $ in the case when $a=1$ and by $\mathcal{V}^{-}\left(  t\right)  $
to $\mathcal{V}\left(  t\right)  $ in the case when $a=-1.$ We decompose
\begin{equation}
\left.  \mathcal{V}^{\pm}\left(  t\right)  \phi=\sum_{j=1}^{3}\mathcal{V}%
_{j}^{\pm}\left(  t\right)  ,\right.  \label{4.15}%
\end{equation}
where%
\[
\mathcal{V}_{1}^{\pm}\left(  t\right)  =\sqrt{\frac{it}{2\pi}}%
{\displaystyle\int\limits_{-\infty}^{\infty}}
e^{-\frac{it}{2}\left(  k-x\right)  ^{2}}\theta\left(  k\right)  \left(
T\left(  k\right)  \mp1\right)  m_{+}\left(  tx,k\right)  \phi\left(
k\right)  dk,
\]%
\[
\mathcal{V}_{2}^{\pm}\left(  t\right)  =\sqrt{\frac{it}{2\pi}}%
{\displaystyle\int\limits_{-\infty}^{\infty}}
e^{-\frac{it}{2}\left(  k-x\right)  ^{2}}\theta\left(  k\right)  R_{+}\left(
k\right)  m_{+}\left(  tx,k\right)  \phi\left(  -k\right)  dk,
\]%
\[
\mathcal{V}_{3}^{\pm}\left(  t\right)  =\sqrt{\frac{it}{2\pi}}%
{\displaystyle\int\limits_{-\infty}^{\infty}}
e^{-\frac{it}{2}\left(  k-x\right)  ^{2}}\left(  \theta\left(  -k\right)
\pm\theta\left(  k\right)  \right)  m_{+}\left(  tx,k\right)  \phi\left(
k\right)  dk.
\]
We first study $\mathcal{V}_{1}^{\pm}\left(  t\right)  .$ Using that
$\int_{-\infty}^{\infty}e^{-\frac{it}{2}k^{2}}dk=\sqrt{\frac{2\pi}{it}},$ we
decompose $\mathcal{V}_{1}^{\pm}\left(  t\right)  $ as
\begin{equation}
\mathcal{V}_{1}^{\pm}\left(  t\right)  =\left(  T\left(  x\right)
\mp1\right)  \phi\left(  x\right)  +\mathcal{V}_{11}^{\pm}\left(  t\right)
+\mathcal{V}_{12}^{\pm}\left(  t\right)  , \label{4.4}%
\end{equation}
with%
\[
\mathcal{V}_{11}^{\pm}\left(  t\right)  =\left(  T\left(  x\right)
\mp1\right)  \left(  m_{+}\left(  tx,x\right)  -1\right)  \phi\left(
x\right)  ,
\]%
\[
\mathcal{V}_{12}^{\pm}\left(  t\right)  =\sqrt{\frac{it}{2\pi}}%
{\displaystyle\int\limits_{-\infty}^{\infty}}
e^{-\frac{it}{2}\left(  k-x\right)  ^{2}}\Phi_{1}\left(  x,k\right)  dk
\]
and%
\[
\Phi_{1}\left(  x,k\right)  =\theta\left(  k\right)  \left(  T\left(
k\right)  \mp1\right)  m_{+}\left(  tx,k\right)  \phi\left(  k\right)
-\left(  T\left(  x\right)  \mp1\right)  m_{+}\left(  tx,x\right)  \phi\left(
x\right)  .
\]
From (\ref{3.2}) $\delta=0$ and (\ref{3.3}) it follows%
\begin{equation}
\left\vert \mathcal{V}_{11}^{\pm}\left(  t\right)  \right\vert \leq
C\left\vert x\right\vert \left\langle tx\right\rangle ^{-1}\left\Vert
\phi\right\Vert _{\mathbf{L}^{\infty}}\leq\frac{C}{t}\left\Vert \phi
\right\Vert _{\mathbf{H}^{1}}. \label{4.2}%
\end{equation}
Now, we estimate $\mathcal{V}_{12}^{\pm}\left(  t\right)  .$ Using the
identity%
\begin{equation}
e^{-\frac{it}{2}\left(  k-x\right)  ^{2}}=B\partial_{k}\left(  \left(
k-x\right)  e^{-\frac{it}{2}\left(  k-x\right)  ^{2}}\right)  \label{4.7}%
\end{equation}
with $B:=\left(  1-it\left(  k-x\right)  ^{2}\right)  ^{-1},$ integrating by
parts in the definition of $\mathcal{V}_{12}\left(  t\right)  $ and using
$\left(  k-x\right)  \partial_{k}B=2it\left(  k-x\right)  ^{2}B^{2}$ we have%
\[
\left.  \mathcal{V}_{12}^{\pm}\left(  t\right)  =\sqrt{\frac{it}{2\pi}}%
{\displaystyle\int\limits_{-\infty}^{\infty}}
e^{-\frac{it}{2}\left(  k-x\right)  ^{2}}\left(  k-x\right)  B\partial_{k}%
\Phi_{1}\left(  x,k\right)  dk+2it\sqrt{\frac{it}{2\pi}}%
{\displaystyle\int\limits_{-\infty}^{\infty}}
e^{-\frac{it}{2}\left(  k-x\right)  ^{2}}\left(  k-x\right)  ^{2}B^{2}\Phi
_{1}\left(  x,k\right)  dk.\right.
\]
It follows from (\ref{3.4}), (\ref{3.2}), (\ref{3.5}), (\ref{3.8}) and
\begin{equation}
\left\vert \phi\left(  k\right)  -\phi\left(  x\right)  \right\vert \leq
C\left\vert k-x\right\vert ^{\frac{1}{2}}\left\Vert \partial_{k}%
\phi\right\Vert _{\mathbf{L}^{2}} \label{4.8}%
\end{equation}
that%
\[
\left\vert \Phi_{1}\left(  x,k\right)  \right\vert \leq C\left\vert
k-x\right\vert \left\vert \phi\left(  k\right)  \right\vert +C\left\vert
k-x\right\vert ^{\frac{1}{2}}\left\Vert \partial_{k}\phi\right\Vert
_{\mathbf{L}^{2}}%
\]
and
\[
\left\vert \partial_{k}\Phi_{1}\left(  x,k\right)  \right\vert \leq C\left(
\left\vert \phi\left(  k\right)  \right\vert +\left\vert \partial_{k}%
\phi\left(  k\right)  \right\vert \right)  .
\]
Then, we estimate%
\[
\left.
\begin{array}
[c]{c}%
\left\vert \mathcal{V}_{12}^{\pm}\left(  t\right)  \right\vert \leq C\sqrt{t}%
{\displaystyle\int\limits_{-\infty}^{\infty}}
\frac{\left\vert k-x\right\vert \left(  \left\vert \phi\left(  k\right)
\right\vert +\left\vert \partial_{k}\phi\left(  k\right)  \right\vert \right)
}{1+t\left(  k-x\right)  ^{2}}dk+Ct\sqrt{t}%
{\displaystyle\int\limits_{-\infty}^{\infty}}
\frac{\left(  k-x\right)  ^{2}\left(  \left\vert k-x\right\vert \left\vert
\phi\left(  k\right)  \right\vert +\left\vert k-x\right\vert ^{\frac{1}{2}%
}\left\Vert \partial_{k}\phi\right\Vert _{\mathbf{L}^{2}}\right)  }{\left(
1+t\left(  k-x\right)  ^{2}\right)  ^{2}}dk\\
\leq C\sqrt{t}%
{\displaystyle\int\limits_{-\infty}^{\infty}}
\frac{\left\vert k-x\right\vert \left(  \left\vert \phi\left(  k\right)
\right\vert +\left\vert \partial_{k}\phi\left(  k\right)  \right\vert \right)
}{1+t\left(  k-x\right)  ^{2}}dk+C\sqrt{t}\left\Vert \phi\right\Vert
_{\mathbf{H}^{1}}%
{\displaystyle\int\limits_{-\infty}^{\infty}}
\frac{\left\vert k-x\right\vert ^{\frac{1}{2}}}{1+t\left(  k-x\right)  ^{2}%
}dk+C%
{\displaystyle\int\limits_{-\infty}^{\infty}}
\frac{\left\vert \phi\left(  k\right)  \right\vert }{\left(  1+t\left(
k-x\right)  ^{2}\right)  ^{1/2}}dk.
\end{array}
\right.
\]
Hence, via Cauchy-Schwartz inequality we obtain%
\begin{equation}
\left.
\begin{array}
[c]{c}%
\left\vert \mathcal{V}_{12}^{\pm}\left(  t\right)  \right\vert \leq C\sqrt
{t}\left(
{\displaystyle\int\limits_{-\infty}^{\infty}}
\frac{\left(  k-x\right)  ^{2}}{\left(  1+t\left(  k-x\right)  ^{2}\right)
^{2}}dk\right)  ^{1/2}\left(
{\displaystyle\int\limits_{-\infty}^{\infty}}
\left(  \left\vert \phi\left(  k\right)  \right\vert ^{2}+\left\vert
\partial_{k}\phi\left(  k\right)  \right\vert ^{2}\right)  dk\right)  ^{1/2}\\
+C\left\Vert \phi\right\Vert _{\mathbf{H}^{1}}\sqrt{t}%
{\displaystyle\int\limits_{-\infty}^{\infty}}
\frac{\left\vert k\right\vert ^{\frac{1}{2}}}{1+tk^{2}}dk+C\left(
{\displaystyle\int\limits_{-\infty}^{\infty}}
\frac{dk}{1+tk^{2}}\right)  ^{\frac{1}{2}}\left\Vert \phi\right\Vert
_{\mathbf{L}^{2}}\\
\leq C\left\Vert \phi\right\Vert _{\mathbf{H}^{1}}\sqrt{t}\left(  \left(
{\displaystyle\int\limits_{-\infty}^{\infty}}
\frac{k^{2}}{\left(  1+tk^{2}\right)  ^{2}}dk\right)  ^{1/2}+%
{\displaystyle\int\limits_{-\infty}^{\infty}}
\frac{\left\vert k\right\vert ^{\frac{1}{2}}}{1+tk^{2}}dk\right)
+Ct^{-1/4}\left\Vert \phi\right\Vert _{\mathbf{L}^{2}}\leq Ct^{-1/4}\left\Vert
\phi\right\Vert _{\mathbf{H}^{1}}.
\end{array}
\right.  \label{4.3}%
\end{equation}
Using (\ref{4.2}) and (\ref{4.3}) in (\ref{4.4}) we get
\begin{equation}
\left\vert \mathcal{V}_{1}^{\pm}\left(  t\right)  -\left(  T\left(  x\right)
\mp1\right)  \phi\left(  x\right)  \right\vert \leq Ct^{-1/4}\left\Vert
\phi\right\Vert _{\mathbf{H}^{1}}. \label{4.5}%
\end{equation}
Next, we decompose $\mathcal{V}_{2}^{\pm}\left(  t\right)  $ as%
\[
\mathcal{V}_{2}^{\pm}\left(  t\right)  =R_{+}\left(  x\right)  \phi\left(
-x\right)  +\mathcal{V}_{21}\left(  t\right)  +\mathcal{V}_{22}\left(
t\right)  ,
\]
with%
\[
\mathcal{V}_{21}^{\pm}\left(  t\right)  =R_{+}\left(  x\right)  \left(
m_{+}\left(  tx,x\right)  -1\right)  \phi\left(  -x\right)
\]
and%
\[
\mathcal{V}_{22}^{\pm}\left(  t\right)  =\sqrt{\frac{it}{2\pi}}%
{\displaystyle\int\limits_{-\infty}^{\infty}}
e^{-\frac{it}{2}\left(  k-x\right)  ^{2}}\left(  \theta\left(  k\right)
R_{+}\left(  k\right)  m_{+}\left(  tx,k\right)  \phi\left(  -k\right)
-R_{+}\left(  x\right)  m_{+}\left(  tx,x\right)  \phi\left(  -x\right)
\right)  dk.
\]
Then, by using (\ref{3.4}), (\ref{3.2}), (\ref{3.5}), (\ref{3.12}),
(\ref{3.17}) and (\ref{4.8}), similarly to (\ref{4.5}) we deduce%
\begin{equation}
\left\vert \mathcal{V}_{2}^{\pm}\left(  t\right)  -R_{+}\left(  x\right)
\phi\left(  -x\right)  \right\vert \leq Ct^{-1/4}\left\Vert \phi\right\Vert
_{\mathbf{H}^{1}}. \label{4.6}%
\end{equation}
We now turn to $\mathcal{V}_{3}^{\pm}\left(  t\right)  .$ For $\mathcal{V}%
_{3}^{+}$ we write
\[
\mathcal{V}_{3}^{+}\left(  t\right)  =m_{+}\left(  tx,x\right)  \phi\left(
x\right)  +\mathcal{V}_{31}^{+}%
\]
with%
\[
\mathcal{V}_{31}^{+}=\sqrt{\frac{it}{2\pi}}%
{\displaystyle\int\limits_{-\infty}^{\infty}}
e^{-\frac{it}{2}\left(  k-x\right)  ^{2}}\left(  m_{+}\left(  tx,k\right)
\phi\left(  k\right)  -m_{+}\left(  tx,x\right)  \phi\left(  x\right)
\right)  dk.
\]
For $\mathcal{V}_{3}^{-}$ we have%
\begin{equation}
\mathcal{V}_{3}^{-}\left(  t\right)  =\phi\left(  x\right)  \sqrt{\frac
{it}{2\pi}}%
{\displaystyle\int\limits_{-\infty}^{\infty}}
e^{-\frac{it}{2}\left(  k-x\right)  ^{2}}\left(  \theta\left(  -k\right)
-\theta\left(  k\right)  \right)  dk+\mathcal{V}_{31}^{-}\left(  t\right)
+\mathcal{V}_{32}^{-}\left(  t\right)  +\mathcal{V}_{33}^{-}\left(  t\right)
, \label{4.13}%
\end{equation}
where%
\[
\mathcal{V}_{31}^{-}\left(  t\right)  =\left(  m_{+}\left(  tx,x\right)
-1\right)  \phi\left(  x\right)  \sqrt{\frac{it}{2\pi}}%
{\displaystyle\int\limits_{-\infty}^{\infty}}
e^{-\frac{it}{2}\left(  k-x\right)  ^{2}}\left(  \theta\left(  -k\right)
-\theta\left(  k\right)  \right)  dk
\]%
\[
\mathcal{V}_{32}^{-}\left(  t\right)  =-\sqrt{\frac{it}{2\pi}}%
{\displaystyle\int\limits_{0}^{\infty}}
e^{-\frac{it}{2}\left(  k-x\right)  ^{2}}\left(  m_{+}\left(  tx,k\right)
\phi\left(  k\right)  -m_{+}\left(  tx,x\right)  \phi\left(  x\right)
\right)  dk
\]
and%
\[
\mathcal{V}_{33}^{-}\left(  t\right)  =\sqrt{\frac{it}{2\pi}}%
{\displaystyle\int\limits_{-\infty}^{0}}
e^{-\frac{it}{2}\left(  k-x\right)  ^{2}}\left(  m_{+}\left(  tx,k\right)
\phi\left(  k\right)  -m_{+}\left(  tx,x\right)  \phi\left(  x\right)
\right)  dk.
\]
Similarly to (\ref{4.3}) we show that
\[
\left\vert \mathcal{V}_{31}^{+}\left(  t\right)  \right\vert \leq
Ct^{-1/4}\left\Vert \phi\right\Vert _{\mathbf{H}^{1}},
\]
and then,
\begin{equation}
\left\vert \mathcal{V}_{3}^{+}\left(  t\right)  -m_{+}\left(  tx,x\right)
\phi\left(  x\right)  \right\vert \leq Ct^{-1/4}\left\Vert \phi\right\Vert
_{\mathbf{H}^{1}}. \label{4.16}%
\end{equation}
Using (\ref{3.2}) we get%
\begin{equation}
\left\vert \mathcal{V}_{31}^{-}\left(  t\right)  \right\vert \leq C\left\vert
\left(  m_{+}\left(  tx,x\right)  -1\right)
{\displaystyle\int\limits_{-\sqrt{t}x}^{\sqrt{t}x}}
e^{-\frac{i}{2}k^{2}}dk\right\vert \left\Vert \phi\right\Vert _{\mathbf{H}%
^{1}}\leq\frac{C}{\sqrt{t}}\left\Vert \phi\right\Vert _{\mathbf{H}^{1}}.
\label{4.12}%
\end{equation}
Taking into account identity (\ref{4.7}) and integrating by parts in
$\mathcal{V}_{32}^{-}\left(  t\right)  $ we have
\begin{align*}
\mathcal{V}_{32}^{-}\left(  t\right)   &  =-\sqrt{\frac{it}{2\pi}}%
\frac{xe^{-\frac{it}{2}x^{2}}}{1-itx^{2}}\left(  m_{+}\left(  tx,0\right)
\phi\left(  0\right)  -m_{+}\left(  tx,x\right)  \phi\left(  x\right)  \right)
\\
&  +\sqrt{\frac{it}{2\pi}}%
{\displaystyle\int\limits_{0}^{\infty}}
\left(  k-x\right)  e^{-\frac{it}{2}\left(  k-x\right)  ^{2}}\partial
_{k}\left(  B\left(  m_{+}\left(  tx,k\right)  \phi\left(  k\right)
-m_{+}\left(  tx,x\right)  \phi\left(  x\right)  \right)  \right)  dk.
\end{align*}
Using (\ref{3.5}) and (\ref{4.8}), with $k=0,$ we estimate%
\[
\left\vert m_{+}\left(  tx,0\right)  \phi\left(  0\right)  -m_{+}\left(
tx,x\right)  \phi\left(  x\right)  \right\vert \leq C\left\vert x\right\vert
^{\frac{1}{2}}\left\Vert \phi\right\Vert _{\mathbf{H}^{1}}.
\]
Then,%
\[
\left\vert \sqrt{\frac{it}{2\pi}}\frac{xe^{-\frac{it}{2}x^{2}}}{1-itx^{2}%
}\left(  m_{+}\left(  tx,0\right)  \phi\left(  0\right)  -m_{+}\left(
tx,x\right)  \phi\left(  x\right)  \right)  \right\vert \leq C\left\Vert
\phi\right\Vert _{\mathbf{H}^{1}}\frac{\sqrt{t}\left\vert x\right\vert ^{3/2}%
}{\left(  1+tx^{2}\right)  ^{3/4}}\leq Ct^{-1/4}\left\Vert \phi\right\Vert
_{\mathbf{H}^{1}}.
\]
Moreover, similarly to (\ref{4.3}) we control
\[
\left\vert \sqrt{\frac{it}{2\pi}}%
{\displaystyle\int\limits_{0}^{\infty}}
\left(  k-x\right)  e^{-\frac{it}{2}\left(  k-x\right)  ^{2}}\partial
_{k}\left(  B\left(  m_{+}\left(  tx,k\right)  \phi\left(  k\right)
-m_{+}\left(  tx,x\right)  \phi\left(  x\right)  \right)  \right)
dk\right\vert \leq Ct^{-1/4}\left\Vert \phi\right\Vert _{\mathbf{H}^{1}},
\]
and thus,
\begin{equation}
\left\vert \mathcal{V}_{32}^{-}\left(  t\right)  \right\vert \leq
Ct^{-1/4}\left\Vert \phi\right\Vert _{\mathbf{H}^{1}}. \label{4.10}%
\end{equation}
Similarly we get%
\begin{equation}
\left\vert \mathcal{V}_{33}^{-}\left(  t\right)  \right\vert \leq
Ct^{-1/4}\left\Vert \phi\right\Vert _{\mathbf{H}^{1}}. \label{4.11}%
\end{equation}
Using (\ref{4.12}), (\ref{4.10}) and (\ref{4.11}) in (\ref{4.13}) we get%
\begin{equation}
\left\vert \mathcal{V}_{3}^{-}\left(  t\right)  -\phi\left(  x\right)
\sqrt{\frac{it}{2\pi}}%
{\displaystyle\int\limits_{-\infty}^{\infty}}
e^{-\frac{it}{2}\left(  k-x\right)  ^{2}}\left(  \theta\left(  -k\right)
-\theta\left(  k\right)  \right)  dk\right\vert \leq Ct^{-1/4}\left\Vert
\phi\right\Vert _{\mathbf{H}^{1}}. \label{4.14}%
\end{equation}
Introducing (\ref{4.5}), (\ref{4.6}), (\ref{4.16}) and (\ref{4.14}) into
(\ref{4.15}) we obtain%
\begin{equation}
\left.  \left\vert \mathcal{V}^{+}\left(  t\right)  \phi-T\left(  x\right)
\phi\left(  x\right)  -R_{+}\left(  x\right)  \phi\left(  -x\right)  -\left(
m_{+}\left(  tx,x\right)  -1\right)  \phi\left(  x\right)  \right\vert \leq
Ct^{-1/4}\left\Vert \phi\right\Vert _{\mathbf{H}^{1}}\right.  \label{4.17}%
\end{equation}
and%
\begin{equation}
\left\vert \mathcal{V}^{-}\left(  t\right)  \phi-T\left(  x\right)
\phi\left(  x\right)  -R_{+}\left(  x\right)  \phi\left(  -x\right)
-\phi\left(  x\right)  \left(  1+\sqrt{\frac{it}{2\pi}}%
{\displaystyle\int\limits_{-\infty}^{\infty}}
e^{-\frac{it}{2}\left(  k-x\right)  ^{2}}\left(  \theta\left(  -k\right)
-\theta\left(  k\right)  \right)  dk\right)  \right\vert \leq Ct^{-1/4}%
\left\Vert \phi\right\Vert _{\mathbf{H}^{1}}. \label{4.18}%
\end{equation}
Using (\ref{3.2}) and (\ref{4.8}) with $k=0,$ we get%
\[
\left.  \left\vert \left(  m_{+}\left(  tx,x\right)  -1\right)  \left(
\phi\left(  x\right)  -\phi\left(  0\right)  \right)  \right\vert \leq\frac
{C}{\sqrt{t}}\left\Vert \partial_{k}\phi\right\Vert _{\mathbf{L}^{2}}.\right.
\]
Also, by (\ref{3.2}) $\left\vert m_{+}\left(  tx,x\right)  -1\right\vert \leq
C.$ Using the last two inequalities in (\ref{4.17}), we arrive to (\ref{4.20})
in the case $x\geq0.$ Using (\ref{4.8}), with $k=0,$ to estimate%
\[
\left\vert \phi\left(  x\right)  -\phi\left(  0\right)  \right\vert \left\vert
1+\sqrt{\frac{it}{2\pi}}%
{\displaystyle\int\limits_{-\infty}^{\infty}}
e^{-\frac{it}{2}\left(  k-x\right)  ^{2}}\left(  \theta\left(  -k\right)
-\theta\left(  k\right)  \right)  dk\right\vert \leq Ct^{-\frac{1}{4}%
}\left\Vert \phi\right\Vert _{\mathbf{H}^{1}},
\]
and noting that%
\[
1+\sqrt{\frac{it}{2\pi}}%
{\displaystyle\int\limits_{-\infty}^{\infty}}
e^{-\frac{it}{2}\left(  k-x\right)  ^{2}}\left(  \theta\left(  -k\right)
-\theta\left(  k\right)  \right)  dk=\sqrt{\frac{2i}{\pi}}%
{\displaystyle\int\limits_{\sqrt{t}x}^{\infty}}
e^{-\frac{it}{2}k^{2}}dk,
\]
we deduce (\ref{4.20a}) in the case $x\geq0.$\

We consider now the case $x\leq0.$ Introducing (\ref{4.43}) into (\ref{2.10})
we have%
\begin{align*}
\mathcal{V}\left(  t\right)  \phi &  =\sqrt{\frac{it}{2\pi}}\int_{-\infty
}^{\infty}e^{-\frac{it}{2}\left(  k-\left\vert x\right\vert \right)  ^{2}%
}\theta\left(  k\right)  T\left(  k\right)  m_{-}\left(  tx,k\right)
\phi\left(  -k\right)  dk\\
&  +\sqrt{\frac{it}{2\pi}}\int_{-\infty}^{\infty}e^{-\frac{it}{2}\left(
k-\left\vert x\right\vert \right)  ^{2}}\theta\left(  k\right)  R_{-}\left(
k\right)  m_{-}\left(  tx,k\right)  \phi\left(  k\right)  dk\\
&  +\sqrt{\frac{it}{2\pi}}\int_{-\infty}^{\infty}e^{-\frac{it}{2}\left(
k-\left\vert x\right\vert \right)  ^{2}}\theta\left(  -k\right)  m_{-}\left(
tx,k\right)  \phi\left(  -k\right)  dk
\end{align*}
Therefore, proceeding similarly to the case of $x\geq0$, we obtain the result
for $x\leq0.$
\end{proof}

Next we estimate the derivative $\partial_{x}\mathcal{V}\left(  t\right)  .$

\begin{lemma}
\label{Lemma 4.2}Suppose that the estimates (\ref{3.5}) with $\delta>\frac
{1}{2},$ (\ref{3.6}) and (\ref{3.7}) with $\delta=0$ are true. Also, asume
that (\ref{3.4}), (\ref{3.2}), (\ref{3.8}), (\ref{3.3}), (\ref{3.12}) and
(\ref{3.17}), are satisfied. Moreover, if $a=1,$ let $\phi\in\mathbf{H}^{1}$
be odd and if $a=-1,$ suppose that $\phi$ is even. Moreover, suppose that
$\phi$ can be represented as $\phi=\mathcal{F}\psi,$ for some $\psi
\in\mathbf{H}^{0,1}.$ Then, the estimate
\begin{equation}
\left\Vert \partial_{x}\mathcal{V}\left(  t\right)  \phi\right\Vert
_{\mathbf{L}^{2}}\leq C\left\Vert \phi\right\Vert _{\mathbf{L}^{\infty}}%
\log\left\langle t\right\rangle +C\left\Vert \phi\right\Vert _{\mathbf{H}^{1}}
\label{4.66}%
\end{equation}
is true for all $t\geq1$.
\end{lemma}

\begin{proof}
We consider the case of $x\geq0.$ We depart from relation (\ref{4.19}). We
denote%
\[
\Theta\left(  tx,k\right)  =\theta\left(  k\right)  m_{+}\left(  tx,k\right)
\left(  T\left(  k\right)  \phi\left(  k\right)  +R_{+}\left(  k\right)
\phi\left(  -k\right)  \right)  +\theta\left(  -k\right)  m_{+}\left(
tx,k\right)  \phi\left(  k\right)  .
\]
Then,
\begin{equation}
\mathcal{V}\left(  t\right)  \phi=\sqrt{\frac{it}{2\pi}}\int_{-\infty}%
^{\infty}e^{-\frac{it}{2}\left(  k-x\right)  ^{2}}\Theta\left(  tx,k\right)
dk. \label{4.22}%
\end{equation}
Taking the derivative with respect to $x$ in the last relation we obtain%
\begin{equation}
\left.  \partial_{x}\mathcal{V}\left(  t\right)  \phi=-\sqrt{\frac{it}{2\pi}%
}\int_{-\infty}^{\infty}\partial_{k}e^{-\frac{it}{2}\left(  k-x\right)  ^{2}%
}\Theta\left(  tx,k\right)  dk+\sqrt{\frac{it}{2\pi}}\int_{-\infty}^{\infty
}e^{-\frac{it}{2}\left(  k-x\right)  ^{2}}\partial_{x}\left(  \Theta\left(
tx,k\right)  \right)  dk.\right.  \label{4.74}%
\end{equation}
Integrating by parts in the first term of the right-hand side of (\ref{4.74})
and using (\ref{7.60}) we get%
\begin{equation}
\partial_{x}\mathcal{V}\left(  t\right)  \phi=I_{1}+I_{2}, \label{4.65}%
\end{equation}
where%
\[
I_{1}=\sqrt{\frac{it}{2\pi}}\int_{-\infty}^{\infty}e^{-\frac{it}{2}\left(
x-k\right)  ^{2}}\partial_{k}\Theta\left(  tx,k\right)  dk
\]
and%
\[
I_{2}=\sqrt{\frac{it}{2\pi}}\int_{-\infty}^{\infty}e^{-\frac{it}{2}\left(
x-k\right)  ^{2}}\partial_{x}\left(  \Theta\left(  tx,k\right)  \right)  dk.
\]
We begin by estimating $I_{1}.$ We split $I_{1}$ as follows%
\begin{equation}
I_{1}=\sum_{j=1}^{4}I_{1j}, \label{4.36}%
\end{equation}
where%
\[
I_{1j}=\sqrt{\frac{it}{2\pi}}\int_{-\infty}^{\infty}e^{-\frac{it}{2}\left(
x-k\right)  ^{2}}\Theta_{j}\left(  k\right)  dk=\mathcal{V}_{0}\left(
t\right)  \Theta_{j},
\]
\ $j=1,2,$ and%
\[
I_{1j}=\sqrt{\frac{it}{2\pi}}\int_{-\infty}^{\infty}e^{-\frac{it}{2}\left(
x-k\right)  ^{2}}\Theta_{j}\left(  tx,k\right)  dk,
\]
\ $j=3,4,$ with%
\[
\Theta_{1}\left(  k\right)  =\theta\left(  k\right)  \left(  \left(
\partial_{k}T\left(  k\right)  \right)  \phi\left(  k\right)  +\partial
_{k}R_{+}\left(  k\right)  \phi\left(  -k\right)  \right)  ,
\]%
\[
\Theta_{2}\left(  k\right)  =\theta\left(  k\right)  T\left(  k\right)
\partial_{k}\phi\left(  k\right)  +\theta\left(  k\right)  R_{+}\left(
k\right)  \partial_{k}\phi\left(  -k\right)  +\theta\left(  -k\right)
\partial_{k}\phi\left(  k\right)  ,
\]%
\[
\Theta_{3}\left(  tx,k\right)  =\left(  \partial_{k}m_{+}\left(  tx,k\right)
\right)  \left(  \theta\left(  k\right)  T\left(  k\right)  \phi\left(
k\right)  +\theta\left(  k\right)  R_{+}\left(  k\right)  \phi\left(
-k\right)  +\theta\left(  -k\right)  \phi\left(  k\right)  \right)  +\left(
m_{+}\left(  tx,k\right)  -1\right)  \Theta_{1}\left(  k\right)  ,
\]
and%
\[
\Theta_{4}\left(  tx,k\right)  =\left(  m_{+}\left(  tx,k\right)  -1\right)
\Theta_{2}\left(  k\right)  .
\]
By using (\ref{2.11}) we estimate
\begin{equation}
\left\Vert I_{1j}\right\Vert _{\mathbf{L}^{2}\left(  \mathbb{R}^{+}\right)
}=\left\Vert \mathcal{V}_{0}\left(  t\right)  \Theta_{j}\right\Vert
_{\mathbf{L}^{2}\left(  \mathbb{R}^{+}\right)  }\leq C\left\Vert \Theta
_{j}\right\Vert _{\mathbf{L}^{2}}, \label{4.30}%
\end{equation}
for $j=1,2.$ Relations (\ref{3.8}) and (\ref{3.17}) imply
\begin{equation}
\left\Vert \Theta_{1}\right\Vert _{\mathbf{L}^{2}}\leq C\left(  \left\Vert
\partial_{k}T\right\Vert _{\mathbf{L}^{\infty}}+\left\Vert \partial
_{k}R\right\Vert _{\mathbf{L}^{\infty}}\right)  \left\Vert \phi\right\Vert
_{\mathbf{L}^{2}}\leq C\left\Vert \phi\right\Vert _{\mathbf{L}^{2}}.
\label{4.31}%
\end{equation}
Moreover, using (\ref{3.4}) we estimate
\begin{equation}
\left\Vert \Theta_{2}\right\Vert _{\mathbf{L}^{2}}\leq C\left(  1+\left\Vert
T\right\Vert _{\mathbf{L}^{\infty}}+\left\Vert R\right\Vert _{\mathbf{L}%
^{\infty}}\right)  \left\Vert \partial_{k}\phi\right\Vert _{\mathbf{L}^{2}%
}\leq C\left\Vert \partial_{k}\phi\right\Vert _{\mathbf{L}^{2}}. \label{4.32}%
\end{equation}
Hence, it follows from (\ref{4.30}), (\ref{4.31}), (\ref{4.32}) that%
\begin{equation}
\left\Vert I_{1j}\right\Vert _{\mathbf{L}^{2}\left(  \mathbb{R}^{+}\right)
}\leq C\left\Vert \phi\right\Vert _{\mathbf{H}^{1}}, \label{4.33}%
\end{equation}
for $j=1,2.$ We use (\ref{3.4}), (\ref{3.2}), (\ref{3.5}) with $\delta
>\frac{1}{2}$, (\ref{3.8}) and (\ref{3.17}) to control $\Theta_{3}.$ We obtain%
\begin{align*}
&  \left\vert \Theta_{3}\left(  tx,k\right)  \right\vert \leq C\left\vert
\partial_{k}m_{+}\left(  tx,k\right)  \right\vert \left(  \left\Vert
T\right\Vert _{\mathbf{L}^{\infty}}+\left\Vert R\right\Vert _{\mathbf{L}%
^{\infty}}+1\right)  \left(  \left\vert \phi\left(  k\right)  \right\vert
+\left\vert \phi\left(  -k\right)  \right\vert \right) \\
&  +C\left\vert m_{+}\left(  tx,k\right)  -1\right\vert \left(  \left\Vert
\partial_{k}T\right\Vert _{\mathbf{L}^{\infty}}+\left\Vert \partial
_{k}R\right\Vert _{\mathbf{L}^{\infty}}\right)  \left(  \left\vert \phi\left(
k\right)  \right\vert +\left\vert \phi\left(  -k\right)  \right\vert \right)
\\
&  \leq C\left\langle tx\right\rangle ^{-\delta}\left\langle k\right\rangle
^{-1}\left(  \left\vert \phi\left(  k\right)  \right\vert +\left\vert
\phi\left(  -k\right)  \right\vert \right)  .
\end{align*}
Hence, via Cauchy-Schwartz inequality we derive%
\begin{equation}
\left.
\begin{array}
[c]{c}%
\left\Vert I_{13}\right\Vert _{\mathbf{L}^{2}\left(  \mathbb{R}^{+}\right)
}=\left\Vert \sqrt{\frac{it}{2\pi}}\int_{-\infty}^{\infty}e^{-\frac{it}%
{2}\left(  x-k\right)  ^{2}}\Theta_{3}\left(  tx,k\right)  dk\right\Vert
_{\mathbf{L}^{2}\left(  \mathbb{R}^{+}\right)  }\leq Ct^{\frac{1}{2}%
}\left\Vert \left\langle tx\right\rangle ^{-\delta}\int_{-\infty}^{\infty
}\left\langle k\right\rangle ^{-1}\left(  \left\vert \phi\left(  k\right)
\right\vert +\left\vert \phi\left(  -k\right)  \right\vert \right)
dk\right\Vert _{\mathbf{L}^{2}\left(  \mathbb{R}^{+}\right)  }\\
\leq Ct^{\frac{1}{2}}\left\Vert \left\langle tx\right\rangle ^{-\delta
}\right\Vert _{\mathbf{L}^{2}\left(  \mathbb{R}^{+}\right)  }\left\Vert
\left\langle k\right\rangle ^{-1}\right\Vert _{\mathbf{L}^{2}}\left\Vert
\phi\right\Vert _{\mathbf{L}^{2}}\leq C\left\Vert \phi\right\Vert
_{\mathbf{L}^{2}}.
\end{array}
\right.  \label{4.34}%
\end{equation}
In the same manner,\ by using (\ref{3.4}) and (\ref{3.2}) we find%
\begin{align*}
\left\vert \Theta_{4}\left(  x,k\right)  \right\vert  &  \leq C\left\vert
m_{+}\left(  tx,k\right)  -1\right\vert \left(  \left\Vert T\right\Vert
_{\mathbf{L}^{\infty}}+\left\Vert R\right\Vert _{\mathbf{L}^{\infty}%
}+1\right)  \left(  \left\vert \partial_{k}\phi\left(  k\right)  \right\vert
+\left\vert \partial_{k}\phi\left(  -k\right)  \right\vert \right) \\
&  \leq C\left\langle tx\right\rangle ^{-1}\left\langle k\right\rangle
^{-1}\left(  \left\vert \partial_{k}\phi\left(  k\right)  \right\vert
+\left\vert \partial_{k}\phi\left(  -k\right)  \right\vert \right)  ,
\end{align*}
and thus,
\begin{equation}
\left.
\begin{array}
[c]{c}%
\left\Vert I_{14}\right\Vert _{\mathbf{L}^{2}\left(  \mathbb{R}^{+}\right)
}=Ct^{\frac{1}{2}}\left\Vert \int_{-\infty}^{\infty}e^{-\frac{it}{2}\left(
x-k\right)  ^{2}}\Theta_{4}\left(  tx,k\right)  dk\right\Vert _{\mathbf{L}%
^{2}\left(  \mathbb{R}^{+}\right)  }\\
\leq Ct^{\frac{1}{2}}\left\Vert \left\langle tx\right\rangle ^{-1}%
\int_{-\infty}^{\infty}\left\langle k\right\rangle ^{-1}\left(  \left\vert
\partial_{k}\phi\left(  k\right)  \right\vert +\left\vert \partial_{k}%
\phi\left(  -k\right)  \right\vert \right)  dk\right\Vert _{\mathbf{L}%
^{2}\left(  \mathbb{R}^{+}\right)  }\\
\leq Ct^{\frac{1}{2}}\left\Vert \left\langle tx\right\rangle ^{-1}\right\Vert
_{\mathbf{L}^{2}\left(  \mathbb{R}^{+}\right)  }\left\Vert \left\langle
k\right\rangle ^{-1}\right\Vert _{\mathbf{L}^{2}}\left\Vert \partial_{k}%
\phi\right\Vert _{\mathbf{L}^{2}}\leq C\left\Vert \partial_{k}\phi\right\Vert
_{\mathbf{L}^{2}}.
\end{array}
\right.  \label{4.35}%
\end{equation}
Therefore, from (\ref{4.36}), (\ref{4.33}), (\ref{4.34}) and (\ref{4.35}) we
conclude that
\begin{equation}
\left\Vert I_{1}\right\Vert _{\mathbf{L}^{2}\left(  \mathbb{R}^{+}\right)
}\leq C\left\Vert \phi\right\Vert _{\mathbf{H}^{1}}. \label{4.64}%
\end{equation}

Next, we turn to the term $I_{2}$. We have%
\begin{align*}
I_{2}  &  =\sqrt{\frac{i}{2\pi}}t^{\frac{3}{2}}\int_{0}^{2x}e^{-\frac{it}%
{2}\left(  x-k\right)  ^{2}}\left(  T\left(  k\right)  \phi\left(  k\right)
+R_{+}\left(  k\right)  \phi\left(  -k\right)  \right)  \left(  \partial
_{x}m_{+}\right)  \left(  tx,k\right)  dk\\
&  +\sqrt{\frac{i}{2\pi}}t^{\frac{3}{2}}\int_{2x}^{\infty}e^{-\frac{it}%
{2}\left(  x-k\right)  ^{2}}\left(  T\left(  k\right)  \phi\left(  k\right)
+R_{+}\left(  k\right)  \phi\left(  -k\right)  \right)  \left(  \partial
_{x}m_{+}\right)  \left(  tx,k\right)  dk\\
&  +\sqrt{\frac{i}{2\pi}}t^{\frac{3}{2}}\int_{-\infty}^{0}e^{-\frac{it}%
{2}\left(  x-k\right)  ^{2}}\left(  \partial_{x}m_{+}\right)  \left(
tx,k\right)  \phi\left(  k\right)  dk.
\end{align*}
Making the change of variables $k=2x-k^{\prime}$ in the third integral of the
right-hand side of the last relation we decompose%
\begin{equation}
I_{2}=I_{21}+I_{22}+I_{23}, \label{4.62}%
\end{equation}
where%
\[
I_{21}=\sqrt{\frac{i}{2\pi}}t^{\frac{3}{2}}\int_{0}^{2x}e^{-\frac{it}%
{2}\left(  x-k\right)  ^{2}}\Theta_{5}\left(  t,x,k\right)  dk
\]
with%
\[
\Theta_{5}\left(  t,x,k\right)  =\left(  T\left(  k\right)  \phi\left(
k\right)  +R_{+}\left(  k\right)  \phi\left(  -k\right)  \right)  \left(
\partial_{x}m_{+}\right)  \left(  tx,k\right)  ,
\]
and%
\[
I_{22}=\sqrt{\frac{i}{2\pi}}t^{\frac{3}{2}}\int_{2x}^{\infty}e^{-\frac{it}%
{2}\left(  x-k\right)  ^{2}}\Theta_{6}\left(  t,x,k\right)  dk,
\]
with%
\[
\Theta_{6}\left(  t,x,k\right)  =\left(  \partial_{x}m_{+}\right)  \left(
tx,2x-k\right)  \left(  \phi\left(  2x-k\right)  -\phi\left(  -k\right)
\right)
\]
and%
\[
I_{23}=\sqrt{\frac{i}{2\pi}}t^{\frac{3}{2}}\int_{2x}^{\infty}e^{-\frac{it}%
{2}\left(  x-k\right)  ^{2}}\Theta_{7}\left(  t,x,k\right)  dk,
\]
with%
\[
\Theta_{7}\left(  t,x,k\right)  =\left(  T\left(  k\right)  \phi\left(
k\right)  +R_{+}\left(  k\right)  \phi\left(  -k\right)  \right)  \left(
\partial_{x}m_{+}\right)  \left(  tx,k\right)  +\left(  \partial_{x}%
m_{+}\right)  \left(  tx,2x-k\right)  \phi\left(  -k\right)  .
\]
Using (\ref{3.4}) and (\ref{3.6}) with $\delta=0$ we get%
\[
\left\vert \Theta_{5}\left(  t,x,k\right)  \right\vert \leq C\left\vert
\left(  \partial_{x}m_{+}\right)  \left(  tx,k\right)  \right\vert \left(
\left\Vert T\right\Vert _{\mathbf{L}^{\infty}}+\left\Vert R\right\Vert
_{\mathbf{L}^{\infty}}\right)  \left(  \left\vert \phi\left(  k\right)
\right\vert +\left\vert \phi\left(  -k\right)  \right\vert \right)  \leq
C\left\langle tx\right\rangle ^{-2}\left\Vert \phi\right\Vert _{\mathbf{L}%
^{\infty}}.
\]
Thus, we obtain
\[
\left\vert I_{21}\right\vert \leq Ct^{\frac{3}{2}}\int_{0}^{2x}\left\vert
\Theta_{5}\left(  t,x,k\right)  \right\vert dk\leq Ct^{\frac{3}{2}}\left\vert
x\right\vert \left\langle tx\right\rangle ^{-2}\left\Vert \phi\right\Vert
_{\mathbf{L}^{\infty}},
\]
and therefore
\begin{equation}
\left\Vert I_{21}\right\Vert _{\mathbf{L}^{2}\left(  \mathbb{R}^{+}\right)
}\leq Ct^{\frac{3}{2}}\left\Vert \phi\right\Vert _{\mathbf{L}^{\infty}%
}\left\Vert \left\vert x\right\vert \left\langle tx\right\rangle
^{-2}\right\Vert _{\mathbf{L}^{2}\left(  \mathbb{R}^{+}\right)  }\leq
C\left\Vert \phi\right\Vert _{\mathbf{L}^{\infty}}. \label{4.59}%
\end{equation}
By Plancharel's theorem we have%
\[
\left\Vert \phi\left(  2x-k\right)  -\phi\left(  -k\right)  \right\Vert
_{\mathbf{L}^{2}}=\left\Vert \left(  e^{2ixy}-1\right)  \left(  \mathcal{F}%
_{0}\phi\right)  \left(  y\right)  \right\Vert _{\mathbf{L}^{2}}\leq
C\left\vert x\right\vert \left\Vert y\left(  \mathcal{F}_{0}\phi\right)
\left(  y\right)  \right\Vert _{\mathbf{L}^{2}}=C\left\vert x\right\vert
\left\Vert \partial_{k}\phi\right\Vert _{\mathbf{L}^{2}}.
\]
Then, using (\ref{3.6}) with $\delta=0,$ via Cauchy-Schwartz inequality it
follows%
\begin{align*}
\left\vert I_{22}\right\vert  &  \leq Ct^{\frac{3}{2}}\int_{2x}^{\infty
}\left\vert \Theta_{6}\left(  t,x,k\right)  \right\vert dk\leq Ct^{\frac{3}%
{2}}\int_{2x}^{\infty}\left\vert \left(  \partial_{x}m_{+}\right)  \left(
tx,2x-k\right)  \right\vert \left\vert \phi\left(  2x-k\right)  -\phi\left(
-k\right)  \right\vert dk\\
&  \leq Ct^{\frac{3}{2}}\left\langle tx\right\rangle ^{-2}\int_{2x}^{\infty
}\left\langle k-2x\right\rangle ^{-1}\left\vert \phi\left(  2x-k\right)
-\phi\left(  -k\right)  \right\vert dk\\
&  \leq Ct^{\frac{3}{2}}\left\langle tx\right\rangle ^{-2}\left\Vert
\left\langle k\right\rangle ^{-1}\right\Vert _{\mathbf{L}^{2}}\left\Vert
\phi\left(  2x-k\right)  -\phi\left(  -k\right)  \right\Vert _{\mathbf{L}^{2}%
}\leq Ct^{\frac{1}{2}}\left\langle tx\right\rangle ^{-1}\left\Vert
\partial_{k}\phi\right\Vert _{\mathbf{L}^{2}}.
\end{align*}
Hence
\begin{equation}
\left\Vert I_{22}\right\Vert _{\mathbf{L}^{2}\left(  \mathbb{R}^{+}\right)
}\leq Ct^{\frac{1}{2}}\left\Vert \partial_{k}\phi\right\Vert _{\mathbf{L}^{2}%
}\left\Vert \left\langle tx\right\rangle ^{-1}\right\Vert _{\mathbf{L}%
^{2}\left(  \mathbb{R}^{+}\right)  }\leq C\left\Vert \partial_{k}%
\phi\right\Vert _{\mathbf{L}^{2}}. \label{4.60}%
\end{equation}
Integrating by parts in $I_{23}$ we find%
\begin{equation}
\left.
\begin{array}
[c]{c}%
I_{23}=-\frac{t^{\frac{1}{2}}}{\sqrt{2\pi i}}%
{\displaystyle\int\limits_{2x}^{\infty}}
\partial_{k}e^{-\frac{it}{2}\left(  x-k\right)  ^{2}}\left(  k-x\right)
^{-1}\Theta_{7}\left(  t,x,k\right)  dk=\frac{t^{\frac{1}{2}}}{\sqrt{2\pi i}%
}\dfrac{e^{-\frac{it}{2}x^{2}}\Theta_{7}\left(  t,x,2x\right)  }{x}\\
-\frac{t^{\frac{1}{2}}}{\sqrt{2\pi i}}%
{\displaystyle\int\limits_{2x}^{\infty}}
e^{-\frac{it}{2}\left(  x-k\right)  ^{2}}\left(  k-x\right)  ^{-2}\Theta
_{7}\left(  t,x,k\right)  dk+\frac{t^{\frac{1}{2}}}{\sqrt{2\pi i}}%
{\displaystyle\int\limits_{2x}^{\infty}}
e^{-\frac{it}{2}\left(  x-k\right)  ^{2}}\left(  k-x\right)  ^{-1}\partial
_{k}\Theta_{7}\left(  t,x,k\right)  dk,
\end{array}
\right.  \label{4.56}%
\end{equation}
Let us now estimate $\Theta_{7}\left(  t,x,k\right)  .$ First, observe that by
(\ref{3.7}) with $\delta=0$%
\begin{equation}
\left\vert \left(  \partial_{x}m_{+}\right)  \left(  tx,2x-k\right)  -\left(
\partial_{x}m_{+}\right)  \left(  tx,k\right)  \right\vert =\left\vert
\int_{k}^{2x-k}\left(  \partial_{k}\partial_{x}m_{+}\right)  \left(
tx,k^{\prime}\right)  dk^{\prime}\right\vert \leq C\left\vert k-x\right\vert
\left\langle tx\right\rangle ^{-1}\leq C\frac{\left\vert k-x\right\vert
}{\left\langle k-x\right\rangle }\left\langle tx\right\rangle ^{-1},
\label{4.75}%
\end{equation}
for $\left\vert k-x\right\vert \leq1.$ Moreover, for $\left\vert
k-x\right\vert \geq1,$ by (\ref{3.6}) with $\delta=0$ we have%
\begin{equation}
\left\vert \left(  \partial_{x}m_{+}\right)  \left(  tx,2x-k\right)  -\left(
\partial_{x}m_{+}\right)  \left(  tx,k\right)  \right\vert \leq\left\vert
\left(  \partial_{x}m_{+}\right)  \left(  tx,2x-k\right)  \right\vert
+\left\vert \left(  \partial_{x}m_{+}\right)  \left(  tx,k\right)  \right\vert
\leq C\frac{\left\vert k-x\right\vert }{\left\langle k-x\right\rangle
}\left\langle tx\right\rangle ^{-2}. \label{4.76}%
\end{equation}
Thus, from (\ref{4.75}) and (\ref{4.76})
\begin{equation}
\left\vert \left(  \partial_{x}m_{+}\right)  \left(  tx,2x-k\right)  -\left(
\partial_{x}m_{+}\right)  \left(  tx,k\right)  \right\vert \leq C\frac
{\left\vert k-x\right\vert }{\left\langle k-x\right\rangle }\left\langle
tx\right\rangle ^{-1}. \label{4.39}%
\end{equation}
If $a=1,$ $\phi\ $is odd. When $a=-1,$ $\phi\ $is supposed to be even. Then,
\begin{equation}
\Theta_{7}\left(  t,x,k\right)  =\left(  \left(  T\left(  k\right)
\mp1\right)  \phi\left(  k\right)  +R_{+}\left(  k\right)  \phi\left(
-k\right)  \right)  \left(  \partial_{x}m_{+}\right)  \left(  tx,k\right)
+\left(  \left(  \partial_{x}m_{+}\right)  \left(  tx,2x-k\right)  -\left(
\partial_{x}m_{+}\right)  \left(  tx,k\right)  \right)  \phi\left(  -k\right)
, \label{4.37}%
\end{equation}
for $a=\pm1.$ By (\ref{3.4}), (\ref{3.6}) with $\delta=0,$ (\ref{3.3}) and
(\ref{3.12}) we estimate
\begin{equation}
\left\vert \left(  \partial_{x}m_{+}\right)  \left(  tx,k\right)  \right\vert
\left(  \left\vert T\left(  k\right)  \mp1\right\vert \left\vert \phi\left(
k\right)  \right\vert +\left\vert R_{+}\left(  k\right)  \right\vert
\left\vert \phi\left(  -k\right)  \right\vert \right)  \leq C\left\langle
tx\right\rangle ^{-1}\left\vert k\right\vert \left(  \left\vert \phi\left(
k\right)  \right\vert +\left\vert \phi\left(  -k\right)  \right\vert \right)
. \label{4.40}%
\end{equation}
From (\ref{4.39}) we see that%
\begin{equation}
\left\vert \left(  \partial_{x}m_{+}\right)  \left(  tx,2x-k\right)
-\partial_{x}m_{+}\left(  tx,k\right)  \right\vert \left\vert \phi\left(
k\right)  \right\vert \leq C\left\vert k-x\right\vert \left\langle
tx\right\rangle ^{-1}\left\vert \phi\left(  k\right)  \right\vert .
\label{4.41}%
\end{equation}
Thus, by (\ref{4.37}), (\ref{4.40}) and (\ref{4.41}) we deduce%
\begin{align*}
\left\vert \Theta_{7}\left(  t,x,k\right)  \right\vert  &  \leq\left\vert
\left(  \partial_{x}m_{+}\right)  \left(  tx,k\right)  \right\vert \left(
\left\vert T\left(  k\right)  \mp1\right\vert \left\vert \phi\left(  k\right)
\right\vert +\left\vert R_{+}\left(  k\right)  \right\vert \left\vert
\phi\left(  -k\right)  \right\vert \right)  +\left\vert \left(  \partial
_{x}m_{+}\right)  \left(  tx,2x-k\right)  -\partial_{x}m_{+}\left(
tx,k\right)  \right\vert \left\vert \phi\left(  -k\right)  \right\vert \\
&  \leq C\left\langle tx\right\rangle ^{-1}\left(  \left\vert k\right\vert
+\left\vert k-x\right\vert \right)  \left(  \left\vert \phi\left(  k\right)
\right\vert +\left\vert \phi\left(  -k\right)  \right\vert \right)  ,
\end{align*}
for $a=\pm1.$ Therefore, we get
\begin{align*}
&  \left\vert \frac{t^{\frac{1}{2}}}{\sqrt{2\pi i}}\frac{e^{-\frac{it}{2}%
x^{2}}\Theta_{7}\left(  t,x,2x\right)  }{x}\right\vert +\left\vert
\frac{t^{\frac{1}{2}}}{\sqrt{2\pi i}}\int_{2x}^{\infty}e^{-\frac{it}{2}\left(
x-k\right)  ^{2}}\left(  k-x\right)  ^{-2}\Theta_{7}\left(  t,x,k\right)
dk\right\vert \\
&  \leq Ct^{\frac{1}{2}}\frac{1}{x}\left\vert \Theta_{7}\left(  t,x,2x\right)
\right\vert +Ct^{\frac{1}{2}}%
{\displaystyle\int\limits_{2x}^{\infty}}
\left(  k-x\right)  ^{-2}\left\vert \Theta_{7}\left(  t,x,k\right)
\right\vert dk\\
&  \leq Ct^{\frac{1}{2}}\left\langle tx\right\rangle ^{-1}\left\Vert
\phi\right\Vert _{\mathbf{L}^{\infty}}+Ct^{\frac{1}{2}}\left\langle
tx\right\rangle ^{-1}%
{\displaystyle\int\limits_{2x}^{\infty}}
\left(  k-x\right)  ^{-1}\left(  \left\vert \phi\left(  k\right)  \right\vert
+\left\vert \phi\left(  -k\right)  \right\vert \right)  dk\\
&  \leq Ct^{\frac{1}{2}}\left\langle tx\right\rangle ^{-1}\left(  \left\Vert
\phi\right\Vert _{\mathbf{L}^{\infty}}+\left\Vert \phi\right\Vert
_{\mathbf{L}^{2}}\right)  +Ct^{\frac{1}{2}}\left\langle tx\right\rangle
^{-1}\left\Vert \phi\right\Vert _{\mathbf{L}^{\infty}}\log\left(  1+\left\vert
x\right\vert ^{-1}\right)  .
\end{align*}
Hence,%
\begin{equation}
\left.
\begin{array}
[c]{c}%
\left\Vert \frac{t^{\frac{1}{2}}}{\sqrt{2\pi i}}\dfrac{e^{-\frac{it}{2}x^{2}%
}\Theta_{7}\left(  t,x,2x\right)  }{x}\right\Vert _{\mathbf{L}^{2}\left(
\mathbb{R}^{+}\right)  }+\left\Vert \frac{t^{\frac{1}{2}}}{\sqrt{2\pi i}}%
{\displaystyle\int\limits_{2x}^{\infty}}
e^{-\frac{it}{2}\left(  x-k\right)  ^{2}}\left(  k-x\right)  ^{-2}\Theta
_{7}\left(  t,x,k\right)  dk\right\Vert _{\mathbf{L}^{2}\left(  \mathbb{R}%
^{+}\right)  }\\
\leq Ct^{\frac{1}{2}}\left(  \left\Vert \phi\right\Vert _{\mathbf{L}^{\infty}%
}+\left\Vert \phi\right\Vert _{\mathbf{L}^{2}}\right)  \left\Vert \left\langle
tx\right\rangle ^{-1}\right\Vert _{\mathbf{L}^{2}\left(  \mathbb{R}%
^{+}\right)  }+Ct^{\frac{1}{2}}\left\Vert \phi\right\Vert _{\mathbf{L}%
^{\infty}}\left\Vert \left\langle tx\right\rangle ^{-1}\log\left(
1+\left\vert x\right\vert ^{-1}\right)  \right\Vert _{\mathbf{L}^{2}\left(
\mathbb{R}^{+}\right)  }\leq C\left\Vert \phi\right\Vert _{\mathbf{L}^{\infty
}}\log\left\langle t\right\rangle +C\left\Vert \phi\right\Vert _{\mathbf{L}%
^{2}}.
\end{array}
\right.  \label{4.57}%
\end{equation}
Derivating (\ref{4.37}) with respect to $k$ we get%
\begin{equation}
\left.
\begin{array}
[c]{c}%
\partial_{k}\Theta_{7}\left(  t,x,k\right)  =\left(  \left(  T\left(
k\right)  \mp1\right)  \phi\left(  k\right)  +R_{+}\left(  k\right)
\phi\left(  -k\right)  \right)  \left(  \partial_{k}\partial_{x}m_{+}\right)
\left(  tx,k\right) \\
+\left(  \partial_{k}T\left(  k\right)  \phi\left(  k\right)  +\partial
_{k}R_{+}\left(  k\right)  \phi\left(  -k\right)  \right)  \left(
\partial_{x}m_{+}\right)  \left(  tx,k\right) \\
+\left(  \left(  T\left(  k\right)  \mp1\right)  \partial_{k}\phi\left(
k\right)  -R_{+}\left(  k\right)  \left(  \partial_{k}\phi\right)  \left(
-k\right)  \right)  \left(  \partial_{x}m_{+}\right)  \left(  tx,k\right) \\
-\left(  \left(  \partial_{x}m_{+}\right)  \left(  tx,2x-k\right)  -\left(
\partial_{x}m_{+}\right)  \left(  tx,k\right)  \right)  \left(  \partial
_{k}\phi\right)  \left(  -k\right) \\
-\left(  \left(  \partial_{k}\partial_{x}m_{+}\right)  \left(  tx,2x-k\right)
+\left(  \partial_{k}\partial_{x}m_{+}\right)  \left(  tx,k\right)  \right)
\phi\left(  -k\right)  .
\end{array}
\right.  \label{4.55}%
\end{equation}
Using (\ref{3.4}) and (\ref{3.7}) with $\delta=0$ we get%
\begin{equation}
\left\vert \left(  \left(  T\left(  k\right)  \mp1\right)  \phi\left(
k\right)  +R_{+}\left(  k\right)  \phi\left(  -k\right)  \right)  \left(
\partial_{k}\partial_{x}m_{+}\right)  \left(  tx,k\right)  \right\vert \leq
C\left\langle tx\right\rangle ^{-1}\left(  \left\vert \phi\left(  k\right)
\right\vert +\left\vert \phi\left(  -k\right)  \right\vert \right)  .
\label{4.50}%
\end{equation}
From (\ref{3.6}) with $\delta=0,$ (\ref{3.8}) and (\ref{3.17}) we derive%
\begin{equation}
\left\vert \left(  \partial_{k}T\left(  k\right)  \phi\left(  k\right)
+\partial_{k}R_{+}\left(  k\right)  \phi\left(  -k\right)  \right)  \left(
\partial_{x}m_{+}\right)  \left(  tx,k\right)  \right\vert \leq C\left\langle
tx\right\rangle ^{-1}\left(  \left\vert \phi\left(  k\right)  \right\vert
+\left\vert \phi\left(  -k\right)  \right\vert \right)  . \label{4.51}%
\end{equation}
By (\ref{3.4}), (\ref{3.6}) with $\delta=0,$ (\ref{3.3}) and (\ref{3.12})
\begin{equation}
\left\vert \left(  \left(  T\left(  k\right)  \mp1\right)  \partial_{k}%
\phi\left(  k\right)  -R_{+}\left(  k\right)  \left(  \partial_{k}\phi\right)
\left(  -k\right)  \right)  \left(  \partial_{x}m_{+}\right)  \left(
tx,k\right)  \right\vert \leq C\left\vert k\right\vert \left\langle
tx\right\rangle ^{-1}\left\langle k\right\rangle ^{-1}\left(  \left\vert
\partial_{k}\phi\left(  k\right)  \right\vert +\left\vert \partial_{k}%
\phi\left(  -k\right)  \right\vert \right)  . \label{4.52}%
\end{equation}
From (\ref{4.39}) we get%
\begin{equation}
\left\vert \left(  \left(  \partial_{x}m_{+}\right)  \left(  tx,2x-k\right)
-\left(  \partial_{x}m_{+}\right)  \left(  tx,k\right)  \right)  \partial
_{k}\phi\left(  -k\right)  \right\vert \leq C\frac{\left\vert k-x\right\vert
}{\left\langle k-x\right\rangle }\left\langle tx\right\rangle ^{-1}\left\vert
\partial_{k}\phi\left(  k\right)  \right\vert . \label{4.53}%
\end{equation}
Finally, by (\ref{3.7}) with $\delta=0$ we estimate%
\begin{equation}
\left\vert \left(  \left(  \partial_{k}\partial_{x}m_{+}\right)  \left(
tx,2x-k\right)  +\left(  \partial_{k}\partial_{x}m_{+}\right)  \left(
tx,k\right)  \right)  \phi\left(  -k\right)  \right\vert \leq C\left\langle
tx\right\rangle ^{-1}\left\vert \phi\left(  -k\right)  \right\vert .
\label{4.54}%
\end{equation}
Using (\ref{4.50})-(\ref{4.54}) in (\ref{4.55}) we deduce
\[
\left.  \left\vert \partial_{k}\Theta_{7}\left(  t,x,k\right)  \right\vert
\leq C\left\langle tx\right\rangle ^{-1}\left(  \left\vert \phi\left(
k\right)  \right\vert +\left\vert \phi\left(  -k\right)  \right\vert \right)
+C\left(  \frac{\left\vert k\right\vert }{\left\langle k\right\rangle }%
+\frac{\left\vert k-x\right\vert }{\left\langle k-x\right\rangle }\right)
\left\langle tx\right\rangle ^{-1}\left(  \left\vert \partial_{k}\phi\left(
k\right)  \right\vert +\left\vert \partial_{k}\phi\left(  -k\right)
\right\vert \right)  .\right.
\]
Hence,
\begin{align*}
&  \left\vert \frac{t^{\frac{1}{2}}}{\sqrt{2\pi i}}\int_{2x}^{\infty}%
e^{-\frac{it}{2}\left(  x-k\right)  ^{2}}\left(  k-x\right)  ^{-1}\partial
_{k}\Theta_{7}\left(  t,x,k\right)  dk\right\vert \\
&  \leq Ct^{\frac{1}{2}}\left\langle tx\right\rangle ^{-1}\int_{2x}^{\infty
}\left(  k-x\right)  ^{-1}\left(  \left\vert \phi\left(  k\right)  \right\vert
+\left\vert \phi\left(  -k\right)  \right\vert \right)  dk\\
&  +Ct^{\frac{1}{2}}\left\langle tx\right\rangle ^{-1}\int_{2x}^{\infty
}\left(  k-x\right)  ^{-1}\left(  \frac{\left\vert k\right\vert }{\left\langle
k\right\rangle }+\frac{\left\vert k-x\right\vert }{\left\langle
k-x\right\rangle }\right)  \left(  \left\vert \partial_{k}\phi\left(
k\right)  \right\vert +\left\vert \partial_{k}\phi\left(  -k\right)
\right\vert \right)  dk\\
&  \leq Ct^{\frac{1}{2}}\left\langle tx\right\rangle ^{-1}\left(  \left\Vert
\phi\right\Vert _{\mathbf{L}^{\infty}}\log\left(  1+\left\vert x\right\vert
^{-1}\right)  +\left\Vert \phi\right\Vert _{\mathbf{H}^{1}}\right)  ,
\end{align*}
and then
\begin{equation}
\left.
\begin{array}
[c]{c}%
\left\Vert \frac{t^{\frac{1}{2}}}{\sqrt{2\pi i}}\int_{2x}^{\infty}%
e^{-\frac{it}{2}\left(  x-k\right)  ^{2}}\left(  k-x\right)  ^{-1}\partial
_{k}\Theta_{7}\left(  t,x,k\right)  dk\right\Vert _{\mathbf{L}^{2}\left(
\mathbb{R}^{+}\right)  }\\
\leq Ct^{\frac{1}{2}}\left\Vert \phi\right\Vert _{\mathbf{H}^{1}}\left\Vert
\left\langle tx\right\rangle ^{-1}\right\Vert _{\mathbf{L}^{2}\left(
\mathbb{R}^{+}\right)  }+Ct^{\frac{1}{2}}\left\Vert \phi\right\Vert
_{\mathbf{L}^{\infty}}\left\Vert \left\langle tx\right\rangle ^{-1}\log\left(
1+\left\vert x\right\vert ^{-1}\right)  \right\Vert _{\mathbf{L}^{2}\left(
\mathbb{R}^{+}\right)  }\\
\leq C\left(  \left\Vert \phi\right\Vert _{\mathbf{L}^{\infty}}\log
\left\langle t\right\rangle +\left\Vert \phi\right\Vert _{\mathbf{H}^{1}%
}\right)  .
\end{array}
\right.  \label{4.58}%
\end{equation}
Using (\ref{4.57}) and (\ref{4.58}) in (\ref{4.56}) we arrive to
\begin{equation}
\left.  \left\Vert I_{23}\right\Vert _{\mathbf{L}^{2}\left(  \mathbb{R}%
^{+}\right)  }\leq C\left(  \left\Vert \phi\right\Vert _{\mathbf{L}^{\infty}%
}\log\left\langle t\right\rangle +\left\Vert \phi\right\Vert _{\mathbf{H}^{1}%
}\right)  \right.  \label{4.61}%
\end{equation}
Introducing the estimates (\ref{4.59}), (\ref{4.60}) and (\ref{4.61}) into
(\ref{4.62}) we obtain%
\begin{equation}
\left\Vert I_{2}\right\Vert _{\mathbf{L}^{2}\left(  \mathbb{R}^{+}\right)
}\leq C\left(  \left\Vert \phi\right\Vert _{\mathbf{L}^{\infty}}%
\log\left\langle t\right\rangle +\left\Vert \phi\right\Vert _{\mathbf{H}^{1}%
}\right)  , \label{4.63}%
\end{equation}
Therefore, from (\ref{4.65}), (\ref{4.64}) and (\ref{4.63}) we control
$\left\Vert \partial_{x}\mathcal{V}\left(  t\right)  \phi\right\Vert
_{\mathbf{L}^{2}\left(  \mathbb{R}^{+}\right)  }$ by $C\left(  \left\Vert
\phi\right\Vert _{\mathbf{L}^{\infty}}\log\left\langle t\right\rangle
+\left\Vert \phi\right\Vert _{\mathbf{H}^{1}}\right)  .$ Proceeding similarly
we estimate%
\[
\left\Vert \partial_{x}\mathcal{V}\left(  t\right)  \phi\right\Vert
_{\mathbf{L}^{2}\left(  \mathbb{R}^{-}\right)  }\leq C\left(  \left\Vert
\phi\right\Vert _{\mathbf{L}^{\infty}}\log\left\langle t\right\rangle
+\left\Vert \phi\right\Vert _{\mathbf{H}^{1}}\right)  .
\]
Hence, we attain (\ref{4.66}).
\end{proof}

\section{\label{S5}Estimates for the inverse operator $\mathcal{V}^{-1}$}

In this section we prove Lemma \ref{L2}. We want to obtain an asymptotic
expansion for $\mathcal{V}^{-1}\left(  t\right)  ,$ as $t\rightarrow\infty,$
as well as a control of the $\mathbf{L}^{2}-$norm of the derivative
$\partial_{k}\mathcal{V}^{-1}\left(  t\right)  .$ These results are presented
in Lemmas \ref{Lemma 5.1} and \ref{Lemma 5.2} below. Recall that%
\[
\Lambda\left(  k\right)  =\theta\left(  k\right)  R_{+}\left(  k\right)
+\theta\left(  -k\right)  R_{-}\left(  -k\right)  .
\]

\begin{lemma}
\label{Lemma 5.1}Suppose that (\ref{3.4}) and (\ref{3.2}) with $\delta>0$ are
verified. Then, the estimate%
\[
\left\Vert \mathcal{V}^{-1}\left(  t\right)  \phi-\overline{T\left(
\left\vert k\right\vert \right)  }\phi\left(  k\right)  -\overline
{\Lambda\left(  -k\right)  }\phi\left(  -k\right)  \right\Vert _{\mathbf{L}%
^{\infty}}\leq C\left\vert \phi\left(  0\right)  \right\vert +Ct^{-\frac{1}%
{4}}\left\Vert \phi\right\Vert _{\mathbf{H}^{1}},
\]
is true for all $t\geq1.$
\end{lemma}

\begin{proof}
Recall that%
\begin{equation}
\mathcal{V}^{-1}\left(  t\right)  \phi=\sqrt{\frac{t}{2\pi i}}\int_{-\infty
}^{\infty}e^{\frac{it}{2}\left(  k-x\right)  ^{2}}\overline{\Phi\left(
tx,k\right)  }\phi\left(  x\right)  dx, \label{4.67}%
\end{equation}
where
\[
\overline{\Phi\left(  x,k\right)  }=\theta\left(  k\right)  \overline{T\left(
k\right)  }\overline{m_{+}\left(  x,k\right)  }+\theta\left(  -k\right)
\overline{T\left(  -k\right)  }\overline{m_{-}\left(  x,-k\right)  }.
\]
Using (\ref{4.68}) and (\ref{4.69}) in (\ref{4.67}), we have%
\begin{equation}
\left.
\begin{array}
[c]{c}%
\mathcal{V}^{-1}\left(  t\right)  \phi=\sqrt{\frac{t}{2\pi i}}\theta\left(
k\right)  \overline{T\left(  k\right)  }%
{\displaystyle\int_{0}^{\infty}}
e^{\frac{it}{2}\left(  k-x\right)  ^{2}}\overline{m_{+}\left(  tx,k\right)
}\phi\left(  x\right)  dx\\
+\sqrt{\frac{t}{2\pi i}}\theta\left(  -k\right)  \overline{R_{+}\left(
-k\right)  }%
{\displaystyle\int_{-\infty}^{0}}
e^{\frac{it}{2}\left(  k-x\right)  ^{2}}\overline{m_{+}\left(  -tx,-k\right)
}\phi\left(  -x\right)  dx\\
+\sqrt{\frac{t}{2\pi i}}\theta\left(  -k\right)
{\displaystyle\int_{0}^{\infty}}
e^{\frac{it}{2}\left(  k-x\right)  ^{2}}\overline{m_{+}\left(  tx,k\right)
}\phi\left(  x\right)  dx\\
+\sqrt{\frac{t}{2\pi i}}\theta\left(  k\right)  \overline{R_{-}\left(
k\right)  }%
{\displaystyle\int_{0}^{\infty}}
e^{\frac{it}{2}\left(  k-x\right)  ^{2}}\overline{m_{-}\left(  -tx,k\right)
}\phi\left(  -x\right)  dx\\
+\sqrt{\frac{t}{2\pi i}}\theta\left(  k\right)
{\displaystyle\int_{-\infty}^{0}}
e^{\frac{it}{2}\left(  k-x\right)  ^{2}}\overline{m_{-}\left(  tx,-k\right)
}\phi\left(  x\right)  dx\\
+\sqrt{\frac{t}{2\pi i}}\theta\left(  -k\right)  \overline{T\left(  -k\right)
}%
{\displaystyle\int_{-\infty}^{0}}
e^{\frac{it}{2}\left(  k-x\right)  ^{2}}\overline{m_{-}\left(  tx,-k\right)
}\phi\left(  x\right)  dx.
\end{array}
\right.  \label{4.73}%
\end{equation}
Using the notation
\[
\mathcal{V}_{0}\left(  -t\right)  \phi=\sqrt{\frac{t}{2\pi i}}\int_{-\infty
}^{\infty}e^{\frac{it}{2}\left(  k-x\right)  ^{2}}\phi\left(  x\right)  dx
\]
we write%
\begin{align*}
\mathcal{V}^{-1}\left(  t\right)  \phi &  =\theta\left(  k\right)
\overline{T\left(  k\right)  }\mathcal{V}_{0}\left(  -t\right)  \left(
\theta\left(  x\right)  \left(  \phi\left(  x\right)  -\phi\left(  0\right)
\right)  \right) \\
&  +\theta\left(  -k\right)  \overline{R_{+}\left(  -k\right)  }%
\mathcal{V}_{0}\left(  -t\right)  \left(  \theta\left(  -x\right)  \left(
\phi\left(  -x\right)  -\phi\left(  0\right)  \right)  \right) \\
&  +\theta\left(  -k\right)  \mathcal{V}_{0}\left(  -t\right)  \left(
\theta\left(  x\right)  \left(  \phi\left(  x\right)  -\phi\left(  0\right)
\right)  \right) \\
&  +\theta\left(  k\right)  \overline{R_{-}\left(  k\right)  }\mathcal{V}%
_{0}\left(  -t\right)  \left(  \theta\left(  x\right)  \left(  \phi\left(
-x\right)  -\phi\left(  0\right)  \right)  \right) \\
&  +\theta\left(  k\right)  \mathcal{V}_{0}\left(  -t\right)  \left(
\theta\left(  -x\right)  \left(  \phi\left(  x\right)  -\phi\left(  0\right)
\right)  \right) \\
&  +\theta\left(  -k\right)  \overline{T\left(  -k\right)  }\mathcal{V}%
_{0}\left(  -t\right)  \left(  \theta\left(  -x\right)  \left(  \phi\left(
x\right)  -\phi\left(  0\right)  \right)  \right)  +R_{1}+R_{2},
\end{align*}
where%
\begin{align*}
R_{1}  &  =\phi\left(  0\right)  \theta\left(  k\right)  \overline{T\left(
k\right)  }\mathcal{V}_{0}\left(  -t\right)  \theta\left(  x\right)
+\phi\left(  0\right)  \theta\left(  -k\right)  \overline{R_{+}\left(
-k\right)  }\mathcal{V}_{0}\left(  -t\right)  \theta\left(  -x\right) \\
&  +\phi\left(  0\right)  \theta\left(  -k\right)  \mathcal{V}_{0}\left(
-t\right)  \theta\left(  x\right)  +\phi\left(  0\right)  \theta\left(
k\right)  \overline{R_{-}\left(  k\right)  }\mathcal{V}_{0}\left(  -t\right)
\theta\left(  x\right) \\
&  +\phi\left(  0\right)  \theta\left(  k\right)  \mathcal{V}_{0}\left(
-t\right)  \theta\left(  -x\right)  +\phi\left(  0\right)  \theta\left(
-k\right)  \overline{T\left(  -k\right)  }\mathcal{V}_{0}\left(  -t\right)
\theta\left(  -x\right)
\end{align*}
and%
\begin{align*}
R_{2}  &  =\sqrt{\frac{t}{2\pi i}}\theta\left(  k\right)  \overline{T\left(
k\right)  }\int_{0}^{\infty}e^{\frac{it}{2}\left(  k-x\right)  ^{2}}\left(
\overline{m_{+}\left(  tx,k\right)  }-1\right)  \phi\left(  x\right)  dx\\
&  +\sqrt{\frac{t}{2\pi i}}\theta\left(  -k\right)  \overline{R_{+}\left(
-k\right)  }\int_{-\infty}^{0}e^{\frac{it}{2}\left(  k-x\right)  ^{2}}\left(
\overline{m_{+}\left(  -tx,-k\right)  }-1\right)  \phi\left(  -x\right)  dx\\
&  +\sqrt{\frac{t}{2\pi i}}\theta\left(  -k\right)  \int_{0}^{\infty}%
e^{\frac{it}{2}\left(  k-x\right)  ^{2}}\left(  \overline{m_{+}\left(
tx,k\right)  }-1\right)  \phi\left(  x\right)  dx\\
&  +\sqrt{\frac{t}{2\pi i}}\theta\left(  k\right)  \overline{R_{-}\left(
k\right)  }\int_{0}^{\infty}e^{\frac{it}{2}\left(  k-x\right)  ^{2}}\left(
\overline{m_{-}\left(  -tx,k\right)  }-1\right)  \phi\left(  -x\right)  dx\\
&  +\sqrt{\frac{t}{2\pi i}}\theta\left(  k\right)  \int_{-\infty}^{0}%
e^{\frac{it}{2}\left(  k-x\right)  ^{2}}\left(  \overline{m_{-}\left(
tx,-k\right)  }-1\right)  \phi\left(  x\right)  dx\\
&  +\sqrt{\frac{t}{2\pi i}}\theta\left(  -k\right)  \overline{T\left(
-k\right)  }\int_{-\infty}^{0}e^{\frac{it}{2}\left(  k-x\right)  ^{2}}\left(
\overline{m_{-}\left(  tx,-k\right)  }-1\right)  \phi\left(  x\right)  dx.
\end{align*}
Using (\ref{2.3}), (\ref{3.4}) and $\theta\left(  k\right)  \theta\left(
-k\right)  =0$ for $k\neq0,$ we find%
\begin{equation}
\left\vert \mathcal{V}^{-1}\left(  t\right)  \phi-\overline{T\left(
\left\vert k\right\vert \right)  }\phi\left(  k\right)  -\overline
{\Lambda\left(  -k\right)  }\phi\left(  -k\right)  \right\vert \leq
C\left\vert \phi\left(  0\right)  \right\vert +Ct^{-\frac{1}{4}}\left\Vert
\phi\right\Vert _{\mathbf{H}^{1}}+R_{1}+R_{2}, \label{4.70}%
\end{equation}
It follows from (\ref{3.4}) that%
\begin{equation}
\left\Vert R_{1}\right\Vert _{\mathbf{L}^{\infty}}\leq C\left\vert \phi\left(
0\right)  \right\vert . \label{4.71}%
\end{equation}
In order to estimate $R_{2}$ we use (\ref{3.4}) and (\ref{3.2}) with
$\delta>0$ to obtain%
\begin{equation}
\left\Vert R_{2}\right\Vert _{\mathbf{L}^{\infty}}\leq Ct^{\frac{1}{2}%
}\left\Vert \phi\right\Vert _{\mathbf{L}^{\infty}}\int_{-\infty}^{\infty
}\left\langle tx\right\rangle ^{-1-\delta}dx\leq Ct^{-\frac{1}{2}}\left\Vert
\phi\right\Vert _{\mathbf{L}^{\infty}}. \label{4.72}%
\end{equation}
Therefore the result of the lemma follows from (\ref{4.70}), (\ref{4.71}) and
(\ref{4.72}).
\end{proof}

In the next lemma we estimate the derivative $\partial_{k}\mathcal{V}%
^{-1}\left(  t\right)  $.

\begin{lemma}
\label{Lemma 5.2}Suppose that (\ref{2.44}), (\ref{3.2}) with $\delta=0,$
(\ref{3.5}) with $\delta>\frac{1}{2},$ and (\ref{3.6}) with $\delta=0$ are
true. Also, assume that $m_{+}\left(  x,0\right)  =m_{-}\left(  -x,0\right)  $
and (\ref{3.19}), (\ref{3.8}), (\ref{3.3}), (\ref{3.12}), (\ref{3.17}). If
$a=1,$ let $\phi\in\mathbf{H}^{1}$ be odd and if $a=-1,$ suppose that $\phi$
is even. Then the estimate
\begin{equation}
\left\Vert \partial_{k}\mathcal{V}^{-1}\left(  t\right)  \phi\right\Vert
_{\mathbf{L}^{2}}\leq Ct^{\frac{1}{2}}\left\vert \phi\left(  0\right)
\right\vert +C\left\Vert \phi\right\Vert _{\mathbf{H}^{1}} \label{5.13}%
\end{equation}
is valid for all $t\geq1$.
\end{lemma}

\begin{proof}
Taking into account (\ref{4.73}), we split $\mathcal{V}^{-1}\left(  t\right)
$ in two parts:
\begin{equation}
\mathcal{V}^{-1}\left(  t\right)  \phi=\theta\left(  k\right)  I_{+}\left(
k\right)  +\theta\left(  -k\right)  I_{-}\left(  k\right)  , \label{5.1}%
\end{equation}
where
\begin{align*}
I_{+}\left(  k\right)   &  =\sqrt{\frac{t}{2\pi i}}\int_{0}^{\infty}%
e^{\frac{it}{2}\left(  k-x\right)  ^{2}}\overline{T\left(  k\right)
}\overline{m_{+}\left(  tx,k\right)  }\phi\left(  x\right)  dx\\
&  +\sqrt{\frac{t}{2\pi i}}\int_{0}^{\infty}e^{\frac{it}{2}\left(  k-x\right)
^{2}}\overline{R_{-}\left(  k\right)  }\overline{m_{-}\left(  -tx,k\right)
}\phi\left(  -x\right)  dx\\
&  +\sqrt{\frac{t}{2\pi i}}\int_{-\infty}^{0}e^{\frac{it}{2}\left(
k-x\right)  ^{2}}\overline{m_{-}\left(  tx,-k\right)  }\phi\left(  x\right)
dx
\end{align*}
and
\begin{align*}
&  I_{-}\left(  k\right)  =\sqrt{\frac{t}{2\pi i}}\int_{-\infty}^{0}%
e^{\frac{it}{2}\left(  k-x\right)  ^{2}}\overline{T\left(  -k\right)
}\overline{m_{-}\left(  tx,-k\right)  }\phi\left(  x\right)  dx\\
&  +\sqrt{\frac{t}{2\pi i}}\int_{-\infty}^{0}e^{\frac{it}{2}\left(
k-x\right)  ^{2}}\overline{R_{+}\left(  -k\right)  }\overline{m_{+}\left(
-tx,-k\right)  }\phi\left(  -x\right)  dx\\
&  +\sqrt{\frac{t}{2\pi i}}\int_{0}^{\infty}e^{\frac{it}{2}\left(  k-x\right)
^{2}}\overline{m_{+}\left(  tx,k\right)  }\phi\left(  x\right)  dx.
\end{align*}
Note that if $a=\pm1,$ by (\ref{3.3}) $T\left(  0\right)  =\pm1$ and by
(\ref{3.12}) $R_{\pm}\left(  0\right)  =0.$ Also, by assumption $m_{+}\left(
x,0\right)  =m_{-}\left(  -x,0\right)  .$ Then, using that $\phi\ $is odd when
$a=1$ and $\phi$ is even if $a=-1,$ we get%
\begin{equation}
I_{+}\left(  0\right)  =\sqrt{\frac{t}{2\pi i}}\int_{0}^{\infty}e^{\frac
{it}{2}x^{2}}\overline{m_{+}\left(  tx,0\right)  }\left(  a\phi\left(
x\right)  +\phi\left(  -x\right)  \right)  dx=0 \label{5.25}%
\end{equation}
and
\begin{equation}
I_{-}\left(  0\right)  =\sqrt{\frac{t}{2\pi i}}\int_{0}^{\infty}e^{\frac
{it}{2}x^{2}}\overline{m_{+}\left(  tx,0\right)  }\left(  a\phi\left(
-x\right)  +\phi\left(  x\right)  \right)  dx=0. \label{5.26}%
\end{equation}
Thus, derivating (\ref{5.1}) we get%
\[
\partial_{k}\mathcal{V}^{-1}\left(  t\right)  \phi=\theta\left(  k\right)
\partial_{k}I_{+}\left(  k\right)  +\theta\left(  -k\right)  \partial_{k}%
I_{-}\left(  k\right)  ,
\]
and hence,%
\begin{equation}
\left\Vert \partial_{k}\mathcal{V}^{-1}\left(  t\right)  \phi\right\Vert
_{L^{2}}\leq\left\Vert \partial_{k}I_{+}\left(  k\right)  \right\Vert
_{L^{2}\left(  \mathbb{R}^{+}\right)  }+\left\Vert \partial_{k}I_{+}\left(
k\right)  \right\Vert _{L^{2}\left(  \mathbb{R}^{-}\right)  }. \label{5.23}%
\end{equation}
Making the change $z=k-x\ $\ we get%
\begin{align*}
I_{+}\left(  k\right)   &  =\sqrt{\frac{t}{2\pi i}}\int_{-\infty}^{k}%
e^{\frac{it}{2}z^{2}}\overline{T\left(  k\right)  }\overline{m_{+}\left(
t\left(  k-z\right)  ,k\right)  }\phi\left(  k-z\right)  dz\\
&  +\sqrt{\frac{t}{2\pi i}}\int_{-\infty}^{k}e^{\frac{it}{2}z^{2}}%
\overline{R_{-}\left(  k\right)  }\overline{m_{-}\left(  -t\left(  k-z\right)
,k\right)  }\phi\left(  z-k\right)  dz\\
&  +\sqrt{\frac{t}{2\pi i}}\int_{k}^{\infty}e^{\frac{it}{2}z^{2}}%
\overline{m_{-}\left(  t\left(  k-z\right)  ,-k\right)  }\phi\left(
k-z\right)  dz
\end{align*}
and%
\begin{align*}
&  I_{-}\left(  k\right)  =\sqrt{\frac{t}{2\pi i}}\int_{k}^{\infty}%
e^{\frac{it}{2}z^{2}}\overline{T\left(  -k\right)  }\overline{m_{-}\left(
t\left(  k-z\right)  ,-k\right)  }\phi\left(  k-z\right)  dz\\
&  +\sqrt{\frac{t}{2\pi i}}\int_{k}^{\infty}e^{\frac{it}{2}z^{2}}%
\overline{R_{+}\left(  -k\right)  }\overline{m_{+}\left(  -t\left(
k-z\right)  ,-k\right)  }\phi\left(  z-k\right)  dz\\
&  +\sqrt{\frac{t}{2\pi i}}\int_{-\infty}^{k}e^{\frac{it}{2}z^{2}}%
\overline{m_{+}\left(  t\left(  k-z\right)  ,k\right)  }\phi\left(
k-z\right)  dz.
\end{align*}
For $\partial_{k}I_{+}\left(  k\right)  $ we have%
\begin{align*}
\partial_{k}I_{+}\left(  k\right)   &  =\sqrt{\frac{t}{2\pi i}}\int_{-\infty
}^{k}e^{\frac{it}{2}z^{2}}\partial_{k}\left(  \overline{T\left(  k\right)
}\overline{m_{+}\left(  t\left(  k-z\right)  ,k\right)  }\phi\left(
k-z\right)  \right)  dz\\
&  +\sqrt{\frac{t}{2\pi i}}\int_{-\infty}^{k}e^{\frac{it}{2}z^{2}}\partial
_{k}\left(  \overline{R_{-}\left(  k\right)  }\overline{m_{-}\left(  -t\left(
k-z\right)  ,k\right)  }\phi\left(  z-k\right)  \right)  dz\\
&  +\sqrt{\frac{t}{2\pi i}}\int_{k}^{\infty}e^{\frac{it}{2}z^{2}}\partial
_{k}\left(  \overline{m_{-}\left(  t\left(  k-z\right)  ,-k\right)  }%
\phi\left(  k-z\right)  \right)  dz\\
&  +\sqrt{\frac{t}{2\pi i}}e^{\frac{it}{2}k^{2}}\phi\left(  0\right)  \left(
\overline{T\left(  k\right)  }\overline{m_{+}\left(  0,k\right)  }%
+\overline{R_{-}\left(  k\right)  }\overline{m_{-}\left(  0,k\right)
}-\overline{m_{-}\left(  0,-k\right)  }\right)  .
\end{align*}
Note that by (\ref{3.1})%
\[
\overline{T\left(  k\right)  }\overline{m_{+}\left(  0,k\right)  }%
-\overline{m_{-}\left(  0,-k\right)  }=\overline{R_{-}\left(  k\right)
}\overline{m_{-}\left(  0,k\right)  }.
\]
Then, returning to the old variable of integration $x=k-z$ we get%
\begin{equation}
\left.
\begin{array}
[c]{c}%
\partial_{k}I_{+}\left(  k\right)  =\sqrt{\frac{t}{2\pi i}}%
{\displaystyle\int\limits_{-\infty}^{\infty}}
e^{\frac{it}{2}\left(  k-x\right)  ^{2}}\Theta_{1}^{\left(  +\right)  }\left(
t,x,k\right)  dx+\sqrt{\frac{t}{2\pi i}}%
{\displaystyle\int\limits_{-\infty}^{\infty}}
e^{\frac{it}{2}\left(  k-x\right)  ^{2}}\Theta_{2}^{\left(  +\right)  }\left(
t,x,k\right)  dx\\
+\sqrt{\frac{t}{2\pi i}}%
{\displaystyle\int\limits_{-\infty}^{\infty}}
e^{\frac{it}{2}\left(  k-x\right)  ^{2}}\Theta_{3}^{\left(  +\right)  }\left(
t,x,k\right)  dx+2\sqrt{\frac{t}{2\pi i}}e^{\frac{it}{2}k^{2}}\phi\left(
0\right)  \overline{R_{-}\left(  k\right)  }\overline{m_{-}\left(  0,k\right)
},
\end{array}
\right.  \label{5.2}%
\end{equation}
with%
\[
\Theta_{1}^{\left(  +\right)  }\left(  t,x,k\right)  =\theta\left(  x\right)
\partial_{k}\left(  \overline{T\left(  k\right)  }\overline{m_{+}\left(
tx,k\right)  }\phi\left(  x\right)  +\overline{R_{-}\left(  k\right)
}\overline{m_{-}\left(  -tx,k\right)  }\phi\left(  -x\right)  \right)
+\theta\left(  -x\right)  \partial_{k}\overline{m_{-}\left(  tx,-k\right)
}\phi\left(  x\right)
\]%
\[
\Theta_{2}^{\left(  +\right)  }\left(  t,x,k\right)  =\theta\left(  x\right)
\left(  \overline{T\left(  k\right)  }\overline{m_{+}\left(  tx,k\right)
}\partial_{x}\phi\left(  x\right)  -\overline{R_{-}\left(  k\right)
}\overline{m_{-}\left(  -tx,k\right)  }\left(  \partial_{x}\phi\right)
\left(  -x\right)  \right)  +\theta\left(  -x\right)  \overline{m_{-}\left(
tx,-k\right)  }\partial_{x}\phi\left(  x\right)
\]%
\[
\Theta_{3}^{\left(  +\right)  }\left(  t,x,k\right)  =\theta\left(  x\right)
\left(  \overline{T\left(  k\right)  }\partial_{x}\overline{m_{+}\left(
tx,k\right)  }\phi\left(  x\right)  +\overline{R_{-}\left(  k\right)
}\partial_{x}\overline{m_{-}\left(  -tx,k\right)  }\phi\left(  -x\right)
\right)  +\theta\left(  -x\right)  \partial_{x}\overline{m_{-}\left(
tx,-k\right)  }\phi\left(  x\right)
\]
We estimate the last term in (\ref{5.2}) by using (\ref{2.44}) and
(\ref{3.19})
\begin{equation}
\left.
\begin{array}
[c]{c}%
\left\Vert \sqrt{\frac{t}{2\pi i}}e^{\frac{it}{2}k^{2}}\phi\left(  0\right)
\overline{R_{-}\left(  k\right)  }\overline{m_{-}\left(  0,k\right)
}\right\Vert _{\mathbf{L}^{2}}\\
\leq Ct^{\frac{1}{2}}\left\vert \phi\left(  0\right)  \right\vert \left\Vert
\overline{R_{-}\left(  k\right)  }\overline{m_{-}\left(  0,k\right)
}\right\Vert _{\mathbf{L}^{2}}\leq Ct^{\frac{1}{2}}\left\vert \phi\left(
0\right)  \right\vert \left\Vert \left\langle k\right\rangle ^{-1}\right\Vert
_{\mathbf{L}^{2}}\leq Ct^{\frac{1}{2}}\left\vert \phi\left(  0\right)
\right\vert .
\end{array}
\right.  \label{5.12}%
\end{equation}
To estimate the term in (\ref{5.2}) containing $\Theta_{1}^{\left(  +\right)
}\left(  t,x,k\right)  $ we decompose%
\begin{equation}
\Theta_{1}^{\left(  +\right)  }\left(  t,x,k\right)  =\Theta_{11}^{\left(
+\right)  }\left(  x,k\right)  +\Theta_{12}^{\left(  +\right)  }\left(
t,x,k\right)  , \label{5.3}%
\end{equation}
where%
\[
\Theta_{11}^{\left(  +\right)  }\left(  x,k\right)  =\theta\left(  x\right)
\phi\left(  x\right)  \partial_{k}\overline{T\left(  k\right)  }+\theta\left(
x\right)  \phi\left(  -x\right)  \partial_{k}\overline{R_{-}\left(  k\right)
}%
\]
and%
\begin{align*}
\Theta_{12}^{\left(  +\right)  }\left(  t,x,k\right)   &  =\left(
\overline{m_{+}\left(  tx,k\right)  }-1\right)  \theta\left(  x\right)
\phi\left(  x\right)  \partial_{k}\overline{T\left(  k\right)  }+\left(
\overline{m_{-}\left(  -tx,k\right)  }-1\right)  \theta\left(  x\right)
\phi\left(  -x\right)  \partial_{k}\overline{R_{-}\left(  k\right)  }\\
&  +\overline{T\left(  k\right)  }\theta\left(  x\right)  \phi\left(
x\right)  \partial_{k}\overline{m_{+}\left(  tx,k\right)  }+\overline
{R_{-}\left(  k\right)  }\theta\left(  x\right)  \phi\left(  -x\right)
\partial_{k}\overline{m_{-}\left(  -tx,k\right)  }\\
&  +\theta\left(  -x\right)  \phi\left(  x\right)  \partial_{k}\overline
{m_{-}\left(  tx,-k\right)  }.
\end{align*}
Observe that
\[
\sqrt{\frac{t}{2\pi i}}\int_{-\infty}^{\infty}e^{\frac{it}{2}\left(
k-x\right)  ^{2}}\Theta_{11}^{\left(  +\right)  }\left(  x,k\right)
dx=\partial_{k}\overline{T\left(  k\right)  }\mathcal{V}_{0}\left(  -t\right)
\left(  \theta\left(  x\right)  \phi\left(  x\right)  \right)  +\partial
_{k}\overline{R_{-}\left(  k\right)  }\mathcal{V}_{0}\left(  -t\right)
\left(  \theta\left(  x\right)  \phi\left(  -x\right)  \right)  .
\]
Then, it follows from (\ref{2.11}), (\ref{3.8}) and (\ref{3.17}) that
\begin{equation}
\left.
\begin{array}
[c]{c}%
\left\Vert \sqrt{\frac{t}{2\pi i}}%
{\displaystyle\int\limits_{-\infty}^{\infty}}
e^{\frac{it}{2}\left(  k-x\right)  ^{2}}\Theta_{11}^{\left(  +\right)
}\left(  x,k\right)  dx\right\Vert _{\mathbf{L}^{2}\left(  \mathbb{R}%
^{+}\right)  }\\
\leq\left\Vert \left(  \partial_{k}\overline{T\left(  k\right)  }\right)
\mathcal{V}_{0}\left(  -t\right)  \left(  \theta\left(  x\right)  \phi\left(
x\right)  \right)  \right\Vert _{\mathbf{L}^{2}\left(  \mathbb{R}^{+}\right)
}+\left\Vert \left(  \partial_{k}\overline{R_{-}\left(  k\right)  }\right)
\mathcal{V}_{0}\left(  -t\right)  \left(  \theta\left(  x\right)  \phi\left(
-x\right)  \right)  \right\Vert _{\mathbf{L}^{2}\left(  \mathbb{R}^{+}\right)
}\leq C\left\Vert \phi\right\Vert _{\mathbf{L}^{2}}%
\end{array}
\right.  \label{5.4}%
\end{equation}
Moreover, using (\ref{3.4}), (\ref{3.2}) with $\delta=0,$ (\ref{3.5}) with
$\delta>\frac{1}{2},$ (\ref{3.8}) and (\ref{3.17}) we estimate%
\begin{align*}
\left\vert \Theta_{12}^{\left(  +\right)  }\left(  t,x,k\right)  \right\vert
&  \leq C\theta\left(  x\right)  \left(  \left\vert \overline{m_{+}\left(
tx,k\right)  }-1\right\vert +\left\vert \overline{m_{-}\left(  -tx,k\right)
}-1\right\vert \right)  \left(  \left\vert \phi\left(  x\right)  \right\vert
+\left\vert \phi\left(  -x\right)  \right\vert \right)  \left(  \left\Vert
\partial_{k}T\right\Vert _{\mathbf{L}^{\infty}}+\left\Vert \partial_{k}%
R_{-}\right\Vert _{\mathbf{L}^{\infty}}\right) \\
&  +C\theta\left(  x\right)  \left(  \left\vert \phi\left(  x\right)
\right\vert +\left\vert \phi\left(  -x\right)  \right\vert \right)  \left(
\left\Vert T\right\Vert _{\mathbf{L}^{\infty}}+\left\Vert R_{-}\right\Vert
_{\mathbf{L}^{\infty}}\right)  \left(  \left\vert \partial_{k}\overline
{m_{+}\left(  tx,k\right)  }\right\vert +\left\vert \partial_{k}%
\overline{m_{-}\left(  -tx,k\right)  }\right\vert \right) \\
&  +C\theta\left(  -x\right)  \left\vert \phi\left(  x\right)  \right\vert
\left\vert \partial_{k}\overline{m_{-}\left(  tx,-k\right)  }\right\vert \leq
C\left\langle k\right\rangle ^{-1}\left\langle tx\right\rangle ^{-\delta
}\left(  \left\vert \phi\left(  x\right)  \right\vert +\left\vert \phi\left(
-x\right)  \right\vert \right)  .
\end{align*}
Then,
\begin{equation}
\left.
\begin{array}
[c]{c}%
\left\Vert \sqrt{\frac{t}{2\pi i}}%
{\displaystyle\int\limits_{-\infty}^{\infty}}
e^{\frac{it}{2}\left(  k-x\right)  ^{2}}\Theta_{12}^{\left(  +\right)
}\left(  t,x,k\right)  dx\right\Vert _{\mathbf{L}^{2}\left(  \mathbb{R}%
^{+}\right)  }\leq Ct^{\frac{1}{2}}\left\Vert \left\langle k\right\rangle
^{-1}%
{\displaystyle\int\limits_{-\infty}^{\infty}}
\left\langle tx\right\rangle ^{-\delta}\left\vert \phi\left(  x\right)
\right\vert dx\right\Vert _{\mathbf{L}^{2}\left(  \mathbb{R}^{+}\right)  }\\
\leq Ct^{\frac{1}{2}}\left\Vert \left\langle tx\right\rangle ^{-\delta
}\right\Vert _{\mathbf{L}^{2}}\left\Vert \left\langle k\right\rangle
^{-1}\right\Vert _{\mathbf{L}^{2}}\left\Vert \phi\right\Vert _{\mathbf{L}^{2}%
}\leq C\left\Vert \phi\right\Vert _{\mathbf{L}^{2}}.
\end{array}
\right.  \label{5.5}%
\end{equation}
Combining (\ref{5.3}), (\ref{5.4}) and (\ref{5.5}) we get%
\begin{equation}
\left\Vert \sqrt{\frac{t}{2\pi i}}\int_{-\infty}^{\infty}e^{\frac{it}%
{2}\left(  k-x\right)  ^{2}}\Theta_{1}^{\left(  +\right)  }\left(
t,x,k\right)  dx\right\Vert _{\mathbf{L}^{2}\left(  \mathbb{R}^{+}\right)
}\leq C\left\Vert \phi\right\Vert _{\mathbf{L}^{2}}. \label{5.6}%
\end{equation}
Next we consider the term in (\ref{5.2}) containing $\Theta_{2}^{\left(
+\right)  }\left(  t,x,k\right)  .$ We split
\begin{equation}
\Theta_{2}^{\left(  +\right)  }\left(  t,x,k\right)  =\Theta_{21}^{\left(
+\right)  }\left(  x,k\right)  +\Theta_{22}^{\left(  +\right)  }\left(
t,x,k\right)  , \label{5.7}%
\end{equation}
where%
\[
\Theta_{21}^{\left(  +\right)  }\left(  x,k\right)  =\theta\left(  x\right)
\overline{T\left(  k\right)  }\partial_{x}\phi\left(  x\right)  -\theta\left(
x\right)  \overline{R_{-}\left(  k\right)  }\left(  \partial_{x}\phi\right)
\left(  -x\right)  +\theta\left(  -x\right)  \partial_{x}\phi\left(  x\right)
\]
and%
\begin{align*}
\Theta_{22}^{\left(  +\right)  }\left(  t,x,k\right)   &  =\theta\left(
x\right)  \overline{T\left(  k\right)  }\left(  \overline{m_{+}\left(
tx,k\right)  }-1\right)  \partial_{x}\phi\left(  x\right) \\
&  -\theta\left(  x\right)  \overline{R_{-}\left(  k\right)  }\left(
\overline{m_{-}\left(  -tx,k\right)  }-1\right)  \left(  \partial_{x}%
\phi\right)  \left(  -x\right) \\
&  +\theta\left(  -x\right)  \left(  \overline{m_{-}\left(  tx,-k\right)
}-1\right)  \partial_{x}\phi\left(  x\right)
\end{align*}
By using (\ref{2.11}) and (\ref{3.4}) we estimate%
\begin{equation}
\left.
\begin{array}
[c]{c}%
\left\Vert \sqrt{\frac{t}{2\pi i}}%
{\displaystyle\int\limits_{-\infty}^{\infty}}
e^{\frac{it}{2}\left(  k-x\right)  ^{2}}\Theta_{21}^{\left(  +\right)
}\left(  x,k\right)  dx\right\Vert _{\mathbf{L}^{2}\left(  \mathbb{R}%
^{+}\right)  }\leq\left\Vert \overline{T\left(  k\right)  }\mathcal{V}%
_{0}\left(  -t\right)  \left(  \theta\left(  x\right)  \partial_{x}\phi\left(
x\right)  \right)  \right\Vert _{\mathbf{L}^{2}\left(  \mathbb{R}^{+}\right)
}\\
+\left\Vert \overline{R_{-}\left(  k\right)  }\mathcal{V}_{0}\left(
-t\right)  \left(  \theta\left(  x\right)  \left(  \partial_{x}\phi\right)
\left(  -x\right)  \right)  \right\Vert _{\mathbf{L}^{2}\left(  \mathbb{R}%
^{+}\right)  }+\left\Vert \mathcal{V}_{0}\left(  -t\right)  \left(
\theta\left(  -x\right)  \partial_{x}\phi\left(  x\right)  \right)
\right\Vert _{\mathbf{L}^{2}\left(  \mathbb{R}^{+}\right)  }\leq C\left\Vert
\partial_{x}\phi\right\Vert _{\mathbf{L}^{2}},
\end{array}
\right.  \label{5.8}%
\end{equation}
Moreover, from relation (\ref{3.4}) and (\ref{3.2}) with $\delta=0$ we get
\[
\left.
\begin{array}
[c]{c}%
\left\vert \Theta_{22}^{\left(  +\right)  }\left(  t,x,k\right)  \right\vert
\leq\theta\left(  x\right)  \left\vert \overline{T\left(  k\right)
}\right\vert \left\vert \overline{m_{+}\left(  tx,k\right)  }-1\right\vert
\left\vert \partial_{x}\phi\left(  x\right)  \right\vert +\theta\left(
x\right)  \left\vert \overline{R_{-}\left(  k\right)  }\right\vert \left\vert
\overline{m_{-}\left(  -tx,k\right)  }-1\right\vert \left\vert \left(
\partial_{x}\phi\right)  \left(  -x\right)  \right\vert \\
+\theta\left(  -x\right)  \left\vert \overline{m_{-}\left(  tx,-k\right)
}-1\right\vert \left\vert \partial_{x}\phi\left(  x\right)  \right\vert \leq
C\left\langle k\right\rangle ^{-1}\left\langle tx\right\rangle ^{-1}\left(
\left\vert \partial_{x}\phi\left(  x\right)  \right\vert +\left\vert
\partial_{x}\phi\left(  -x\right)  \right\vert \right)  ,
\end{array}
\right.
\]
and hence,%
\begin{equation}
\left.
\begin{array}
[c]{c}%
\left\Vert \sqrt{\frac{t}{2\pi i}}%
{\displaystyle\int\limits_{-\infty}^{\infty}}
e^{\frac{it}{2}\left(  k-x\right)  ^{2}}\Theta_{22}^{\left(  +\right)
}\left(  t,x,k\right)  dx\right\Vert _{\mathbf{L}^{2}\left(  \mathbb{R}%
^{+}\right)  }\leq Ct^{\frac{1}{2}}\left\Vert \left\langle k\right\rangle
^{-1}%
{\displaystyle\int\limits_{-\infty}^{\infty}}
\left\langle tx\right\rangle ^{-1}\left\vert \partial_{x}\phi\left(  x\right)
\right\vert dx\right\Vert _{\mathbf{L}^{2}\left(  \mathbb{R}^{+}\right)  }\\
\leq Ct^{\frac{1}{2}}\left\Vert \left\langle k\right\rangle ^{-1}\right\Vert
_{\mathbf{L}^{2}}\left\Vert \left\langle tx\right\rangle ^{-1}\right\Vert
_{\mathbf{L}^{2}}\left\Vert \partial_{x}\phi\right\Vert _{\mathbf{L}^{2}}\leq
C\left\Vert \partial_{x}\phi\right\Vert _{\mathbf{L}^{2}}.
\end{array}
\right.  \label{5.9}%
\end{equation}
Using (\ref{5.7}), (\ref{5.8}) and (\ref{5.9}) we obtain%
\begin{equation}
\left\Vert \sqrt{\frac{t}{2\pi i}}%
{\displaystyle\int\limits_{-\infty}^{\infty}}
e^{\frac{it}{2}\left(  k-x\right)  ^{2}}\Theta_{2}^{\left(  +\right)  }\left(
t,x,k\right)  dx\right\Vert _{\mathbf{L}^{2}\left(  \mathbb{R}^{+}\right)
}\leq C\left\Vert \partial_{x}\phi\right\Vert _{\mathbf{L}^{2}}. \label{5.10}%
\end{equation}
Finally\ by (\ref{3.4}) and (\ref{3.6}) with $\delta=0$ we get%
\[
\left.
\begin{array}
[c]{c}%
\left\vert \Theta_{3}^{\left(  +\right)  }\left(  t,x,k\right)  \right\vert
\leq C\theta\left(  x\right)  \left\vert \overline{T\left(  k\right)
}\right\vert \left\vert \partial_{x}\overline{m_{+}\left(  tx,k\right)
}\right\vert \left\vert \phi\left(  x\right)  \right\vert +C\theta\left(
x\right)  \left\vert \overline{R_{-}\left(  k\right)  }\right\vert \left\vert
\partial_{x}\overline{m_{-}\left(  -tx,k\right)  }\right\vert \left\vert
\phi\left(  -x\right)  \right\vert \\
+C\theta\left(  -x\right)  \left\vert \partial_{x}\overline{m_{-}\left(
tx,-k\right)  }\right\vert \left\vert \phi\left(  x\right)  \right\vert \leq
Ct\left\langle k\right\rangle ^{-1}\left\langle tx\right\rangle ^{-2}\left(
\left\vert \phi\left(  x\right)  \right\vert +\left\vert \phi\left(
-x\right)  \right\vert \right)  .
\end{array}
\right.
\]
Therefore, using that $\left\vert \phi\left(  x\right)  -\phi\left(  0\right)
\right\vert \leq C\left\vert x\right\vert ^{\frac{1}{2}}\left\Vert
\partial_{k}\phi\right\Vert _{\mathbf{L}^{2}}$, we conclude%
\begin{equation}
\left.
\begin{array}
[c]{c}%
\left\Vert \sqrt{\frac{t}{2\pi i}}%
{\displaystyle\int\limits_{-\infty}^{\infty}}
e^{\frac{it}{2}\left(  k-x\right)  ^{2}}\Theta_{3}^{\left(  +\right)  }\left(
t,x,k\right)  dx\right\Vert _{\mathbf{L}^{2}\left(  \mathbb{R}^{+}\right)
}\leq Ct^{\frac{3}{2}}\left\Vert \left\langle k\right\rangle ^{-1}\right\Vert
_{\mathbf{L}^{2}\left(  \mathbb{R}^{+}\right)  }%
{\displaystyle\int\limits_{-\infty}^{\infty}}
\left\langle tx\right\rangle ^{-2}\left\vert \phi\left(  x\right)  \right\vert
dx\\
\leq Ct^{\frac{3}{2}}\left\Vert \partial_{x}\phi\right\Vert _{\mathbf{L}^{2}%
}\left\Vert \left\langle tx\right\rangle ^{-2}\left\vert x\right\vert
^{\frac{1}{2}}\right\Vert _{\mathbf{L}^{1}}+Ct^{\frac{3}{2}}\left\vert
\phi\left(  0\right)  \right\vert \left\Vert \left\langle tx\right\rangle
^{-2}\right\Vert _{\mathbf{L}^{1}}\leq Ct^{\frac{1}{2}}\left\vert \phi\left(
0\right)  \right\vert +C\left\Vert \partial_{x}\phi\right\Vert _{\mathbf{L}%
^{2}}.
\end{array}
\right.  \label{5.11}%
\end{equation}
Introducing (\ref{5.12}), (\ref{5.6}), (\ref{5.10}) and (\ref{5.11}) into
(\ref{5.2}) we obtain%
\[
\left\Vert \partial_{k}I_{+}\left(  k\right)  \right\Vert _{\mathbf{L}%
^{2}\left(  \mathbb{R}^{+}\right)  }\leq Ct^{\frac{1}{2}}\left\vert
\phi\left(  0\right)  \right\vert +C\left\Vert \phi\right\Vert _{\mathbf{H}%
^{1}}.
\]
Proceeding similarly to estimate $\partial_{k}I_{-},$ from (\ref{5.23}) we
attain (\ref{5.13}).
\end{proof}


\begin{thebibliography}{99}                                                                                               %


\bibitem {adams}R.A. {Adams and J.J.F. Fournier,} {\ \textit{Sobolev spaces.
Second edition.} Pure and Applied Mathematics (Amsterdam) \textbf{140}.
Elsevier/Academic Press, Amsterdam, (2003)}

\bibitem {aktosun}T. Aktosun, M. Klaus, Martin and C. van der Mee,
\textit{Factorization of scattering matrices due to partitioning of potentials
in one-dimensional Schr\"{o}dinger-type equations.} J. Math. Phys. \textbf{37}
12 (1996) 5897--5915.

\bibitem {Aktosun2}T. Aktosun, M. Klaus, Martin and C. van der Mee,
\textit{Wave scattering in one dimension with absorption.} J. Math. Phys.
\textbf{39} 4 (1998) 1957--1992.

\bibitem {aktosun1}T. Aktosun, M. Klaus, Martin and C. van der Mee, \textit{On
the number of bound states for the one-dimensional Schr\"{o}dinger equation.}
J. Math. Phys. \textbf{39} 9 (1998) 4249--4256.

\bibitem {Carles3}P. Antonelli, R. Carles and J. D. Silva, \textit{Scattering
for nonlinear Schr\"{o}dinger equation under partial harmonic confinement.}
Comm. Math. Phys. \textbf{334} 1 (2015) 367--396.

\bibitem {Cuccagna}S. Cuccagna, V. Georgiev, and N. Visciglia. \textit{Decay
and scattering of small solutions of pure power NLS in }$R$\textit{ with
}$p>3$\textit{ and with a potential.} Comm. Pure Appl. Math., \textbf{67 }6
(2014) 957--981.

\bibitem {Yajima}N. Asano, T. Taniuti, and N. Yajima, \textit{Perturbation
method for nonlinear wave modulation,} II, J. Math. Phys., \textbf{10}, (1969) 2020-2024.

\bibitem {Barab}J.E. Barab, \textit{Non-existence of asymptotically free
solutions for nonlinear Schr\"{o}dinger equation,} J.Math.Phys., \textbf{25}
11 (1984) 3270--3273.

\bibitem {Benney}D. J. Benney and A. C. Newell, \textit{The propagation of
nonlinear wave envelopes}, J. Math. Phys., \textbf{46} (1967) 133-139.

\bibitem {Berge}L. Berg\'{e}, \textit{Wave collapse in physics: principles and
applications to light and plasma waves, }Phys. Rep., \textbf{303 }(1998) 259-370.

\bibitem {Bespalov}V. I. Bespalov and V. I. Talanov, \textit{Filamentary
structure of light beams in nonlinear liquids,} JETP Lett., \textbf{3 }(1966) 307-310.

\bibitem {Bransden}B. Bransden, C. Joachain, \textit{Quantum Mechanics}, 2nd
ed, Pearson/Prentice Hall, 2000.

\bibitem {Carles}{Carles R.} \textit{Geometric optics and long range
scattering for one-dimensional nonlinear Schr\"{o}\-din\-ger equations.}{
Comm. Math. Phys. \textbf{220} 1 (2001) 41--67.}

\bibitem {CazenaveW1}{Cazenave T. and Weissler F. B.} \textit{Rapidly decaying
solutions of the nonlinear Schr\"{o}\-din\-ger equation,}{ Comm. Math. Phys.
\textbf{147} (1992) 75--100.}

\bibitem {Cazenave}T. Cazenave, \textit{Semilinear Schr\"{o}dinger equations},
Courant Institute of Mathematical Sciences, New York; American Mathematical
Society, Providence, RI, (2003) 323 pp.

\bibitem {CN}{Cazenave T. and Naumkin I.} \textit{Local existence, global
existence, and scattering for the nonlinear Schr\"{o}\-din\-ger equation.}{
Commun. Contemp. Math. \textbf{19} 2 (2017) 20 pp.}

\bibitem {CN1}{Cazenave T. and Naumkin I.} \ \textit{Modified scattering for
the critical nonlinear Schr\"{o}dinger equation. (2017) }arXiv:1702.08221

\bibitem {Combes}J.M. Combes, R. Schrader and R. Seiler, \textit{Classical
bounds and limits for energy distributions of Hamilton operators in
electromagnetic fields, }Ann. Phys., \textbf{111 }1 (1978) 1--18.

\bibitem {deGennes}P. G. deGennes, \textit{Superconductivity of Metals and
Alloys}, NewYork: Benjamin, (1966)

\bibitem {DeiftTrubowitz}P. Deift and E. Trubowitz, \textit{Inverse scattering
on the line,} Comm. Pure Appl. Math. \textbf{32} (1979) 121-251.

\bibitem {deift}P. Deift and J. Park, \textit{Long-time asymptotics for
solutions of the NLS equation with a delta potential and even initial data,}
Int. Math. Res. Not. IMRN \textbf{24} (2011) 5505--5624

\bibitem {Eboli}O. J. P. Eboli and G. C. Marques, \textit{Solitons as
Newtonian particles, }Phys. Rev. B, \textbf{28 }2 (1983) 689-696.

\bibitem {Floer}A. Floer and A. Weinstein, \textit{Nonspreading wave packets
for the cubic Schr\"{o}dinger equation with a bounded potential. }J. Funct.
Anal. \textbf{69} 3 (1986) 397--408.

\bibitem {GinibreV1}{Ginibre J. and Velo G.} \textit{On a class of nonlinear
Schr\"{o}\-din\-ger equations. II. Scattering theory, general case,}{ J.
Funct. Anal. \textbf{32}, no.~1 (1979) 33--71.}

\bibitem {GinibreV2}{Ginibre J. and Velo G.} \textit{On a class of nonlinear
Schr\"{o}\-din\-ger equations. III. Special theories in dimensions 1, 2 and
3,}{ Ann. Inst. Henri Poincar\'{e} \textbf{28} (1978) 287--316.}

\bibitem {GinibreOV}{Ginibre J., Ozawa T. and Velo G.} \textit{On the
existence of the wave operators for a class of nonlinear Schr\"{o}\-din\-ger
equations,}{ Ann. Inst. H.~Poin\-ca\-r\'{e} Phys. Th\'{e}or. \textbf{60} 2
(1994) 211--239.}

\bibitem {haka1998}N. Hayashi, E. I. Kaikina and P. I. Naumkin, \textit{On the
scattering theory for the cubic nonlinear Schr\"{o}dinger and Hartree type
equations in one space dimension,} Hokkaido Math. J. \textbf{27 }3 (1998) 651--667.

\bibitem {HayashiNau1}{Hayashi N. and Naumkin P.I.} \textit{Asymptotics for
large time of solutions to the nonlinear Schr\"{o}\-din\-ger and Hartree
equations.}{ Amer. J. Math. \textbf{120} 2 (1998) 369--389.}

\bibitem {HayashiNau2}{Hayashi N. and Naumkin P.I.} \textit{Large time
behavior for the cubic nonlinear Schr\"{o}\-din\-ger equation.}{ Canad. J.
Math. \textbf{54} 5 (2002) 1065--1085.}

\bibitem {HNST}{Hayashi N., Naumkin P.I. , Shimomura A. and Tonegawa S.}
\textit{Modified wave operators for nonlinear Schr\"{o}\-din\-ger equations in
one and two dimensions.}{ Electron. J. Differential Equations 62\ (2004) 16
pp.}

\bibitem {H-O}N. Hayashi and T. Ozawa, \textit{Scattering theory in the
weighted }$\mathbf{L}^{2}(${$\mathbb{R}$}$^{n})$\textit{\ spaces for some
Schr\"{o}dinger equations,} Ann. I.H.P. (Phys. Th\'{e}or.), \textbf{48\ (}%
1988) pp. 17-37.

\bibitem {Karpman}V. I. Karpman and E. M. Kruskal, \textit{Modulated waves in
a nonlinear dispersive media,} Sov. Phys.-JETP\textbf{\ 28} (1969) 277-281

\bibitem {Kato}{Kato T.} {Nonlinear Schr\"{o}\-din\-ger equations, in
\textit{Schr\"{o}\-din\-ger Operators (S\o nderborg, 1988)}, Lecture Notes in
Phys. Springer, Berlin, \textbf{345} (1989) 218--263.}

\bibitem {KitaW}{Kita N. and Wada T.} \textit{Sharp asymptotic behavior of
solutions to nonlinear Schr\"{o}\-din\-ger equations in one space dimension.}{
Funkcial. Ekvac. \textbf{45} 1 (2002) 53--69.}

\bibitem {Klaus}M. Klaus, \textit{Low-energy behaviour of the scattering
matrix for the Schr\"{o}dinger equation on the line.} Inverse Problems
\textbf{4} 2 (1988) 505--512.

\bibitem {NakanishiO}{Nakanishi K. and Ozawa T.} \textit{Remarks on scattering
for nonlinear Schr\"{o}\-din\-ger equations,}{ NoDEA Nonlinear Differential
Equations Appl. \textbf{9} 1 (2002) 45--68.}

\bibitem {Ivan}I. P. Naumkin, \textit{Sharp asymptotic behavior of solutions
for cubic nonlinear Schr\"{o}dinger equations with a potential.} J. Math.
Phys. \textbf{57} 5 (2016) 31 pp.

\bibitem {Ozawa1}{Ozawa T.} \textit{Long range scattering for the nonlinear
Schr\"{o}\-din\-ger equation in one space dimension,}{ Comm. Math. Phys.
\textbf{139} 3 (1991) 479--493.}

\bibitem {Scott}A. C. Scott, F. Y. F. Chu, and D. W. McClaughlin, \textit{The
soliton: A new concept} \textit{in applied science}, PYOC. IEEE \textbf{61}
(1973) 1143-1483.

\bibitem {Shimizu}K. Shimizu and Y. H. Ichikawa, \textit{Automodulation of ion
oscillation modes in plasma,} J. Phys. Soc. Japan \textbf{33} (1972) 189-792

\bibitem {ShimomuraT}{Shimomura A. and Tonegawa S.} {Long-range scattering for
nonlinear Schr\"{o}\-din\-ger equations in one and two space dimensions.
Differential Integral Equations \textbf{17} 1-2 (2004) 127--150.}

\bibitem {Strauss}{Strauss W.A.} {Nonlinear scattering theory, in
\textit{Scattering Theory in Mathematical Physics}, NATO Advanced Study
Institutes Series Volume \textbf{9} (1974) 53-78.}

\bibitem {Strauss2}{Strauss W.A.} {Nonlinear scattering theory at low energy,
J. Funct. Anal. \textbf{41} 1 (1981) 110--133.}

\bibitem {Sulem}C. Sulem and P. L. Sulem, \textit{The nonlinear
Schr\"{o}dinger equation. Self-focusing and wave collapse,} App. Math.
Sciences, Springer-Verlag, New York, \textbf{139} (1999) 350 pp.

\bibitem {Taniuti}T. Taniuti and H. Washimi, \textit{Self trapping and
instability of hydromagnetic waves along the magnetic field in a cold plasma},
Phys. Rev. Lett., \textbf{21} (1968) 209-212

\bibitem {Weder2000}R. Weder, $L^{p}-L^{p^{\prime}}$\textit{\ estimates for
the Schr\"{o}dinger equation on the line and inverse scattering for the
nonlinear Schr\"{o}dinger equation with a potential,} J. Funct. Anal.
\textbf{170} 1 (2000) 37--68.

\bibitem {Wolf}E. L. Wolf, Principles of Electron Tunneling Spectroscopy.
Oxford U.P., New York, (1985).
\end{thebibliography}
\end{document}